\documentclass[11pt]{amsart}
\usepackage[usenames, dvipsnames, table]{xcolor}
\usepackage{amssymb, tikz, tikz-cd, pgfplots, tabto, bm, bbm, url, mathdots} 
\usepackage{verbatim}
\usepackage{hyperref}
\usepackage{enumitem}
\usetikzlibrary{knots, fadings, patterns}
\tikzfading[name=fade out, inner color=transparent!0, outer color=transparent!60]

\usepackage{geometry}
\theoremstyle{definition}
\newtheorem{theorem}{Theorem}[section]
\newtheorem{setup}[theorem]{Setup}
\newtheorem{example}[theorem]{Example}
\newtheorem{examples}[theorem]{Examples}

\newtheorem{remark}[theorem]{Remark}
\newtheorem{definition}[theorem]{Definition}

\newtheorem{proposition}[theorem]{Proposition}
\newtheorem{corollary}[theorem]{Corollary}
\newtheorem{lemma}[theorem]{Lemma}

\numberwithin{equation}{section}

\allowdisplaybreaks

\newcommand{\la}{\langle}
\newcommand{\ra}{\rangle}

\newcommand{\Z}{\mathbb{Z}}
\newcommand{\R}{\mathbb{R}}
\newcommand{\C}{\mathbb{C}}
\newcommand{\Om}{\Omega}
\newcommand{\del}{\partial}
\newcommand{\Del}{\Delta}
\newcommand{\unDel}{\ul{\Delta}}
\newcommand{\ot}{\otimes}
\newcommand{\ul}{\underline}
\newcommand{\un}{\underline}
\newcommand{\bu}{\bullet}
\newcommand{\mc}{\mathcal }

\newcommand{\al}{\alpha}
\newcommand{\be}{\beta}
\newcommand{\ga}{\gamma}
\newcommand{\wt}{\widetilde}
\newcommand{\wh}{\widehat}
\newcommand{\unbr}{\underbrace}
\newcommand{\xbb}{\{x_{\bu\bu}\}}
\newcommand{\ybb}{\{y_{\bu\bu}\}}
\newcommand{\xbp}{\{x'_{\bu\bu}\}}
\newcommand{\ybp}{\{y'_{\bu\bu}\}}
\newcommand{\xbbt}{\{\wt x_{\bu\bu}\}}
\newcommand{\xbpt}{\{\wt x'_{\bu\bu}\}}
\newcommand{\begintik}{ \begin{tikzpicture} \draw [black, line width=2pt] (0,0)--(8,8)--(0,8)--(0,0);}
\newcommand{\btik}[2]{ \begin{tikzpicture} \draw [black, line width=2pt] (#1,#1)--(#2,#2)--(#1,#2)--(#1,#1);}
\newcommand{\dblue}[2]{\draw [blue, line width=3pt] (#1,#1)--(#1,#2)--(#2,#2); }
\newcommand{\dred}[2]{\draw [red, line width=3pt] (#1,#1)--(#1,#2)--(#2,#2); }
\newcommand{\dblack}[2]{\draw [white, line width=4pt] (#1,#1)--(#1,#2)--(#2,#2); \draw [black, line width=2pt] (#1,#1)--(#1,#2)--(#2,#2);}
\newcommand{\scb}[1]{\scalebox{0.7}{#1}}
\newcommand{\St}[2]{{\mathcal S^{#2}(#1)}}
\newcommand{\StIP}[1]{{\mathcal I(#1)}}
\newcommand{\StC}[1]{{\mathcal C(#1)}}

\newcommand{\ceco}{\cellcolor{gray!20}}
\newcommand{\fcirc}[2]{\filldraw[color=black, fill=white] #1 circle (.3); \draw #1 node {$#2$};}

\newcommand{\cA}{\bm{A}}
\newcommand{\cB}{\bm{B}}
\newcommand{\ca}{\bm{a}}
\newcommand{\cb}{\bm{b}}
\newcommand{\cc}{\bm{c}}
\newcommand{\cal}{\bm{\alpha}}
\newcommand{\cbe}{\bm{\beta}}
\newcommand{\cs}{\bm{s}}

\newcommand{\cE}{\bm{E}}
\newcommand{\cX}{\bm{X}}
\newcommand{\cY}{\bm{Y}}
\newcommand{\cZ}{\bm{Z}}

\newcommand{\uA}{\ul{\bm{A}}}
\newcommand{\uB}{\ul{\bm{B}}}
\newcommand{\ua}{\ul{\bm{a}}}
\newcommand{\ub}{\ul{\bm{b}}}
\newcommand{\uc}{\ul{\bm{c}}}
\newcommand{\ue}{\ul{\bm{e}}}
\newcommand{\ual}{\ul{\bm{\alpha}}}
\newcommand{\ube}{\ul{\bm{\beta}}}
\newcommand{\us}{\ul{\bm{s}}}
\newcommand{\wuy}{\wh{\,\,\ul{\bm{y}}\,\,}}

\newcommand{\uE}{\ul{\bm{E}}}
\newcommand{\uX}{\ul{\bm{X}}}
\newcommand{\uY}{\ul{\bm{Y}}}
\newcommand{\uZ}{\ul{\bm{Z}}}

\newcommand{\dH}{\mathcal{H}}
\newcommand{\dDel}{\Theta}
\newcommand{\undDel}{\ul{\Theta}}
\newcommand{\ur}{\ul{r}}

\newcommand{\dE}{\bm{\mathcal E}}
\newcommand{\dX}{\bm{\mathcal X}}
\newcommand{\dY}{\bm{\mathcal Y}}
\newcommand{\dZ}{\bm{\mathcal Z}}

\newcommand{\vE}{\ul{\bm{\mathcal E}}}
\newcommand{\vX}{\ul{\bm{\mathcal X}}}
\newcommand{\vY}{\ul{\bm{\mathcal Y}}}
\newcommand{\vZ}{\ul{\bm{\mathcal Z}}}

\title{Massey products for homotopy inner products}

\author[K.~Poirier]{Kate~Poirier}
  \address{Kate Poirier,
  Department of Mathematics, New York City College of Technology, City University of New York, 300 Jay Street, Brooklyn, NY 11201}
  \email{kpoirier@citytech.cuny.edu}

\author[T.~Tradler]{Thomas~Tradler}
  \address{Thomas Tradler,
  Department of Mathematics, New York City College of Technology, City University of New York, 300 Jay Street, Brooklyn, NY 11201}
  \email{ttradler@citytech.cuny.edu}

\author[S.O.~Wilson]{Scott O.~Wilson}
  \address{Scott O. Wilson,
  Department of Mathematics, Queens College, City University of New York, 65-30 Kissena Blvd, Flushing, NY 11367}
  \email{scott.wilson@qc.cuny.edu}

\keywords{Massey product, $A_\infty$ algebra, Homotopy inner product}
\subjclass[2020]{55S30, 46C99 (primary), 08A65, 57K99 (secondary)}

\begin{document}
\maketitle
\begin{abstract} 
This paper defines Massey-type products for a homotopy inner product on an $A_\infty$ algebra, called \emph{Massey inner products}. We include an explicit description of ordinary Massey products for $A_\infty$ algebras, and for $A_\infty$ modules, and show that Massey product sets can have interesting internal structure. We give non-trivial applications from  lens spaces, links, and low dimensional manifolds, showing Massey inner products contain information beyond ordinary Massey products.
\end{abstract}

\setcounter{tocdepth}{1} \tableofcontents

\section{Introduction}

Massey products are higher order operations on cohomology that detect complexities in a differential graded ring beyond the induced ring structure in cohomology. Since Massey's introduction \cite{Mas58} these have seen numerous applications such as: the study of links \cite{Mas68}; formality in complex and K\"ahler geometry \cite{DGMS}, as well as symplectic topology \cite{BT}, \cite{FM04}; the homotopy theory of configuration spaces \cite{LS}; and more recently, the study of compact $G2$-manifolds \cite{MM1,MM2}. 

The existence of Massey products stems from the associativity condition and the possibility of solving various exactness equations in a differential graded ring in more than one way. The most basic algebraic example has as its geometric counterpart the Borromean rings:
\noindent\begin{minipage}{0.5\textwidth}
\hspace{3cm}
\begin{align*}
dz&= 0, \\
\textrm{if}\,\,\,  z & = xc \pm ay \\
  \textrm{with}\,\,\, dx & = ab \\ 
 \textrm{and}\,\,\,  dy & =bc 
\end{align*}
\end{minipage}%
\hspace{-1cm} 
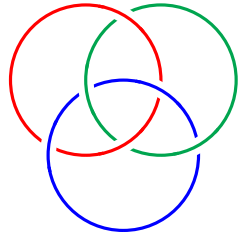
\begin{minipage}{0.5\textwidth}\raggedleft
\[
\begin{tikzpicture}\begin{knot}[clip width=5,clip radius=8pt]
\strand [very thick,red] (0,0) to [out=right, in=down] (1,1)
to [out=up, in=right] (0,2) to [out=left, in=up] (-1,1) to [out=down, in=left] (0,0);
\strand [very thick,Green] (1,0) to [out=right, in=down] (2,1)
to [out=up, in=right] (1,2) to [out=left, in=up] (0,1) to [out=down, in=left] (1,0);
\strand [very thick,blue] (.5,-1) to [out=right, in=down] (1.5,0)
to [out=up, in=right] (.5,1) to [out=left, in=up] (-.5,0) to [out=down, in=left] (.5,-1);
\flipcrossings{3,4}
\end{knot}\end{tikzpicture}
\]
\end{minipage}

\noindent

\vspace{1em}

\noindent
Here, the algebraic expressions on the left detect the triple linking on the right, even though no two circles are linked.

In such cases, Massey products create new cocycles that are never cohomologous to zero, regardless of how the exactness equations are solved. The set of all such solutions requires a  framework of so-called \emph{defining systems}, and establishing well-definedness requires some care. Fortunately, the resulting theory is natural with respect to ring maps and often easy to compute and work with.

In this paper we introduce higher order operations on cohomology classes for inner products. In its it most primitive form, these ``products'' appear again due to the ability to solve exactness equations, but in this case it is combined with the module compatibility of the cap product
\[
(\alpha \cup \beta) \cap x = \alpha \cap (\beta \cap x).
\]
Said differently, the evaluation of cochains on a closed chain $x$ yields a morphism of modules over the ring $C$ of cochains, $C \to C^*$, which is one way to formulate the structure of an inner product $C \otimes C \to R$ that is compatible with a ring structure on $C$.

This yields new higher order invariants of spaces, that are different than traditional Massey products, even on low degree classes with very few inputs. We call these \emph{Massey inner products} to indicate that they arise from the higher-homotopical nature of inner products. 

In the heart of the paper we describe the most general situation: Massey inner products for $A_\infty$ inner products on an $A_\infty$ algebra, where the latter notion was introduced in \cite{T}.  This includes as special cases the so-called cyclic $A_\infty$ algebras (which have a strict inner product) and the most elementary case of a commutative dga with a compatible pairing. 

In the full generality of $A_\infty$ structures, we achieve the following:
\begin{itemize}[label={-}]
\item introduce \emph{defining systems} for inner products (\ref{EQU:dx=-Sx,dx'=-Sx'})
\item define a \emph{staircase inner product} which outputs  \emph{Massey inner products} (\ref{EQU:F(x**,x'**)})
\item prove well-definedness of the Massey inner products (\ref{lem;MIPwelldef})
\item prove naturality of \emph{Massey inner products} with respect to $A_\infty$-morphisms (\ref{COR:equiv-AAIP-gives-equal-Massey}, \ref{PROP:MIP-under-morphism}) and pullback via the homotopy transfer theorem (\ref{thm;pullbacknat})
\item formulate and deduce results on the formality of $A_\infty$ inner products (\ref{DEF:AAIP-formal}, 
 \ref{cor;formalMIP})
 \end{itemize}

With the theory established, we give a few select examples, showing Massey inner products contain information beyond the ordinary Massey products. An illustrative example comes from the following $4$-link,
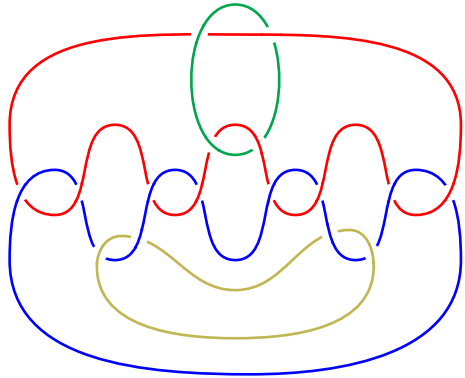
\[
\scalebox{0.4}{
\begin{tikzpicture}
\begin{knot}[clip width=8pt,clip radius=8pt, 
]
\strand [line width=2.5pt,blue]
(-.5,3) to [out=up, in=left] (1,6) to [out=right, in=left] (3,3) to [out=right, in=left] (5,6)
 to [out=right, in=left] (7,3)  to [out=right, in=left] (9,6)  to [out=right, in=left] (11,3)
 to [out=right, in=left] (13,6) to [out=right, in=up] (14.5,3) to [out=down, in=right] (7.5,-.8)
 to [out=left, in=down] (-.5,3);
 \strand [line width=2.5pt,red]
(-.5,7.5) to [out=down, in=left] (1,4.5) to [out=right, in=left] (3,7.5) to [out=right, in=left] (5,4.5)
 to [out=right, in=left] (7,7.5)  to [out=right, in=left] (9,4.5)  to [out=right, in=left] (11,7.5)
 to [out=right, in=left] (13,4.5) to [out=right, in=down] (14.5,7.5) to [out=up, in=right] (7.5,10.5)
  to [out=left, in=up] (-.5,7.5);
 \strand [line width=2.5pt,Green]
(7,6.5) to [out=right, in=right] (7,11.5) to [out=left, in=left] (7,6.5) ;
 \strand [line width=2.5pt,Yellow!70!black]
(3.2,3.8) to [out=right, in=left] (7,2) to [out=right, in=left] (10.8,4) to [out=right, in=up] (11.6,2.8)
to [out=down, in=right] (7,.4) to [out=left, in=down] (2.4,2.8) to [out=up, in=left] (3.2,3.8);
\flipcrossings{2,4, 6, 8, 9, 12, 13, 16}
\end{knot}
\end{tikzpicture}}
\]
which has a nontrivial Massey inner product (Example \ref{EXA:4-component-link}), even while there are no nontrivial quadruple Massey products (over $\Z$).

As with all $A_\infty$ structures, one often has pictures and diagrams in mind which are easier to understand than explicit formulas. For this reason we choose to review the entire theory of ordinary Massey products first, in the setting of $A_\infty$ algebras. This is implicit, as the special case of the $A_\infty$ operad, from the work of \cite{FM}. Here we introduce a graphical calculus for understanding and computing Massey products in the $A_\infty$ context, which carries over well to the case of Massey inner products, the main new conceptual point of this work. As an application of the case of ordinary Massey products, we show the lens spaces $L(p,q)$ have non-trivial $p$-fold Massey products
\[
\la c\cdot x,\dots,c\cdot x\ra=\{c\cdot y\}\subseteq H(L(p,q), \Z_p) 
\]
for any constant $c\in \Z_p-\{0\}$. This yields a new proof of the classical result that if there is a homotopy equivalence of lens spaces $L(p,q)\to L(p,q')$ then $q$ and $q'$ must satisfy 
\[
q\cdot q' \equiv \pm n^2  \mod p \quad \textrm{for some} \quad n\in \Z_p-\{0\}.
\]

The paper gives rise to many potential applications, which include obstructions to pseudo-manifolds being Poincar\'e duality spaces; obstructions to morphisms between spaces and more generally $A_\infty$ algebras with inner products, and various applications to formality and geometry. In particular,  Massey product sets may have internal structure themselves (such as an interesting topology) as they are solutions to polynomial equations in several variables (\ref{EXA:general-minimal}, \ref{COR:minimal-Massey}). As an illustration, in Example \ref{EXA:Massey-circle} the Massey product set is a circle in $\R^2$. This internal structure of the Massey product set together with naturality properties may also prove useful for determining obstructions mentioned above. Thus, there are many new avenues to explore.

Finally, we point out the recent emphasis on bigraded notions of formality and so-called $ABC$-Massey products \cite{ATo,SM}, defined in the context of commutative bi-graded bi-differential algebras, such as complex manifolds. It is expected  there is a parallel version of the theory presented here, of Massey inner products in the bi-graded context, as well. This is left for future work.

This paper is organized as follows. In Section \ref{SEC:Massey-A-infty} we provide a through and graphical overview of  Massey products for $A_\infty$ algebras, establishing basic properties with regard to naturality and the transfer theory of $A_\infty$ algebras. We describe the general structure of Massey products for a minimal $A_\infty$ algebra, and give the application to lens spaces, using technical details from Appendix \ref{APP:L(p,q)-mimimal} to compute the minimal $A_\infty$ structure for $L(p,q)$ over $\Z_p$.
In Section \ref{SEC:Massey-modules} we very briefly review the notions of $A_\infty$ modules over $A_\infty$ algebras, establish the notion of Massey products in this context, and give one example of a nontrivial Massey product coming from the Hopf fibration. 
Section \ref{SEC:Massey-inner-product} contains all the material on Massey inner products, for $A_\infty$ inner products on $A_\infty$ algebras, as well as their properties.
Section \ref{SEC:Massey-cdga} is the special case of cyclic commutative algebras, i.e. commutative dgas with evaluation on a dual element, which could be read in advance in case this most interests the reader. Here, we provide examples of nontrivial Massey inner products, including the example of the $4$-link above.

\textbf{Acknowledgements:} 
The authors would like to thank Aleksandar Milivojevi\'c for several illuminating discussions on Massey products, and for pointing out that the Hopf fibration is not a formal map. The third author acknowledges support provided by a PSC-CUNY Award, jointly funded by The
Professional Staff Congress and The City University of New York.

\section{Massey Products for $A_\infty$ algebras}\label{SEC:Massey-A-infty}

\subsection{Background on $A_\infty$ algebras}\label{SEC:background-A-infty}
Let $A=\bigoplus_{k\geq 0}A_k$ denote a graded module over a commutative and unital ring $R$; in particular, $R$ could be a field of characteristic zero or finite characteristic. An \emph{$A_\infty$ algebra} consists of a sequence of maps $\mu_k:A^{\otimes k}\to A$ for $k=1, 2, 3, \dots$, where the degree is $|\mu_k|=2-k$,  such that for any $a_i\in A$:
\begin{equation}\label{EQU:mu_mu}
\sum_{k+\ell-1=n} \sum_{i} \pm \mu_k(a_1,\dots, \mu_\ell(a_{i+1},\dots,a_{i+\ell}),\dots,a_{n})=0, \quad \forall n\geq 1.
\end{equation}
This can also be stated as $\mu^2=0$ when lifting $\mu$ as a coalgebra map on the tensor coalgebra; see \cite{S, T}. We will sometimes also write $d$ instead of $\mu_1$. The $A_\infty$ algebra is called minimal, if $\mu_1=0$.

To state the correct signs in Equation \eqref{EQU:mu_mu}, we shift $A$ down by one, denoted by $\ul{A}$, 
so that the shifted degrees of elements are $|\ul{a}|= |a| -1$, and similarly for operators.
In the shifted degree, each $\ul{\mu_k}:\ul{A}^{\otimes k}\to \ul{A}$ is of degree $|\ul{\mu_k}|=1$. Then, the sign in \eqref{EQU:mu_mu} is simply given by the Koszul sign rule, i.e., $``\pm''=(-1)^{|\ul{\mu_\ell}|\cdot (|\ul{a_1}|+\dots+|\ul{a_i}|)}$, 
so that \eqref{EQU:mu_mu} becomes
\begin{equation}\label{EQU:ULmu_ULmu}
\sum_{k+\ell-1=n}\,\, \sum_{i+j+1=k} \ul{\mu_k}\circ (\underbrace{id\otimes \dots \otimes id}_{i\text{ many}}  \otimes\ul{\mu_\ell}\otimes \underbrace{id\otimes \dots \otimes id}_{j\text{ many}})=0, \quad \forall n\geq 1
\end{equation}

The relation between the shifted and the unshifted $\mu$'s can be expressed via the shift operator $\mathfrak{s}:A\to \un{A}, a \mapsto \un{a},$ by taking $\ul{\mu_k}=\mathfrak{s}\circ \mu_k\circ (\mathfrak{s}^{-1}\ot \dots\ot \mathfrak{s}^{-1})$. Then (c.f. \cite[Proposition 1.4]{T}):
\begin{align}\label{EQU:muk-vs-ulmuk}
\ul{\mu_k}(\ul{a_1},\dots,\ul{a_k})=&\mathfrak{s}\circ \mu_k\circ (\mathfrak{s}^{-1}\ot \dots\ot \mathfrak{s}^{-1})(\ul{a_1},\dots,\ul{a_k})
\\ \nonumber
=&(-1)^{|\ul{a_1}|\cdot (k-1)+|\ul{a_2}|\cdot (k-2)+\dots +|\ul{a_{k-1}}|\cdot 1} \cdot  \ul{\mu_k(a_1,\dots,a_k)}.
\end{align}
For example, $\ul{d}(\ul{a})=\ul{\mu_1}(\ul{a})=\ul{\mu_1(a)}=\ul{d(a)}$ and $\ul{\mu_2}(\ul{a_1},\ul{a_2})=(-1)^{|\ul{a_1}|}\cdot \ul{\mu_2(a_1,a_2)}$. Thus, if for example $\mu_3=0$, Equation \eqref{EQU:ULmu_ULmu} for $n=3$ is $\ul{\mu_2}(\ul{\mu_2}(\ul{a_1},\ul{a_2}),\ul{a_3})+(-1)^{|\un{a_1}|}\cdot\ul{\mu_2}(\ul{a_1},\ul{\mu_2}(\ul{a_2},\ul{a_3}))=0$, which simply becomes the associativity condition, i.e., $\mu_2(\mu_2(a_1,a_2),a_3)=\mu_2(a_1,\mu_2(a_2,a_3))$.

An \emph{$A_\infty$ algebra morphism} $f:A\to B$ between $A_\infty$ algebras $(A,\mu)$ and $(B,\nu)$ consists of maps $f_k:A^{\otimes k}\to B$ for $k=1,2,3, \dots$ of degree $|f_k|=1-k$, such that the shifted maps $\ul{f_k}:\ul{A}^{\otimes k}\to \ul{B}$  (of degree $|\ul{f_k}|=0$) satisfy
\begin{multline}\label{EQU:f_f}
\sum_{k+\ell-1=n}\,\, \sum_{i+j+1=k} \ul{f_k}\circ (\underbrace{id\otimes \dots \otimes id}_{i\text{ many}}  \otimes\ul{\mu_\ell}\otimes \underbrace{id\otimes \dots \otimes id}_{j\text{ many}})
\\
=\sum_{\ell_1+\dots+\ell_k=n}   \ul{\nu_k}\circ (\ul{f_{\ell_1}}\ot\dots\ot\ul{f_{\ell_k}})=0, \quad \forall n\geq 1.
\end{multline}
We also state this equation as $f\circ \mu=\nu\circ f$, where we lift $f$ to the tensor coalgebra as a coalgebra map. As before, we can relate the shifted and the unshifted maps via  $\ul{f_k}=\mathfrak{s}_B\circ f_k\circ (\mathfrak{s}_A^{-1}\ot \dots\ot \mathfrak{s}_A^{-1})$, where $\mathfrak{s}_A$ and $\mathfrak{s}_B$ are the shifts for $A$ and $B$, respectively. We have a relation that is similar to \eqref{EQU:muk-vs-ulmuk}:
\begin{align}\label{EQU:fk-vs-ulfk}
\ul{f_k}(\ul{a_1},\dots,\ul{a_k})=&\mathfrak{s}_B\circ f_k\circ (\mathfrak{s}_A^{-1}\ot \dots\ot \mathfrak{s}_A^{-1})(\ul{a_1},\dots,\ul{a_k})
\\ \nonumber
=&(-1)^{|\ul{a_1}|\cdot (k-1)+|\ul{a_2}|\cdot (k-2)+\dots +|\ul{a_{k-1}}|\cdot 1} \cdot  \ul{f_k(a_1,\dots,a_k)}.
\end{align}
In particular, $\ul{f_1}(\ul{a})=\ul{f_1(a)}$ and $\ul{f_2}(\ul{a_1},\ul{a_2})=(-1)^{|\ul{a_1}|}\cdot \ul{f_2(a_1,a_2)}$.
\subsection{Defining systems and Massey products for $A_\infty$ algebras}\label{SEC:Def-system-Massey-A-infty}

In this section we introduce several definitions, ultimately defining higher Massey products in $A_\infty$ algebras and establishing naturality with respect to $A_\infty$ morphisms. This is implicit in the operad setup of \cite{FM}, though we wish to make it explicit in the $A_\infty$ case and introduce a graphical calculus as it illuminates the introduction of Massey inner products in later sections.

\begin{definition}  \label{defn:triangsys}
For any graded module $A$, a \emph{triangular system}, denoted it by $\xbb_{p,q}$ for $p \leq q$, is a collection of elements $x_{i,j}\in A$ for all $p\leq i\leq j\leq q$ with the possible exception of $x_{p,q}$, such that the degrees of the elements satisfy 
\[
|x_{i,j}|=|x_{i,i}|+|x_{i+1,i+1}|+\dots +|x_{j,j}|+i-j \quad \textrm{for all $i,j$}.
\]
Equivalently, in the shifted degrees, $|\ul{x_{i,j}}|=|\ul{x_{i,i}}|+|\ul{x_{i+1,i+1}}|+\dots +|\ul{x_{j,j}}|$ for all $i,j$. We may denote $\xbb_{p,q}$ by $\xbb$ if the $p$ and $q$ are clear.

 For any $A_\infty$ structure $(A, \mu)$ on $A$, each triangular system $\xbb$ determines an element in $A$ via the  \emph{staircase product of $\xbb$}, which is defined to be the element $\St{\xbb}{p,q}\in A$ given by
\begin{equation}\label{EQU:x**}
\ul{\St{\xbb}{p,q}}\\
:=\sum_{\tiny\begin{matrix}{r\geq 2}\\{p=j_0\leq j_1< j_2<\dots<j_r=q}\end{matrix}}
\ul{\mu_{r}}(\ul{x_{j_0,j_1}},\ul{x_{j_1+1,j_2}},\dots, \ul{x_{j_{r-2}+1,j_{r-1}}},\ul{x_{j_{r-1}+1,j_r}})
\end{equation}

\end{definition}

Several remarks are in order regarding the definition.  First, note that the index set for the sum has no \emph{gap} and no \emph{overlap} when going from $p=j_0$ to $j_1$, then from $j_1+1$ to $j_2$, from $j_2+1$ to $j_3$, all the way to $j_r=q$. Secondly, 
we explain a graphical calculus by means of example. We depict a triangular system, for example $\xbb_{2,5}$, on the grid as follows:
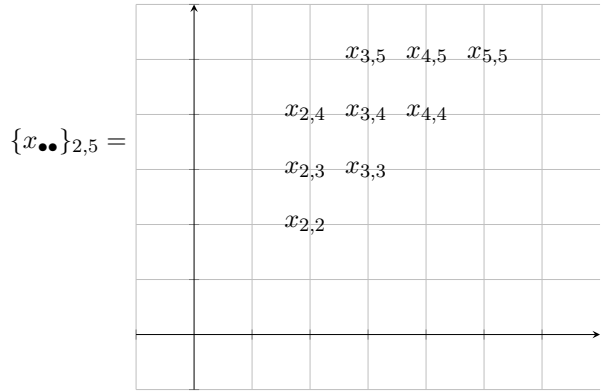
\[
\scalebox{0.9}{
\begin{tikzpicture}
\begin{axis}[axis lines = center, grid=both, xmin=-1, xmax=7, ymin=-1, ymax=6, xtick={-1,...,6},  ytick={-1,...,6}, xticklabels={}, yticklabels={}];
\end{axis}
 \node[anchor=north] at (-1,4) {$\xbb_{2,5}=$};
 \node[anchor=north] at (2.5,2.7) {$x_{2,2}$};
 \node[anchor=north] at (2.5,3.5) {$x_{2,3}$};
 \node[anchor=north] at (2.5,4.35) {$x_{2,4}$};
 \node[anchor=north] at (3.4,3.5) {$x_{3,3}$};
 \node[anchor=north] at (3.4,4.35) {$x_{3,4}$};
 \node[anchor=north] at (3.4,5.2) {$x_{3,5}$};
 \node[anchor=north] at (4.3,4.35) {$x_{4,4}$};
 \node[anchor=north] at (4.3,5.2) {$x_{4,5}$};
 \node[anchor=north] at (5.2,5.2) {$x_{5,5}$};
\end{tikzpicture}}
\]
The term $\ul{\mu_4}(\ul{x_{2,4}},\ul{x_{5,8}},\ul{x_{9,9}},\ul{x_{10,12}})$, which appears in the sum for $\St{\xbb_{2,12}}{2,12}$ is shown in the following ``staircase'' diagram:
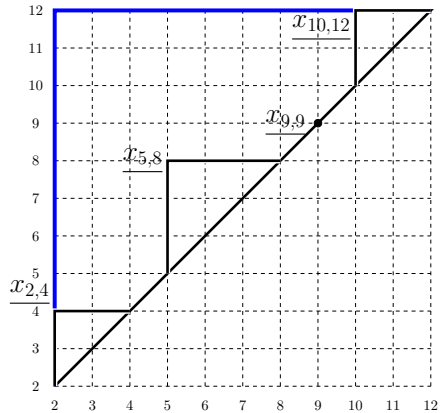
\[
\scalebox{0.5}{
\btik{2}{12}
\foreach \x in {2,3,4,5,6,7,8,9,10,11,12} 
{\draw[dashed, line width=.005pt] (\x,2)--(\x,12); \node at (\x,1.5){$\x$};
\draw[dashed] (2,\x)--(12,\x); \node at (1.5, \x){$\x$};}
\dblue{2}{12}
\dblack{2}{4}  \node[anchor=east] at (2,4.5) {\huge $\ul{x_{2,4}}$}; 
\dblack{5}{8}  \node[anchor=east] at (5,8) {\huge $\ul{x_{5,8}}$}; 
\draw [fill, black] (9,9) circle [radius=.1];  \node[anchor=east] at (8.8,9) {\huge $\ul{x_{9,9}}$}; 
\dblack{10}{12}  \node[anchor=north east] at (10,12) {\huge $\ul{x_{10,12}}$}; 
\end{tikzpicture} }
\]
Here the {\color{blue}color blue} indicates the application of a product $\mu_r$. In short we sum over all these terms, which we may write as sum over ``staircases'':
\[
\scalebox{0.45}{
\begintik
 \node[anchor=east] at (-1,5) {\huge $\ul{\St{\xbb}{p,q}}=\sum$}; 
\dblue{0}{8}
\dblack{0}{3} \dblack{3.3}{5} \dblack{5.3}{6} \dblack{6.3}{8}
\end{tikzpicture} }
\]

In many instances, we may be given some triangular system $\xbb_{p,q}$, and, for some given $p\leq p'\leq q'\leq q$, we restrict $\xbb_{p,q}$ to $\xbb_{p',q'}$. Thus, it makes sense to write $\St\xbb{p',q'}$ without referring to the index set (as long as the $x$'s are defined for the needed expression). Note that in particular, $ \St\xbb{\ell,\ell}=0$, and for small values of $p-q$, the expressions $\St\xbb{p,q}$ are given by (c.f. \eqref{EQU:muk-vs-ulmuk} for signs):
\begin{align*}
 \St\xbb{1,1}=&0  \\
 \St\xbb{1,2}=&(-1)^{|\ul{x_{1,1}}|}\cdot \mu_2(x_{1,1},x_{2,2})  \\
 \St \xbb{1,3}=&(-1)^{|\ul{x_{1,1}}|}\cdot \mu_2(x_{1,1},x_{2,3})+(-1)^{|\ul{x_{1,2}}|}\cdot \mu_2(x_{1,2},x_{3,3})\\\
 &+(-1)^{|\ul{x_{2,2}}|}\cdot \mu_3(x_{1,1},x_{2,2},x_{3,3})  \\
  \St\xbb{1,4}=&(-1)^{|\ul{x_{1,1}}|}\cdot\mu_2(x_{1,1},x_{2,4})+(-1)^{|\ul{x_{1,2}}|}\cdot\mu_2(x_{1,2},x_{3,4})
 +(-1)^{|\ul{x_{1,3}}|}\cdot\mu_2(x_{1,3},x_{4,4})
  \\& 
 +(-1)^{|\ul{x_{2,2}}|}\cdot\mu_3(x_{1,1},x_{2,2},x_{3,4})  +(-1)^{|\ul{x_{2,3}}|}\cdot\mu_3(x_{1,1},x_{2,3},x_{4,4})\\
 & +(-1)^{|\ul{x_{3,3}}|}\cdot\mu_3(x_{1,2},x_{3,3},x_{4,4}) \\
 & +(-1)^{|\ul{x_{1,1}}|+|\ul{x_{3,3}}|}\cdot\mu_4(x_{1,1},x_{2,2},x_{3,3},x_{4,4}).
\end{align*}

Next, we impose a natural condition on triangular systems in order for higher Massey products to be defined.

\begin{definition} \label{defn;defnsys}
A triangular system $\xbb = \xbb_{1,n}$ is called a \emph{defining system} if 
\[
d(x_{i,j})=- \St\xbb{i,j}
\]
for all  indices $(i,j) \neq (1,n)$, where  $1\leq i\leq j\leq n$.
\end{definition}

In particular, the definition implies that $d(x_{\ell,\ell})=- \St\xbb{\ell,\ell}=0$ for all $\ell$, i.e. the diagonal elements $x_{\ell,\ell}$ are cocycles. Informally, the higher conditions say that each element $x_{i,j}$ expresses the cocycle $ \St\xbb{i,j}$ as a coboundary.

We remark that in shifted degrees the condition for a defining system is that 
\begin{equation}\label{EQU:dxij=summuxxx}
\ul{d(x_{i,j})}=-\ul{ \St\xbb{i,j}}= \hspace{1em}-
\sum_{\tiny\begin{matrix}r\geq 2\\ i=j_0\leq j_1< j_2<\dots<j_r=j\end{matrix}}
\ul{\mu_{r}}(\ul{x_{j_0,j_1}},\dots, \ul{x_{j_{r-1}+1,j_r}}).
\end{equation}

We can now state our first result on the staircase product of defining systems:

\begin{proposition}\label{PROP:x**-closed}
If $\xbb$ is a defining system, then 
\[
d(  \St\xbb{1,n})=0.
\] 
Thus, the staircase product induces a map from the set of defining systems to the cohomology of $A$,
\[
\big\{\text{defining systems } \xbb \big\}\to H(A), \quad \quad \xbb\mapsto [ \St\xbb{1,n}].
\]
\end{proposition}

\begin{proof}
Evaluating $\ul{d(  \St\xbb{1,n})}$ gives
\begin{multline*}
\ul{d}( \ul{ \St\xbb{1,n}}) = \sum_{r\geq 2}\sum \ul{\mu_1}(\ul{\mu_r}(\ul{x_{j_0,j_1}},\dots,\ul{x_{j_{r-1}+1,j_r}}))\\
=- \sum_{r\geq 2}\sum(-1)^\varepsilon\cdot \ul{\mu_r}(\ul{x_{j_0,j_1}},\dots,\ul{\mu_1}(\ul{x_{j_{\ell}+1,j_{\ell+1}}}),\dots,\ul{x_{j_{r-1}+1,j_r}})\hspace{1.4cm}\\
 - \sum_{r,s\geq 2}\sum(-1)^\varepsilon\cdot \ul{\mu_r}(\ul{x_{j_0,j_1}},\dots,\ul{\mu_s}(\ul{x_{j_{\ell}+1,j_{\ell+1}}},\dots),\dots,\ul{x_{j_{r-1}+1,j_r}})=0,
\end{multline*}
where $(-1)^\varepsilon=(-1)^{|\ul{x_{j_{0},j_{1}}}|+\dots+|\ul{x_{j_{\ell-1}+1,j_\ell}}|}$. The second equality comes from the $A_\infty$ relations $\mu^2=0$. The last equality follows, because $\xbb_{1,n}$ is a defining system, and so $\ul{d}(\ul{x_{j_{\ell}+1,j_{\ell+1}}})=-\ul{ \St\xbb{j_{\ell}+1,j_{\ell+1}}}$, and thus the two sums cancel.
\end{proof}

\begin{definition}
For $x_{1,1},x_{2,2},x_{3,3},\dots,x_{n,n}\in A$, $n \geq 3$, we define the \emph{$n^{th}$ Massey product} (or \emph{$n^{th}$ Massey product set}) to be the subset of the cohomology $H(A)$ of $A$ given by
\begin{equation*} 
\la x_{1,1},x_{2,2},\dots,x_{n,n} \ra 
 :=\Big\{ [ \St\xbb{1,n}]\in H(A) \Big| \xbb_{1,n} \text{ is a defining system} \Big\} \subseteq H(A).
\end{equation*}
Note that this is well defined, since $\St\xbb{1,n}$ is closed by the previous proposition.

We say that the Massey product is \emph{trivial} if $\la x_{1,1},\dots,x_{n,n} \ra$ is either the empty set or contains $0\in H(A)$.
\end{definition}

The Massey product $\la x_{1,1},x_{2,2},\dots,x_{n,n} \ra$ depends only on the cohomology classes of the $x_{i,i}$'s, as we now explain.

\begin{lemma}\label{LEM:Zs}
Let $\xbb_{1,n}$ be a defining system. For a fixed $1\leq k\leq \ell\leq n$ with $(k,\ell)\neq(1,n)$, and $z\in A$ with $|z|=|x_{k,\ell}|-1$, let
\begin{align*} 
\wt{x_{k,\ell}}&:= x_{k,\ell}+dz,\quad \text{and let}\\
\wt{x_{i,j}}&:= x_{i,j}\quad\text{ for any $i,j$ for which either $k<i$ or $j<\ell$.}
\end{align*}
Then these $\wt{x_{i,j}}$ can be completed to a defining system $\{\wt{x_{\bu\bu}}\}$ for which
\[
[  \St{\{ \wt{x_{\bu\bu}} \}}{p,q} ]= [  \St{\{ {x_{\bu\bu} }\}}{p,q} ] \in H(A),
\]
for all $1\leq p\leq q\leq n$.
\end{lemma}

In the case where we vary elements on the diagonal $\wt{x_{\ell,\ell}}=x_{\ell,\ell}+dz$, and we take $(p,q)=(1,n)$, the above lemma shows that all of the cohomology classes in $\la x_{1,1},\dots, x_{\ell,\ell},\dots,x_{n,n}\ra$ and those in $\la x_{1,1},\dots, x_{\ell,\ell}+dz,\dots,x_{n,n}\ra$ coincide. Thus we obtain:
\begin{theorem} \label{thm;MPsetmap}
The $n^{th}$ Massey product induces a well-defined set map
\[
\la -,\dots,- \ra :H(A)\times\dots\times H(A)\to \mathcal P(H(A)),
 \quad \la [x_{1,1}],\dots,[x_{n,n}] \ra:=\la x_{1,1},\dots,x_{n,n} \ra
\]
where  $\mathcal P(H(A))$ is the power set of $H(A)$. For any given collection of cohomology classes, we refer to the output as the Massey product set.
\end{theorem}

Now we give:

\begin{proof}[Proof of Lemma \ref{LEM:Zs}] 
We modify the given defining system $\xbb$ appropriately by letting
\begin{equation}\label{EQU:xtilde=x+Z}
\wt{x_{p,q}}:=x_{p,q}+Z_{p,q},
\end{equation}
for all $p,q$ with $1\leq p\leq k$ and $\ell\leq q\leq n$ but $(p,q)\neq (k,\ell)$, where $Z_{p,q}$ is defined by
\begin{equation}\label{EQU:Def-of-Z-explicit}
\ul{Z_{p,q}}:=\sum_{r\geq 2} \sum_{\tiny\begin{matrix}{s,t\geq 0, \quad s+t+1=r}\\{p\leq j_1< j_2<\dots<j_s=k-1}\\{\ell+1\leq j'_1< j'_2<\dots<j'_t=q} \end{matrix}}
(-1)^{\varepsilon}\cdot 
\ul{\mu_{r}}(\ul{x_{p,j_1}},\dots, \ul{x_{j_{s-1}+1,j_s}}, \ul{z},\ul{x_{\ell+1,j'_1}}, \dots , \ul{x_{j'_{t-1}+1,j'_t}}).
\end{equation}
Here, the sign is $(-1)^{\varepsilon}=(-1)^{|\ul{x_{p,j_1}}|+\dots+|\ul{x_{j_{s-1}+1,j_s}}|}$. Moreover, the second sum is a sum over all indices in the staircase product \eqref{EQU:x**} that include the index $(k,\ell)$, and there we replace the $x_{k,\ell}$ by $z$. We note that, when $s=0$ (i.e., $k=p$), the middle condition under the sum is considered to be vacuous, and, when $t=0$ (i.e., $\ell=q$), the bottom equation under the sum considered to be vacuous. In pictures this is depicted as: 
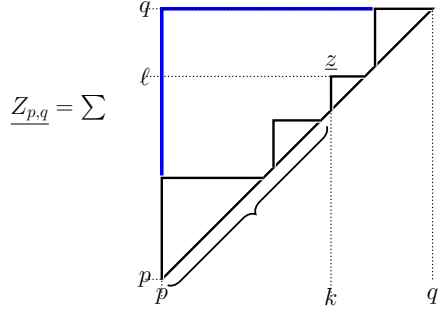
\[
\scalebox{0.45}{
\begintik
 \node[anchor=east] at (-1.5,5) {\huge $\ul{Z_{p,q}}=\sum$}; 
\dblue{0}{8}
\dblack{0}{3} \dblack{3.3}{4.7} \dblack{5}{6} \dblack{6.3}{8}
\node at (5,6.4) {\huge $\ul{z}$};  
\draw [dotted, line width=1pt] (5,-0.5)--(5,6)--(-0.5,6); 
\node at (5,-0.5) {\huge $k$};  \node at (-.5,6) {\huge $\ell$}; 
\draw [dotted, line width=1pt] (0,-0.5)--(0,0)--(-0.5,0); 
\node at (0,-0.5) {\huge $p$};  \node at (-.5,0) {\huge $p$}; 
\draw [dotted, line width=1pt] (8,0)--(8,8)--(-0.5,8); 
\node at (8,-0.5) {\huge $q$};  \node at (-.5,8) {\huge $q$}; 
\draw [decorate, line width=1.5pt, decoration = {brace, raise=7pt, amplitude=10pt}] (4.7,4.7) --  (0,0);
\end{tikzpicture} }
\]
where we keep track of the sign by putting the brace on the elements whose degrees contribute to $(-1)^\varepsilon$. 

It suffices to show that for all $1\leq p\leq q\leq n$,
\begin{equation}\label{EQU:xtilde=x+dZ}
 \St{\{\wt{x_{\bu\bu}}\}}{p,q}= \St{\{{x_{\bu\bu}}\}}{p,q}-d(Z_{p,q}),
\end{equation}
since then $\{\wt{x_{\bu\bu}}\}$ is a defining system, because for all $1\leq i\leq j\leq n$ with $(i,j) \neq (1,n)$,
\begin{equation}\label{EQU:dxtilde=-sxtilde}
d(\wt{x_{i,j}})=d(x_{i,j})+d(Z_{i,j})=- \St\xbb{i,j}+d(Z_{i,j})=- \St{\{\wt{x_{\bu\bu}}\}}{i,j}.
\end{equation}
Towards this goal, we compute $-\ul{d(Z_{p,q})}$ as follows:
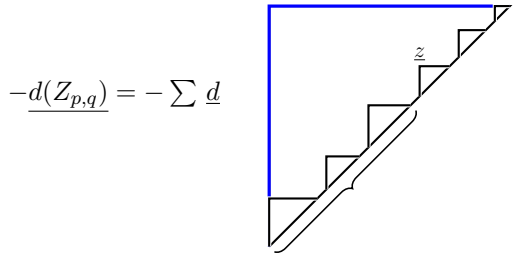
\begin{align*}
&\scalebox{0.4}{\begintik
 \node[anchor=east] at (-1.5,5) {\scalebox{2.3}{$-\ul{d(Z_{p,q})}=-\sum \, \ul{d}$}};  \dblue{0}{8}
\dblack{0}{1.6} \dblack{1.9}{3} \dblack{3.3}{4.7} \dblack{5}{6} \dblack{6.3}{7.2} \dblack{7.5}{8}
\node at (5,6.4) {\huge $\ul{z}$};  
\draw [decorate, line width=1.5pt, decoration = {brace, raise=7pt, amplitude=10pt}] (4.7,4.7) --  (0,0);
\end{tikzpicture} }
\\
&\scalebox{0.4}{\begintik
 \node[anchor=east] at (-.5,5) {\huge $=\sum$};  \dblue{0}{8}
\dblack{0}{1.6} \dblack{1.9}{3} \dblack{3.3}{4.7} \dblack{5}{6} \dblack{6.3}{7.2} \dblack{7.5}{8}
\node at (5,6.4) {\huge $\ul{z}$};  
\node at (2,2.6) {\scalebox{3.6}{$\ul{d}(\,\,\,\,\,)$}};
\draw [decorate, line width=1.5pt, decoration = {brace, raise=7pt, amplitude=10pt}] (4.7,4.7) --  (1.9,1.9);
\end{tikzpicture} }
\scalebox{0.4}{\begintik
 \node[anchor=east] at (-.5,5) {\huge $+\sum$};    \dblue{0}{8}
\dblack{0}{1.6} \dblack{1.9}{3} \dblack{3.3}{4.7} \dblack{5}{6} \dblack{6.3}{7.2} \dblack{7.5}{8}
\node at (4.6,6.6) {\huge $\ul{d(z)}$};  
\end{tikzpicture} }
\scalebox{0.4}{\begintik
 \node[anchor=east] at (-.5,5) {\huge $+\sum$};    \dblue{0}{8}
\dblack{0}{1.6} \dblack{1.9}{3} \dblack{3.3}{4.7} \dblack{5}{6} \dblack{6.3}{7.2} \dblack{7.5}{8}
\node at (4.5,6) {\huge $\ul{z}$};  
\node at (6.4,6.9) {\scalebox{3.6}{$\ul{d}(\,\,\,\,)$}};
 \draw [decorate, line width=1.5pt, decoration = {brace, raise=7pt, amplitude=10pt}] (6,6) --  (5,5);
\end{tikzpicture} }
\\
&\scalebox{0.4}{\begintik
 \node[anchor=east] at (-.5,5) {\huge $+\sum$};  \dblue{0}{8} \dblue{1.9}{4.7}
\dblack{0}{1.6} \dblack{1.9}{3} \dblack{3.3}{4.7} \dblack{5}{6} \dblack{6.3}{7.2} \dblack{7.5}{8}
 \node at (5,6.4) {\huge $\ul{z}$};  
 \draw [decorate, line width=1.5pt, decoration = {brace, raise=7pt, amplitude=10pt}] (4.7,4.7) --  (1.9,1.9);
\end{tikzpicture} }
\scalebox{0.4}{\begintik
 \node[anchor=east] at (-.5,5) {\huge $+\sum$};    \dblue{0}{8}  \dblue{3.3}{7.2}
\dblack{0}{1.6} \dblack{1.9}{3} \dblack{3.3}{4.7} \dblack{5}{6} \dblack{6.3}{7.2} \dblack{7.5}{8}
\node at (5,6.4) {\huge $\ul{z}$};  
 \draw [decorate, line width=1.5pt, decoration = {brace, raise=7pt, amplitude=10pt}] (4.7,4.7) --  (3.3,3.3);
\end{tikzpicture} }
\scalebox{0.4}{\begintik
 \node[anchor=east] at (-.5,5) {\huge $+\sum$};    \dblue{0}{8}  \dblue{6.3}{8}
\dblack{0}{1.6} \dblack{1.9}{3} \dblack{3.3}{4.7} \dblack{5}{6} \dblack{6.3}{7.2} \dblack{7.5}{8}
 \node at (5,6.4) {\huge $\ul{z}$};  
 \draw [decorate, line width=1.5pt, decoration = {brace, raise=7pt, amplitude=10pt}] (6,6) --  (5,5);
\end{tikzpicture} }
\end{align*}
The last equality comes from $\mu^2=0$. Now, the first and the fourth terms cancel, because $\xbb$ is a defining system (see \eqref{EQU:dxij=summuxxx}), and similarly the third and the sixth terms cancel. We thus get
\begin{equation}\label{EQU:d(Z)}
\scalebox{0.4}{\begintik
 \node[anchor=east] at (-.5,5) {\scalebox{2.3}{$-\ul{d(Z_{p,q})}=\sum$}};    \dblue{0}{8}
\dblack{0}{1.6} \dblack{1.9}{3} \dblack{3.3}{4.7} \dblack{5}{6} \dblack{6.3}{7.2} \dblack{7.5}{8}
\node at (4.6,6.6) {\huge $\ul{d(z)}$};  
\end{tikzpicture} }
\scalebox{0.4}{\begintik
 \node[anchor=east] at (-.5,5) {\scalebox{2.3}{$+\sum$}};    \dblue{0}{8}  \dblue{3.3}{7.2}
\dblack{0}{1.6} \dblack{1.9}{3} \dblack{3.3}{4.7} \dblack{5}{6} \dblack{6.3}{7.2} \dblack{7.5}{8}
 \node at (5,6.4) {\huge $\ul{z}$};  
 \draw [decorate, line width=1.5pt, decoration = {brace, raise=7pt, amplitude=10pt}] (4.7,4.7) --  (3.3,3.3);
\end{tikzpicture} }
\end{equation}

Now, we compute $\ul{ \St{\{\wt{x_{\bu\bu}}\}}{p,q}}$ to be
\begin{multline}\label{EQU:S(xtilde)=...}
\hspace{-2mm}\scalebox{0.4}{\begintik
 \node[anchor=east] at (-.5,5) {\scalebox{2.3}{$\ul{ \St{\{\wt{x_{\bu\bu}}\}}{p,q}}=\sum$}};    \dblue{0}{8}
\dblack{0}{1.6} \dblack{1.9}{3} \dblack{3.3}{6} \dblack{6.3}{7.2} \dblack{7.5}{8} 
\node at (0,1.7) {\scalebox{3}{$\wt{\,\,}$}}; 
\node at (1.9,3.1) {\scalebox{3}{$\wt{\,\,}$}}; 
\node at (3.3,6.1) {\scalebox{3}{$\wt{\,\,}$}}; 
\node at (6.3,7.3) {\scalebox{3}{$\wt{\,\,}$}}; 
\node at (7.5,8.1) {\scalebox{3}{$\wt{\,\,}$}}; 
\end{tikzpicture} }
\\
\scalebox{0.4}{\begintik
 \node[anchor=east] at (-0.7,5) {\scalebox{2.3}{$=\sum$}};    \dblue{0}{8}
\dblack{0}{1.6} \dblack{1.9}{3} \dblack{3.3}{4.5} 
\dblack{4.8}{6.9} \dblack{7.2}{8} 
\draw [dotted, line width=1pt] (4.3,-0.5)--(4.3,5.5)--(-0.5,5.5); 
\node at (4.4,-0.5) {\huge $k$};  \node at (-.5,5.5) {\huge $\ell$}; 
\end{tikzpicture} }
\scalebox{0.4}{\begintik
 \node[anchor=east] at (-1.2,5) {\scalebox{2.3}{$+\sum$}};    \dblue{0}{8}
\dblack{0}{1.6} \dblack{1.9}{3} \dblack{3.3}{6}  \dblack{6.3}{7.2} \dblack{7.5}{8} 
\node at (3.3,6.1) {\scalebox{3}{$\wt{\,\,}$}}; 
\draw [dotted, line width=1pt] (4.4,-0.5)--(4.4,5.2)--(-0.5,5.2); 
\node at (4.4,-0.5) {\huge $k$};  \node at (-.5,5.2) {\huge $\ell$}; 
\draw [dotted, line width=1pt] (3.3,-0.5)--(3.3,3.3);  \node at (3.3,-0.5) {\huge $i$}; 
\draw [dotted, line width=1pt] (6,-0.5)--(6,6);  \node at (6,-0.5) {\huge $j$}; 
\end{tikzpicture} }
\scalebox{0.4}{\begintik
 \node[anchor=east] at (-0.7,5) {\scalebox{2.3}{$+\sum$}};    \dblue{0}{8}
\dblack{0}{1.6} \dblack{1.9}{3} \dblack{3.3}{4} \dblack{4.3}{5.5} 
\dblack{5.8}{7.1} \dblack{7.4}{8} 
\draw [dotted, line width=1pt] (4.3,-0.5)--(4.3,5.5)--(-0.5,5.5); 
\node at (4.4,-0.5) {\huge $k$};  \node at (-.5,5.5) {\huge $\ell$}; 
 \node at (4.4,5.6) {\scalebox{3}{$\wt{\,\,}$}};  
\end{tikzpicture} }
\\
\scalebox{0.4}{\begintik
 \node[anchor=east] at (-0.7,5) {\scalebox{2.3}{$=\sum$}};    \dblue{0}{8}
\dblack{0}{1.6} \dblack{1.9}{3} \dblack{3.3}{4.5} 
\dblack{4.8}{6.9} \dblack{7.2}{8} 
\draw [dotted, line width=1pt] (4.3,-0.5)--(4.3,5.5)--(-0.5,5.5); 
\node at (4.4,-0.5) {\huge $k$};  \node at (-.5,5.5) {\huge $\ell$}; 
\end{tikzpicture} }
\scalebox{0.4}{\begintik
 \node[anchor=east] at (-1.2,5) {\scalebox{2.3}{$+\sum$}};    \dblue{0}{8}
\dblack{0}{1.6} \dblack{1.9}{3} \dblack{3.3}{6}  \dblack{6.3}{7.2} \dblack{7.5}{8} 
\node at (3.3,6.5) {\huge{$\ul{x_{i,j}}$}}; 
\draw [dotted, line width=1pt] (4.4,-0.5)--(4.4,5.2)--(-0.5,5.2); 
\node at (4.4,-0.5) {\huge $k$};  \node at (-.5,5.2) {\huge $\ell$}; 
\draw [dotted, line width=1pt] (3.3,-0.5)--(3.3,3.3);  \node at (3.3,-0.5) {\huge $i$}; 
\draw [dotted, line width=1pt] (6,-0.5)--(6,6);  \node at (6,-0.5) {\huge $j$}; 
\end{tikzpicture} }\hspace{-5mm}
\scalebox{0.4}{\begintik
 \node[anchor=east] at (-0.7,5) {\scalebox{2.3}{$+\sum$}};    \dblue{0}{8}
\dblack{0}{1.6} \dblack{1.9}{3} \dblue{3.3}{6} \dblack{6.3}{7.2} \dblack{7.5}{8} 
\draw [dotted, line width=1pt] (4.3,-0.5)--(4.3,5.1)--(-0.5,5.1); 
\node at (4.3,-0.5) {\huge $k$};  \node at (-.5,5.1) {\huge $\ell$}; 
\node at (4.3,5.5) {\huge $\ul{z}$}; 
\dblack{3.3}{4} \dblack{4.3}{5.1} \dblack{5.4}{6}
\draw [dotted, line width=1pt] (3.3,-0.5)--(3.3,3.3);  \node at (3.3,-0.5) {\huge $i$}; 
\draw [dotted, line width=1pt] (6,-0.5)--(6,6);  \node at (6,-0.5) {\huge $j$}; 
 \draw [decorate, line width=1.5pt, decoration = {brace, raise=7pt, amplitude=10pt}] (4,4) --  (3.3,3.3);
\end{tikzpicture} }
\\
\scalebox{0.4}{\begintik
 \node[anchor=east] at (-0.7,5) {\scalebox{2.3}{$+\sum$}};    \dblue{0}{8}
\dblack{0}{1.6} \dblack{1.9}{3} \dblack{3.3}{4} \dblack{4.3}{5.5} 
\dblack{5.8}{7.1} \dblack{7.4}{8} 
\draw [dotted, line width=1pt] (4.3,-0.5)--(4.3,5.5)--(-0.5,5.5); 
\node at (4.4,-0.5) {\huge $k$};  \node at (-.5,5.5) {\huge $\ell$}; 
\node at (4.3,6.05) {\huge $\ul{x_{k,\ell}}$};  
\end{tikzpicture} }
\scalebox{0.4}{\begintik
 \node[anchor=east] at (-0.7,5) {\scalebox{2.3}{$+\sum$}};    \dblue{0}{8}
\dblack{0}{1.6} \dblack{1.9}{3} \dblack{3.3}{4} \dblack{4.3}{5.5} 
\dblack{5.8}{7.1} \dblack{7.4}{8} 
\draw [dotted, line width=1pt] (4.3,-0.5)--(4.3,5.5)--(-0.5,5.5); 
\node at (4.4,-0.5) {\huge $k$};  \node at (-.5,5.5) {\huge $\ell$}; 
\node at (4.3,6.1) {\huge $\ul{d(z)}$};  
\end{tikzpicture} }
\end{multline}
The second equality uses that $\ul{\wt{x_{i,j}}}=\ul{x_{i,j}}$ when $k<i$ or $j<\ell$, (and we distinguish the two cases whether $\ul{\wt{x_{k,\ell}}}$ does or does not get applied to some $\ul{\mu_r}$). The third equality uses that $\ul{\wt{x_{i,j}}}=\ul{x_{i,j}}+\ul{Z_{i,j}}$ when $i\leq k$ and $\ell\leq j$ but $(i,j)\neq (k,\ell)$, while $\ul{\wt{x_{k,\ell}}}=\ul{x_{k,\ell}}+\ul{d(z)}$.

In the five sums on the right hand side, the first, second and fourth sum are exactly $\ul{\St\xbb{p,q}}$, whereas the second and the fifth sum are $-\ul{d(Z_{p,q})}$ by \eqref{EQU:d(Z)}. We thus get the desired result \eqref{EQU:xtilde=x+dZ}, i.e., $\ul{\St{\{\wt{x_{\bu\bu}}\}}{p,q}}=\ul{\St{\{{x_{\bu\bu}}\}}{p,q}}-\ul{d(Z_{p,q})}$.
\end{proof}

Another consequence of Lemma \ref{LEM:Zs} is that when computing defining systems, we only need to consider all choices representing cohomology classes at each level.

\begin{definition}
Two defining systems $\xbb_{1,n}$ and $\{x'_{\bu,\bu}\}_{1,n}$ are called \emph{essentially different} if there exist some $i<j$, such that the elements $x_{i,j}$ and $x'_{i,j}$ do not differ by an exact element, i.e., 
\[
 x_{i,j}\neq x'_{i,j}+dz \quad \text{for any }z\in A
\]
\end{definition}

The proof of the following proposition is given by inductively applying Lemma \ref{LEM:Zs}, which we leave to the reader.

\begin{proposition}
There exists a (non-unique) set $\mathcal D$ of defining systems such that
\begin{enumerate}
\item
$\la x_{1,1},x_{2,2},\dots,x_{n,n} \ra 
 =\Big\{ [ \St\xbb{1,n}]\in H(A) \Big| \xbb_{1,n} \in \mathcal D \Big\} \subseteq H(A)$, and
\item
any two elements $\xbb_{1,n}\neq \{x'_{\bu,\bu}\}_{1,n}$ in $\mathcal D$ are essentially different.
\end{enumerate}
\end{proposition}

\subsection{Naturality of Massey product sets}
 The Massey products behave well under $A_\infty$ algebra morphisms.
 
\begin{definition}\label{DEF:f(x**)}
Let $f=f_1+f_2+\dots:A\to B$ be an $A_\infty$ algebra morphism. Let $\xbb_{1,n}$ be a triangular system of elements in $A$. The triangular system of elements $\{y_{\bu\bu}\}_{1,n}$ in $B$ induced by $\xbb_{1,n}$ and $\{f_i\}$ is defined by
\begin{equation}\label{EQU:y**=f(x**)}
\ul{y_{p,q}}:=\sum_{r\geq 1}\sum_{p=j_0\leq j_1< j_2<\dots<j_r=q}
 \ul{f_{r}}(\ul{x_{j_0,j_1}},\ul{x_{j_1+1,j_2}},\dots, \ul{x_{j_{r-2}+1,j_{r-1}}},\ul{x_{j_{r-1}+1,j_r}}).
\end{equation}
In particular, for $p\leq \ell\leq q$, we have $y_{\ell,\ell}=f_1(x_{\ell,\ell})$.
In short, we denote this by $f(\{x_{\bu\bu}\}_{1,n}):=\{y_{\bu\bu}\}_{1,n}$.
\end{definition}

Graphically we represent the application of an $f_r$ with the {\color{red}color red}:
\begin{equation}\label{EQU:ys}
\scalebox{0.45}{
\begintik
 \node[anchor=east] at (-1,5) {\huge $\ul{y_{p,q}}=\sum$}; 
\dred{0}{8}
\dblack{0}{3} \dblack{3.3}{5} \dblack{5.3}{6} \dblack{6.3}{8}
\end{tikzpicture} }
\end{equation}
and the terms $y_{p,q}$ can be represented as a sum over staircases.
Note, that in this sum we also include $f_1$, whereas we did not include $\mu_1$ in the sum for $\St\xbb{p,q}$ in \eqref{EQU:x**}.

The terms $y_{p,q}$ for small $p-q$ are given similarly to the following expressions:
\begin{align*}
y_{1,1}=&f_1(x_{1,1})  \\
y_{1,2}=&f_1(x_{1,2})+(-1)^{|\ul{x_{1,1}}|}\cdot f_2(x_{1,1},x_{2,2})  \\
y_{1,3}=&f_1(x_{1,3})+(-1)^{|\ul{x_{1,1}}|}\cdot f_2(x_{1,1},x_{2,3})+(-1)^{|\ul{x_{1,2}}|}\cdot f_2(x_{1,2},x_{3,3})\\
& +(-1)^{|\ul{x_{2,2}}|}\cdot f_3(x_{1,1},x_{2,2},x_{3,3}) \\
y_{1,4}=&f_1(x_{1,4})+(-1)^{|\ul{x_{1,1}}|}\cdot f_2(x_{1,1},x_{2,4})\\
&+(-1)^{|\ul{x_{1,2}}|}\cdot f_2(x_{1,2},x_{3,4}) +(-1)^{|\ul{x_{1,3}}|}\cdot f_2(x_{1,3},x_{4,4})  \\
&+(-1)^{|\ul{x_{2,2}}|}\cdot f_3(x_{1,1},x_{2,2},x_{3,4})+(-1)^{|\ul{x_{2,3}}|}\cdot f_3(x_{1,1},x_{2,3},x_{4,4})\\
&+(-1)^{|\ul{x_{3,3}}|}\cdot f_3(x_{1,2},x_{3,3},x_{4,4})\\
&+(-1)^{|\ul{x_{1,1}}|+|\ul{x_{3,3}}|}\cdot f_4(x_{1,1},x_{2,2},x_{3,3},x_{4,4})
 \end{align*}

\begin{lemma}\label{LEM:f(defining-system)}
Let $f:A\to B$ be an $A_\infty$ algebra morphism, and let $\xbb_{1,n}$ be a triangular system of elements in $A$.
\begin{enumerate}
\item
If $\xbb_{1,n}$ is a defining system, then $f(\{x_{\bu\bu}\}_{1,n})$ is also a defining system.
\item
If $g:B\to C$ is another $A_\infty$ algebra map, then $g(f(\{x_{\bu\bu}\}_{1,n}))=(g\circ f)(\{x_{\bu\bu}\}_{1,n})$.
\end{enumerate}
\end{lemma}
\begin{proof}
Since $\xbb$ is a defining system,  $d(x_{i,j})=-\St\xbb{i,j}$ for all $1\leq i\leq j\leq n$, $(i,j)\neq(1,n)$. To prove (1), we need to show that $\{y_{\bu\bu}\}_{1,n}=f(\{x_{\bu\bu}\}_{1,n})$ satisfies $d(y_{i,j})=-\St\ybb{i,j}$ for all $1\leq i\leq j\leq n$, $(i,j)\neq(1,n)$. We depict the $A_\infty$ algebra structure for $B$ by graphing it in {\color{ForestGreen}green}, while the $A_\infty$ structure on $A$ is, as before, in {\color{blue}blue}.
\begin{align*}
&\scalebox{0.4}{
\begintik
 \node[anchor=east] at (-1,5) {\scalebox{2.5}{${\color{ForestGreen}\ul{d}}(\ul{y_{p,q}})=\sum {\color{ForestGreen}\ul{d}}$}}; 
\dred{0}{8} 
\dblack{0}{1} \dblack{1.3}{3} \dblack{3.3}{5} \dblack{5.3}{6} \dblack{6.3}{7.2}  \dblack{7.5}{8}
\end{tikzpicture} }
\\
&
\scalebox{0.4}{
\begintik
 \node[anchor=east] at (-.5,5) {\huge $=-\sum$}; 
\draw [ForestGreen, line width=3pt] (0,0)--(0,8)--(8,8);
\dred{0}{3} \dred{3.3}{6} \dred{6.3}{8}
\dblack{0}{1} \dblack{1.3}{3} \dblack{3.3}{5} \dblack{5.3}{6}  \dblack{6.3}{7.2}  \dblack{7.5}{8}
\end{tikzpicture} }\hspace{-6mm}
\scalebox{0.4}{
\begintik
 \node[anchor=east] at (-.5,5) {\huge $+\sum$}; 
\dred{0}{8} \node at (3.7,4.5) {\scalebox{3.6}{${\color{blue}\ul{d}}(\hspace{5mm})$}};
\dblack{0}{1} \dblack{1.3}{3} \dblack{3.3}{5} \dblack{5.3}{6} \dblack{6.3}{7.2}  \dblack{7.5}{8}
\end{tikzpicture} }\hspace{-6mm}
\scalebox{0.4}{
\begintik
 \node[anchor=east] at (-.5,5) {\huge $+\sum$}; 
\dred{0}{8} \dblue{1.3}{6}
\dblack{0}{1} \dblack{1.3}{3} \dblack{3.3}{5} \dblack{5.3}{6} \dblack{6.3}{7.2}  \dblack{7.5}{8}
\end{tikzpicture} }
\\
&
\scalebox{0.4}{
\begintik
 \node[anchor=east] at (-.5,5) {\huge $=-\sum$}; 
\draw [ForestGreen, line width=3pt] (0,0)--(0,8)--(8,8);
\dred{0}{3} \dred{3.3}{6} \dred{6.3}{8}
\dblack{0}{1} \dblack{1.3}{3} \dblack{3.3}{5} \dblack{5.3}{6}  \dblack{6.3}{7.2}  \dblack{7.5}{8}
\node[anchor=west]  at (8,5) {\scalebox{2.5}{$=-\ul{\St\ybb{p,q}}$}};
\end{tikzpicture} }
\end{align*}
The second equality comes from the relation ${\color{ForestGreen}\nu}\circ {\color{red}f}={\color{red}f}\circ {\color{blue}\mu}$. The third equality comes from the fact that $\xbb$ is a defining system, so that the second and third terms cancel ${\color{blue}\ul{d}}(\ul{x_{i,j}})=-\ul{\St\xbb{i,j}}=-\sum {\color{blue}\ul{\mu_r}}(\ul{x_{i,j_1}}, \dots , \ul{x_{ j_{r-1}+1 ,j}})$. The last equality follows from the definition of the $y_{i,j}$s in \eqref{EQU:ys} and Equation \eqref{EQU:x**}.

For (2), let $\{y_{\bu\bu}\}_{1,n}=f(\{x_{\bu\bu}\}_{1,n})$, $\{z_{\bu\bu}\}_{1,n}=g (f(\{x_{\bu\bu}\}_{1,n}))$, and moreover, let $\{z'_{\bu\bu}\}_{1,n}=(g\circ f)(\{x_{\bu\bu}\}_{1,n}))$. Then, since all $f_s$ are of shifted degree $|\ul{f_s}|=0$,
\begin{align*}
\ul{z'_{p,q}}=&\sum_{r\geq 1}\sum_{p=j_0\leq j_1< j_2<\dots<j_r=q}
\ul{(g\circ f)_{r}}(\ul{x_{j_0,j_1}},\dots, \ul{x_{j_{r-1}+1,j_r}})
\\
=& \sum_{r,s_1,\dots, s_r\geq 1}\sum \ul{g_r}(\ul{f_{s_1}}(\ul{x_{j_0,j_1}},\dots),\dots,\ul{f_{s_r}}(\dots, \ul{x_{j_{r-1}+1,j_r}}))
\\
&\hspace{-5mm}\scalebox{0.45}{
\begintik
 \node[anchor=east] at (-1,5) {\huge $=\sum$}; 
\draw [RawSienna, line width=3pt] (0,0)--(0,8)--(8,8);
\dred{0}{3} \dred{3.3}{5} \dred{5.3}{6} \dred{6.3}{8}
\dblack{0}{.8}  \dblack{1}{1.6}  \dblack{1.8}{3} 
\dblack{3.3}{4} \dblack{4.2}{5}
 \dblack{5.3}{5.5}  \dblack{5.7}{6}
  \dblack{6.3}{6.9}   \dblack{7.1}{7.6}   \dblack{7.8}{8}
\end{tikzpicture} }
\\
=&\sum_{r\geq 1}\sum_{p=j_0\leq j_1< j_2<\dots<j_r=q}
 \ul{g_{r}}(\ul{y_{j_0,j_1}},\dots, \ul{y_{j_{r-1}+1,j_r}})= \ul{z_{p,q}}
\end{align*}
Here, we represented the map $g$ in {\color{RawSienna}brown}.
\end{proof}

\begin{proposition} \label{prop;MPnatural}
If $\xbb$ is a defining system in $A$, so that $\{y_{\bu\bu}\}_{1,n}=f(\{x_{\bu\bu}\}_{1,n})$ is a defining system in $B$ for $f:A\to B$ an $A_\infty$ algebra morphism, then the closed elements $\St\xbb{1,n}\in A$ and $\St\ybb{1,n}\in B$ satisfy
\begin{equation}\label{EQU:f(x**<>)=y**<>+exact}
[f_1(\St\xbb{1,n}) ] = [ \St\ybb{1,n} ]  \in H(B).
\end{equation}

In particular, the map $(f_1)_*:H(A)\to H(B)$ on cohomology restricts to a set map on the Massey products satisfying
\begin{equation}\label{MasseySetIsSubset}
(f_1)_*\big( \la[x_{1,1}],\dots,[x_{n,n}]\ra \big) \subseteq  \la[f_1(x_{1,1})],\dots,[f_1(x_{n,n})] \ra.
\end{equation}
\end{proposition}

\begin{remark}
In general the subset indicated is not an equality, even if $f=f_1$ is dga map that is injective on cohomology with $B$ a formal dga, c.f. \cite{MSZ}[Example 2.4]; compare the claim of \cite{TO}[Lemma 5.3].
\end{remark}

\begin{proof}
We need to show Equation \eqref{EQU:f(x**<>)=y**<>+exact}, which we show in the shifted setting:
\begin{align*}
&\scalebox{0.40}{
\begintik
 \node[anchor=east] at (-1,5) {\scalebox{2.2}{${\color{red}\ul{f_1}}(\ul{\St\xbb{1,n}})=\sum {\color{red}\ul{f_1}}$}}; 
\dblue{0}{8}
\dblack{0}{1.8} \dblack{2}{3}  \dblack{3.3}{5} \dblack{5.3}{6} \dblack{6.3}{8}  \dblack{6.3}{8} 
\end{tikzpicture} }
\\
&\scalebox{0.4}{
\begintik
 \node[anchor=east] at (-1,5) {\scalebox{2.2}{$
 {=-
 \sum\limits_{r\geq 2} \sum }$}}; 
\dred{0}{8}  \node[anchor=east] at (-.3,8) {\huge ${\color{red}\ul{f_r}}$}; 
\dblue{2}{6}
\dblack{0}{1.8} \dblack{2}{3} \dblack{3.3}{5} \dblack{5.3}{6} \dblack{6.3}{8}
\end{tikzpicture} }\hspace{-3mm}
\scalebox{0.4}{
\begintik
 \node[anchor=east] at (-1,5) {\scalebox{2.2}{${-\sum\limits_{r\geq 2} \sum }$}}; 
\dred{0}{8}  \node[anchor=east] at (-.3,8) {\huge ${\color{red}\ul{f_r}}$}; 
 \node[anchor=east] at (5.6,4.5) {\scalebox{3}{${\color{blue}\ul{d}}(\hspace{6mm})$}}; 
\dblack{0}{1.8} \dblack{2}{3} \dblack{3.3}{5} \dblack{5.3}{6} \dblack{6.3}{8}
\end{tikzpicture} }
\\
&\scalebox{0.40}{
\begintik
 \node[anchor=east] at (-1,5) {\scalebox{2.2}{$+\sum\limits_{t\geq 2}\sum $}}; 
\dred{0}{8}  \node[anchor=east] at (-.3,8) {\huge ${\color{ForestGreen}\ul{\nu_t}}$}; 
\draw [ForestGreen, line width=3pt] (0,0)--(0,8)--(8,8);
\dred{0}{3} \dred{5.3}{8} 
\dblack{0}{1.8} \dblack{2}{3}  \dblack{3.3}{5} \dblack{5.3}{6} \dblack{6.3}{8}
\draw [fill, red] (3.3,5) circle [radius=.2];   
\end{tikzpicture} }
\scalebox{0.40}{
\begintik
 \node[anchor=east] at (-1,5) {\scalebox{2.2}{$+\sum {\color{ForestGreen}\ul{d}} $}}; 
\dred{0}{8} 
\dblack{0}{1.8} \dblack{2}{3}  \dblack{3.3}{5} \dblack{5.3}{6} \dblack{6.3}{8}
\end{tikzpicture} }
\\
&= \ul{\St\ybb{1,n}} + {\color{ForestGreen}\ul{d}}(\text{some terms})
\end{align*}
The second equality comes from the relation ${\color{red}f}\circ {\color{blue}\mu}={\color{ForestGreen}\nu}\circ {\color{red}f}$. The third equality follows since the first and second term cancel due to $\xbb$ being a defining system, and the third term is precisely $\ul{\St\ybb{1,n}}$, while the fourth term is exact, as claimed.
\end{proof}

\begin{corollary} \label{cor;MPsetbij}
Over a field (of any characteristic), if $f:A\to B$ is an $A_\infty$ quasi-isomorphism, there is an induced bijection of Massey product sets
\[
(f_1)_*\left(\la[x_{1,1}],\dots,[x_{n,n}]\ra\right)= \la[f_1(x_{1,1})],\dots,[f_1(x_{n,n})]\ra\quad\subseteq H(B).
\]

More generally, over any ring $R$, if $A$ and $B$ are $A_\infty$ algebras, and there are $A_\infty$ morphisms $f:A\to B$  and $g:B \to A$ such that $(f_1)_*: H(A) \to H(B)$ and $(g_1)_*:H(B) \to H(A)$ are inverses of each other, then  these maps induce bijections of the above Massey products sets.
\end{corollary}

Thus, Massey product sets on $A$ can be computed using equivalent minimal $A_\infty$ structures on cohomology $H(A)$, as will be done in some applications below. In particular, this includes the well-known case that formal algebras have no non-trivial Massey products, since if there is a defining system, one can choose $x_{i,j}=0$ for all $1\leq i<j\leq n$ with $(i,j)\neq (1,n)$.

\subsection{Computations}

We consider Massey products on minimal $A_\infty$ algebras $(H,\mu=\mu_1+\mu_2+\dots)$, i.e., for the case $\mu_1=0$. We start by analyzing the general situation, before computing specific examples.

\begin{example}\label{EXA:general-minimal}
(Explicit form of Massey product sets for minimal $A_\infty$-algebras) 

Let $(H,\mu=\mu_1+\mu_2+\dots)$ with $\mu_1=0$ be a minimal $A_\infty$ algebra, where each $H^k$ is finitely generated over our ground ring $R$. We want to determine the Massey product for some cycles $x_{1,1},\dots, x_{n,n}$. We first consider the general form $\{x_{\bu\bu}\}$ of a defining system for these cycles.
\[
\begin{matrix}
  & x_{2,n} & \dots & x_{n-1,n} & x_{n,n} \\
 x_{1,n-1} & x_{2,n-1} & \dots & x_{n-1,n-1} \\
 \vdots & \vdots  & \iddots \\
 x_{1,2} & x_{2,2} \\
 x_{1,1}
\end{matrix}
\]
Recall that in shifted degrees, $\deg(i,j):=|\ul{x_{i,j}}|=\sum_{k=i}^j|\ul{x_{k,k}}|$, while each $|\ul{\mu_k}|=1$. In order for this to be a defining system, these elements need to satisfy $\mu_1(x_{p,q})=-\St\xbb{p,q}$ for all $1\leq p\leq q\leq n$ with $(p,q)\neq (1,n)$. Our assumption of $\mu_1=0$ implies that we therefore must have the condition $\St\xbb{p,q}=0$, and, in case the condition is satisfied, each $\ul{x_{p,q}}$ can be chosen to be an arbitrary element in $\ul{H}^{\deg(p,q)}$ (since any element in $\ul{H}^k$ is closed). If we pick a generating set $\{\ul{b^k_{i}}\}_i$ for $\ul{H}^k$, then the $\ul{x_{p,q}}$ are given by
\begin{equation}\label{EQU:d=0-cijl}
\ul{x_{p,q}}=\sum_\ell c_{p,q,\ell}\cdot \ul{b^{\deg(p,q)}_\ell}\quad\quad\text{for any }c_{p,q,\ell}\in R.
\end{equation}
The conditions $\St\xbb{p,q}=0\in \ul{H}^{\deg(p,q)+1}$ are polynomial conditions in the variables $c_{i,j,\ell}$ given by \eqref{EQU:d=0-cijl}:
\begin{equation}\label{EQU:d=0-condition}
0=\sum_{r\geq 2}\sum_{j_1,\dots,j_{r-1}}
\ul{\mu_{r}}(\ul{x_{p,j_1}},\ul{x_{j_1+1,j_2}},\dots, \ul{x_{j_{r-2}+1,j_{r-1}}},\ul{x_{j_{r-1}+1,q}})
\end{equation}
and the Massey product is given by the set of all output values of the polynomial in the variables $c_{i,j,\ell}$  satisfying the above conditions \eqref{EQU:d=0-condition}.
\begin{multline}\label{EQU:d=0-Massey}
\scalebox{.97}[1]{$\la x_{1,1},x_{2,2},\dots,x_{n,n} \ra
= \{[ \St{\{x_{\bu\bu}\}}{1,n}] | \,\, \forall \, c_{i,j,\ell} \, \text{ from \eqref{EQU:d=0-cijl} satisfying \eqref{EQU:d=0-condition}}\},$} \\
\text{where } \ul{\St{\{x_{\bu\bu}\}}{1,n}}
=\sum_{r\geq 2}\sum_{j_1,\dots,j_{r-1}}
\ul{\mu_{r}}(\ul{x_{1,j_1}},\dots,\ul{x_{j_{r-1}+1,n}}).
\end{multline}
\end{example}

\begin{corollary}\label{COR:minimal-Massey}
If $A$ is a dga over a field (of any characteristic), then there is an $A_\infty$ algebra structure on the cohomology $H(A)$ of $A$ which is quasi-isomorphic to $A$. Therefore, the Massey products are given by the set in \eqref{EQU:d=0-Massey} under the polynomial conditions \eqref{EQU:d=0-condition} from Example \ref{EXA:general-minimal} for the minimal $A_\infty$ structure on $H(A)$.
\end{corollary}
\begin{proof}
A theorem by Kadeishvili \cite{Kad82}, any $A_\infty$ algebra over a field can be transferred quasi-isomorphically to its cohomology; see \cite{M, LV} and Theorem \ref{THM:transfer-theorem}.
\end{proof}

We give two examples of minimal $A_\infty$ algebras (that are not dgas) to illustrate the close relationship between Massey products and solutions of polynomial equations. 

\begin{example}\label{EXA:non-triv-R-C}
Let $H=\bigoplus H^k$ with $\ul{H}^k$ generated by $\ul{b^k}$ in shifted degrees $k=2$, $4$, $6$, $8$, $12$, $15$, $16$, $24$, $29$, $31$.
We construct an example with a non-trivial quadruple Massey product
 $\la x_{1,1},x_{2,2},x_{3,3},x_{4,4}\ra$, where $|\ul{x_{i,i}}|=2^i$, i.e., we take 
\[
\ul{x_{1,1}}:=\ul{b^{2}}, \quad \ul{x_{2,2}}:=\ul{b^{4}},\quad \ul{x_{3,3}}:=\ul{b^{8}},\quad \ul{x_{4,4}}:=\ul{b^{16}}.
\]
 Then, a defining system $\{x_{\bu\bu}\}$ lives in the following (shifted) degrees:
\begin{equation}\label{EQU:EXA:24816-degrees}
\begin{matrix}
  & \ul{x_{2,4}} & \ul{x_{3,4}} & \ul{x_{4,4}}  &&    & 
       \ul{H}^{28} &  \ul{H}^{24} &  \ul{H}^{16}  \\
 \ul{x_{1,3}} & \ul{x_{2,3}}  & \ul{x_{3,3}}   &&\hspace{5mm}\in\hspace{5mm} &  
      \ul{H}^{14} &  \ul{H}^{12}  &  \ul{H}^{8} \\
 \ul{x_{1,2}} & \ul{x_{2,2}}                 &&&&   
     \ul{H}^{6} &  \ul{H}^{4}         \\
 \ul{x_{1,1}}                                &&&&&
     \ul{H}^{2}
\end{matrix}
\end{equation}

We set $\ul{\mu_2}(\ul{x_{i,i}}, \ul{x_{j,j}})=0$ for all $i,j$, then $\St{\{x_{\bu\bu}\}}{1,2}=0$, $\St{\{x_{\bu\bu}\}}{2,3}=0$, and $\St{\{x_{\bu\bu}\}}{3,4}=0$ are automatically true, and the $x_{i,i+1}$ are given by 
\[
\ul{x_{1,2}}=c_{1,2}\cdot \ul{b^{6}},\quad\quad
\ul{x_{2,3}}=c_{2,3}\cdot \ul{b^{12}},\quad\quad
\ul{x_{3,4}}=c_{3,4}\cdot \ul{b^{24}}.
\]
for some $c_{i,i+1}\in R$. Next, we set
\[
\ul{\mu_2}(\ul{b^{2}},\ul{b^{12}}):=\ul{\mu_2}(\ul{b^{6}},\ul{b^{8}}):=\ul{b^{15}}
\quad\text{ and }\quad
\ul{\mu_2}(\ul{b^{4}},\ul{b^{24}}):=\ul{\mu_2}(\ul{b^{12}},\ul{b^{16}}):=\ul{b^{29}}
\] 
and $\ul{\mu_3}=0$. Then, the conditions $\St{\{x_{\bu\bu}\}}{1,3}=0$ and $\St{\{x_{\bu\bu}\}}{2,4}=0$ give
\begin{align*}
0=\ul{\St{\{x_{\bu\bu}\}}{1,3}}=\ul{\mu_2}(\ul{x_{1,1}},\ul{x_{2,3}})+\ul{\mu_2}(\ul{x_{1,2}},\ul{x_{3,3}})=(c_{2,3}+c_{1,2})\cdot \ul{b^{15}}
\\
0=\ul{\St{\{x_{\bu\bu}\}}{2,4}}=\ul{\mu_2}(\ul{x_{2,2}},\ul{x_{3,4}})+\ul{\mu_2}(\ul{x_{2,3}},\ul{x_{4,4}})=(c_{3,4}+c_{2,3})\cdot \ul{b^{29}}
\end{align*}
Thus, any defining system must satisfy $c_{1,2}=-c_{2,3}=c_{3,4}$, and, since $\ul{H}^{14}\cong \ul{H}^{28}\cong \{0\}$, we have $\ul{x_{1,3}}=0$ and $\ul{x_{2,4}}=0$. Finally, if we set
\[
\ul{\mu_2}(\ul{b^{6}},\ul{b^{24}}):=\ul{\mu_4}(\ul{b^{2}},\ul{b^{4}},\ul{b^{8}},\ul{b^{16}}):=\ul{b^{31}}
\]
and all other products involving $\ul{\mu_k}$ to be zero. Note, that $\mu^2=0$ since $\mu$ is non-zero only for inputs in even (shifted) degrees, while all of the outputs are in odd (shifted) degrees. With this, we obtain the following Massey product:
\begin{align*}
&\ul{\St{\{x_{\bu\bu}\}}{1,4}}
=\ul{\mu_2}(\ul{x_{1,2}},\ul{x_{3,4}})+\ul{\mu_4}(\ul{x_{1,1}},\ul{x_{2,2}},\ul{x_{3,3}},\ul{x_{4,4}})=(c^2_{1,2}+1)\cdot \ul{b^{31}}\\
&\quad \implies  \la x_{1,1},x_{2,2},x_{3,3},x_{4,4}\ra=\{(c^2+1)\cdot [b^{31}]\,\, | \,\, c\in R\}
\end{align*}

Thus,  if $R=\R$ is the real numbers, the Massey product does not contain zero, since we cannot solve $c^2+1=0$ in $\R$, and thus we have a non-trivial Massey product given by $[1,\infty)\cdot  [b^{31}]\subseteq H^{32}$.

On the other hand, if $R=\C$ is the complex numbers, we can set $c=\sqrt{-1}$, and thus the Massey product is trivial as it contains the zero element. 
\end{example}

\begin{example}\label{EXA:non-op-triv-R-C}
Let $H'$ be similar to the previous Example \ref{EXA:non-triv-R-C}, but with two more generators $\ul{b^{32}}\in \ul{H}^{32}$ and $\ul{b^{63}}\in \ul{H}^{63}$, i.e., 
$\ul{H'}=\ul{H}\oplus R\langle \ul{b^{32}} \rangle \oplus R \langle \ul{b^{63}}\rangle $. 
We set all products involving $\ul{b^{32}}$ to be zero with the single exception that
\[
\ul{\mu_5}(\ul{b^{2}},\ul{b^{4}},\ul{b^{8}},\ul{b^{16}},\ul{b^{32}})=\ul{b^{63}}
\]
Then, to compute $\la x_{1,1},x_{2,2},x_{3,3},x_{4,4},x_{5,5}\ra$ with $\ul{x_{5,5}}:=\ul{b^{32}}$, we need to repeat the consideration from in Example \ref{EXA:non-triv-R-C} when determining all defining systems. However, in this case, if $R=\R$ is the real numbers, the condition $0=\ul{\St{\{x_{\bu\bu}\}}{1,4}}=(c^2_{1,2}+1)\cdot \ul{b^{31}}$ shows that there is no solution, i.e., the Massey product $\la x_{1,1},x_{2,2},x_{3,3},x_{4,4},x_{5,5}\ra$ is trivial. On the other hand, for $R=\C$, we can solve this, and we get a non-trivial Massey product:
\[
\la x_{1,1},x_{2,2},x_{3,3},x_{4,4},x_{5,5}\ra=\{ [b^{63}] \}\subseteq H^{64}
\]
\end{example}

\begin{remark}\label{RMK:Massey-poly-equ}
As in Examples \ref{EXA:non-triv-R-C} and \ref{EXA:non-op-triv-R-C} above, one can generate Massey products depending on any number of polynomial equations in any number of variables. Thus, the existence of non-trivial Massey products is closely related to the existence of solutions to these polynomial equations over the given ring. A thorough investigation of this for dgas, and the behavior under field extension, was given by Milivojevic in \cite{Mil}.
\end{remark}

The following example has a Massey product set given by a circle in $\R^2$. We emphasize that the naturality property of Equation \ref{MasseySetIsSubset} places non-trivial conditions on the existence of  $A_\infty$ morphisms, e.g., an injective linear map $f_1$ cannot map a Massey product set which is a line to a circle, etc.

\begin{example}\label{EXA:Massey-circle}
Let the ground ring $R=\R$ be the real numbers, and let $H$ be similar to the previous examples given by generators $\ul{b^{2^1}}, \ul{b^{2^2}}, \dots, \ul{b^{2^9}}$, as well as the additional even generators $\ul{b^{6}},\ul{b^{12}},\ul{b^{24}},\ul{b^{96}},\ul{b^{192}},\ul{b^{384}}$ and the odd generators $\ul{b^{15}},\ul{b^{29}},\ul{b^{225}},\ul{b^{449}},\ul{b^{511}}$ and two generators $\ul{b^{1023}}, \ul{\wt{b}^{1023}}$. Here each upper index denotes the shifted degree of the generator, i.e., $\ul{b^{k}}\in \ul{H}^k$. We impose that the only non-trivial multiplications are the following:
\begin{align}
&\label{EQU:circle-1}
\ul{\mu_2}(\ul{b^2},\ul{b^{12}})=\ul{\mu_2}(\ul{b^6},\ul{b^8})=\ul{b^{15}}, &&
\ul{\mu_2}(\ul{b^{12}},\ul{b^{16}})=\ul{\mu_2}(\ul{b^4},\ul{b^{24}})=\ul{b^{29}},
\\ &\label{EQU:circle-2}
\ul{\mu_2}(\ul{b^{32}},\ul{b^{192}})=\ul{\mu_2}(\ul{b^{96}},\ul{b^{128}})=\ul{b^{225}}, &&
\ul{\mu_2}(\ul{b^{64}},\ul{b^{384}})=\ul{\mu_2}(\ul{b^{192}},\ul{b^{256}})=\ul{b^{449}},
\\\label{EQU:circle-3}
 \ul{\mu_6}(\ul{b^{6}}&,\ul{b^{24}}, \ul{b^{32}},\ul{b^{64}},\ul{b^{128}},\ul{b^{256}}) =
\ul{\mu_6}(\ul{b^{2}},\ul{b^{4}},\ul{b^{8}},\ul{b^{16}},\ul{b^{96}},\ul{b^{384}})
= -\ul{\mu_8}(\ul{b^{2^1}},\dots,\ul{b^{2^8}})=\ul{b^{511}},\hspace{-6.8cm}
\\\label{EQU:circle-4}
& \ul{\mu_8}(\ul{b^{6}},\ul{b^{8}},\ul{b^{16}},\ul{b^{32}},\ul{b^{64}},\ul{b^{128}},\ul{b^{256}}, \ul{b^{512}})=\ul{b^{1023}}\hspace{-4cm}
\\\label{EQU:circle-5}
& \ul{\mu_8}(\ul{b^{2}},\ul{b^{4}},\ul{b^{8}},\ul{b^{16}},\ul{b^{32}},\ul{b^{64}},\ul{b^{384}},\ul{b^{512}})=\ul{\wt{b}^{1023}}\hspace{-4cm}
\end{align}
Note that $\mu^2=0$ since all inputs for non-trivial $\ul{\mu_k}$ are of even degree, and all outputs are of odd degree. We claim that the $9$-fold Massey product $\la {b^{2^1}}, {b^{2^2}}, \dots, {b^{2^9}} \ra$ is a circle.

Let $\xbb$ be a defining system for $\ul{b^{2^1}}, \ul{b^{2^2}}, \dots, \ul{b^{2^9}}$. By degree reasons, it has to be of the following form (for some coefficients $c_1,\dots, c_6\in R$):
\begin{equation*}
\begin{matrix}
    & 0 & 0 & 0 & 0 & 0 & 0 & 0 & \ul{b^{512}} \\
 0 & 0 & 0 & 0 & 0 & 0 & c_6\cdot \ul{b^{384}} & \ul{b^{256}} \\
 0 & 0 & 0 & 0 & 0 & c_5\cdot \ul{b^{192}} & \ul{b^{128}} \\
 0 & 0 & 0 & 0 & c_4\cdot \ul{b^{96}} & \ul{b^{64}} \\
 0 & 0 & 0 & 0 & \ul{b^{32}} \\
 0 & 0 & c_3\cdot \ul{b^{24}} & \ul{b^{16}} \\
 0 & c_2\cdot \ul{b^{12}} & \ul{b^{8}} \\
 c_1\cdot \ul{b^{6}} & \ul{b^{4}} \\
  \ul{b^{2}}
\end{matrix}
\end{equation*}
These have to satisfy certain conditions to be a defining system. For example, to get to the $(1,3)$-spot, we must have $\ul{\mu_2}(\ul{b^2},c_2\cdot \ul{b^{12}})+\ul{\mu_2}(c_1\cdot \ul{b^6},\ul{b^8})=0$, which by \eqref{EQU:circle-1} means that $c_1+c_2=0$. Similarly, for the $(2,3)$-spot, \eqref{EQU:circle-1} implies $c_2+c_3=0$. Similarly, \eqref{EQU:circle-2} implies for the $(5,7)$- and $(6,8)$-spots that $c_4+c_5=0$ and $c_5+c_6=0$, respectively. Now, for the $(1,8)$-spot, \eqref{EQU:circle-3} implies that
\[
(c_1\cdot c_3+c_4\cdot c_6-1)\cdot \ul{b^{511}}=0
\quad\text{ or, using }c_1=c_3, c_4=c_6:\quad
c_1^2+c_4^2=1.
\]
Finally, we compute the Massey product set using \eqref{EQU:circle-4} and \eqref{EQU:circle-5} to be:
\[
\la {b^{2^1}}, {b^{2^2}}, \dots, {b^{2^9}} \ra=
\{c_1\cdot [b^{1023}]+c_4\cdot [\wt{b}^{1023}]\,\, |\,\, c_1^2+c_4^2=1\}.
\]
\end{example}

In contrast the case of formal algebras having no non-zero Massey products, when only $\mu_2$ and $\mu_3$ are non-zero (i.e., when $A$ is a minimal ``$A_3$ algebra''), this does not imply that the higher Massey products vanish, as is shown in the next example.
Manifolds which admit Sasakian and Vaisman structures are both known to have de Rham forms 
that are quasi-isomorphic to  $A_3$ algebras, c.f. \cite{BFMT} and \cite{SW}, respectively.

\begin{example}
Let $\ul{H}=R\langle \ul{b^2}, \ul{b^4}, \ul{b^6}, \ul{b^8}, \ul{b^{15}}, \ul{b^{16}}, \ul{b^{31}}\rangle$; c.f. \eqref{EQU:EXA:24816-degrees} in Example \ref{EXA:non-triv-R-C}). 
We set $\mu_k$ to be zero, except for
\[
\ul{\mu_2}(\ul{b^6},\ul{b^8}):=\ul{\mu_3}(\ul{b^2},\ul{b^4},\ul{b^8}):=\ul{b^{15}}
\quad\text{ and }\quad
\ul{\mu_3}(\ul{b^6},\ul{b^8},\ul{b^{16}}):=\ul{b^{31}}.
\]
Then $\mu^2=0$, just as in Example \ref{EXA:non-triv-R-C}. Now, for degree reasons, any defining system for $({b^2},{b^4},{b^8},{b^{16}})$ is of the form
\begin{equation*}
\begin{matrix}
  & 0 & 0 & \ul{b^{16}}  \\
 0 & 0  & \ul{b^8}   \\
c \cdot \ul{b^6} & \ul{b^4}    \\
 \ul{b^2} 
\end{matrix}
\end{equation*}
for some $c\in R$, subject to the condition $0=\ul{\mu_2}(c\cdot \ul{b^6},\ul{b^8})+\ul{\mu_3}(\ul{b^2},\ul{b^4},\ul{b^8})=(c+1)\cdot \ul{b^{15}}$. Thus, $c=-1$, and, since $\ul{\mu_3}(-\ul{b^6},\ul{b^8},\ul{b^{16}})=-\ul{b^{31}}$, we compute that the following quadruple Massey product is non-trivial (i.e., non-empty and does not contain zero):
\[
\la {b^2},{b^4},{b^8},{b^{16}}\ra=\{[-b^{31}]\}.
\]
\end{example}

\subsection{Lens spaces}

The next proposition, proved in the Appendix \ref{APP:L(p,q)-mimimal},
gives the $A_\infty$ minimal model of lens spaces $L(p,q)$ over $\Z_p$, for $p$ prime.

\begin{proposition}\label{PROP:HL(p,q)-minimal-model}
Let $L(p,q)$ be the lens space, where $p$ is a prime with $p\neq 2$ and $1\leq q<p$. A minimal model for $L(p,q)$ with $\Z_p$ coefficients is given by the following minimal $A_\infty$ algebra, with generators $e, x, y, z$, in degrees $|e|=0$, $|x|=1$, $|y|=2$, and $|z|=3$,
\[
H=H^\bu(L(p,q),\Z_p)\cong\Z_p  \langle e,x,y,z \rangle.
\]
The only non-trivial products $\mu_k$ on the generators are given by:
\begin{align*}
&\mu_2(e, e)=e,&& \mu_2(e, x)= \mu_2(x, e)=x,\\
&\mu_2(e, y)= \mu_2(y, e)=y,&& \mu_2(e, z)= \mu_2(z, e)=z, \\
&  \mu_2(x, y)= \mu_2(y, x)=z, &&  \mu_p(x, x,\dots,x)= y.
\end{align*}

\end{proposition}

 As a corollary, there are non-trivial Massey products of such lens spaces, in terms of the generators above.
 
\begin{theorem}
The lens spaces $L(p,q)$ have non-trivial $p$-fold Massey products
\[
\la c\cdot x,\dots,c\cdot x\ra=\{c\cdot y\}\subseteq H.
\]
for any constant $c\in \Z_p-\{0\}$.
\end{theorem}

\begin{proof}
Any defining system $\xbb_{1,p}$ is given by elements which are multiples of $x$, i.e., $x_{i,j}=c_{i,j}\cdot x$ for some $c_{i,j}\in \Z_p$, and $\mu_k(x,\dots,x)=0$ for $k=1,\dots, p-1$, whereas $\mu_p(c\cdot x,\dots,c\cdot x)=c^p\cdot \mu_p( x,\dots, x)=c^p\cdot y\equiv c\cdot y$, so that $\St\xbb{1,p}=c\cdot y$.
\end{proof}

As a corollary we derive a new proof of the classical result giving necessary conditions on a homotopy equivalence $L(p,q) \cong L(p,q')$.

\begin{corollary}
If there is a homotopy equivalence of lens spaces $L(p,q)\to L(p,q')$ then $q$ and $q'$ must satisfy 
\[
q\cdot q' \equiv \pm n^2  \mod p \quad \textrm{for some} \quad n\in \Z_p-\{0\}.
\]
\end{corollary}

\begin{proof}
A homology equivalence $f:L(p,q')\to L(p,q)$ induces a co-dga map on the singular chains $f_\bu:S_\bu(L(p,q'),\Z_p)\to S_\bu(L(p,q),\Z_p)$, which maps the fundamental class $[L(p,q')]$ to $\pm[L(p,q)]$. The simplicial chains (from the appendix) include as co-dga maps into singular chains $C_\bu(L(p,q),\Z_p)\hookrightarrow S_\bu(L(p,q),\Z_p)$, and same for $L(p,q')$.  Moreover, from the appendix, we know that the map $C_\bu(L(p,q),\Z_p)=C \to H$ from \eqref{EQU:C-H-H} maps the fundamental cycle $\sum_j \cs_j$ to $s\circ r(\sum_j \cs_j)=s(\dZ)=q\cdot \cZ$, and similarly for $q'$ instead of $q$.

We dualize the simplicial and singular chains and the maps from the previous paragraph over $\Z_p$, to get the lower row of the following diagram, where the dga structures on cohomologies and the vertical maps come from transferring to cohomology the structure on cochains
\ref{APP:L(p,q)-mimimal}:
\[\begin{tikzcd}
   H^\bu(L(p,q),\Z_p) \arrow[d, shift right]   \arrow[r, shift right]
& H^\bu(L(p,q),\Z_p)  \arrow[d, shift right]  \arrow[l, shift right]
& H^\bu(L(p,q'),\Z_p) \arrow[d, shift right]  \arrow[r, shift right]
& H^\bu(L(p,q'),\Z_p) \arrow[d, shift right] \arrow[l, shift right]
\\
   C^\bu(L(p,q),\Z_p) \arrow[u, shift right] 
& S^\bu(L(p,q),\Z_p)  \arrow[u, shift right]\arrow[r, "f^\bu"] \arrow[l]  
& S^\bu(L(p,q'),\Z_p)  \arrow[u, shift right] \arrow[r]  
& C^\bu(L(p,q'),\Z_p) \arrow[u, shift right] 
\end{tikzcd}\]

The top left horizontal map is induced by making the diagram commute with the bottom map $S^\bu(L(p,q),\Z_p)\to C^\bu(L(p,q),\Z_p)$, and then observing that this gives an $A_\infty$ quasi-isomorphism of $H^\bu(L(p,q),\Z_p)$ to itself, which is thus an isomorphism, and so is invertible. The analogous statement holds for $L(p,q')$  on the top right. In particular, we get an induced $A_\infty$ quasi-isomorphism, and thus isomorphism, from the top left to the top right $H^\bu(L(p,q),\Z_p)\to H^\bu(L(p,q'),\Z_p)$ which maps $z\mapsto\pm q \cdot q'{}^{-1}\cdot z$.

Now, any $A_\infty$ quasi-isomorphism $g:H^\bu(L(p,q),\Z_p)\to H^\bu(L(p,q'),\Z_p)$ has an isomorphism $g_1$ in degree $1$, and thus must map $g_1(x)=n_1\cdot x$, $g_1(y)=n_2\cdot y$, and $g_1(z)=n_3\cdot z$ for some $n_1,n_2,n_3\in \Z_p-\{0\}$. Since the Massey product set gives $g_1(\la x,\dots, x\ra)=\la g_1(x),\dots, g_1(x) \ra=\la n_1 x,\dots, n_1 x\ra=\{n_1 \cdot y\}$, but also  $g_1(\la x,\dots, x\ra)=g_1(\{y\})=\{g_1(y)\}=\{n_2\cdot y\}$, we see that $n_1=n_2$. Also, since $g_1$ commutes with $\mu_2$, we have $g_1(z)=g_1(\mu_2(x,y))=\mu_2(g_1(x),g_1(y))=\mu_2(n_1\cdot x,n_1\cdot y)=n_1^2\cdot z$, which means that $n_3=n_1^2$.

Applying this to the map $H^\bu(L(p,q),\Z_p)\to H^\bu(L(p,q'),\Z_p)$ from above, which maps $z\mapsto g_1(z)=\pm q \cdot q'{}^{-1}\cdot z$, we see that $\pm q \cdot q'{}^{-1}=n_1^2$ (in $\Z_p$). Setting $n=n_1\cdot q'$, we get the desired equation $q\cdot q'=\pm n^2$ in $\Z_p$.
\end{proof}

\begin{remark}
The examples $L(p,q)$ serve to illustrate interesting contrasts to the behavior Massey products over fields of characteristic zero. By a recent theorem of Milivojevi\'c, Stelzig, and Zoller \cite{MSZ}, for any non-zero degree map $f:X \to Y$ of $n$-dimension Poincar\'e duality spaces, if $X$ is formal (rationally, as a cdga) then $Y$ is formal as well. In particular, this applies to covering spaces, and the formality of $X$ implies the vanishing of all Massey products on $Y$ (rationally).

Here, in contrast to the case of characteristic zero, the universal coverings $S^3 \to L(p,q)$ have (rationally) formal domain $S^3$, yet $L(p,q)$ have a non-zero $p^{th}$ Massey product (over $\Z_p$).

As a second noteworthy point, here the $p^{th}$ Massey product clearly pulls back trivially to $S^3$ for all $p\geq 3$, 
which is in contrast to Taylor's theorem \cite{Ta}, showing that over the real numbers, non-zero triple Massey products pull back non-trivially along maps of non-zero degree.

\end{remark}

\section{Massey Products for $A_\infty$ modules over $A_\infty$ algebras}\label{SEC:Massey-modules}

The previous section can be repeated, mutatis mutandis, for modules over algebras. Since a primary focus in this paper is \emph{Massey inner products} of the next section, which involve morphisms of modules over algebras, we will only give the most basic definitions, indicate what the interested reader can check, and give one example from the Hopf fibration.

\subsection{Background on $A_\infty$ modules}\label{SEC:background-A-infty-module}

An \emph{$A_\infty$ module over $M$} over the $A_\infty$ algebra $A$ consists of a sequence of maps $\be_{k,\ell}:A^{\ot k}\ot M\ot A^{\ot \ell}\to M$ for $k,\ell\geq 0$, where the degree is $|\be_{k,\ell}|=1-k-\ell$, satisfying
\begin{multline}\label{EQU:ULbe_ULbe}
\sum_{k+\ell-1=r}\,\, \sum_{i+j+1=k} \ul{\be_{k,s}}\circ (\underbrace{id\otimes \dots \otimes id}_{i\text{ many}}  \otimes\ul{\mu_\ell}\otimes \underbrace{id\otimes \dots \otimes id}_{j\text{ many}}\ot id_{\ul M}\ot \underbrace{id\otimes \dots \otimes id}_{s\text{ many}})
\\
+\sum_{k+i=r}\,\, \sum_{\ell+j=s} \ul{\be_{i,j}}\circ (\underbrace{id\otimes \dots \otimes id}_{i\text{ many}}  \otimes\ul{\be_{k,\ell}}\otimes \underbrace{id\otimes \dots \otimes id}_{j\text{ many}})
\\
+\sum_{k+\ell-1=s}\,\, \sum_{i+j+1=k} \ul{\be_{r,k}}\circ (  \underbrace{id\otimes \dots \otimes id}_{r\text{ many}}\ot id_{\ul M}\ot\underbrace{id\otimes \dots \otimes id}_{i\text{ many}}  \otimes\ul{\mu_\ell}\otimes \underbrace{id\otimes \dots \otimes id}_{j\text{ many}})=0, \quad \forall r,s\geq 0
\end{multline}
Here $id=id_{\ul A}$, and $\ul{\be_{k,\ell}}:\ul{A}^{\ot k}\ot \ul{M}\ot \ul{A}^{\ot \ell}\to \ul{M}$ is the shifted map of degree $|\ul{\be_{k,\ell}}|=1$; c.f. Section \ref{SEC:background-A-infty} and \cite[Proposition 2.7]{T}. In short, we also write $\be^2=0$ for \eqref{EQU:ULbe_ULbe}. Moreover, we sometimes may write $M_{/A}$ if we want to emphasize that $M$ is over $A$.

\begin{examples}\label{EXA:A-infty-modules}\quad
\begin{enumerate}
\item
A dga $A$ is an $A_\infty$ algebra, and a dg bimodule $M$ over a dga $A$ is an $A_\infty$ module over $A$ where $\be_{0,0}$ is the differential of $M$, and $\be_{1,0}$ and $\be_{0,1}$ are the left- and right-module structures; c.f. \cite[Example 2.8]{T}.
\item\label{ITEM:A=A/A}
Any $A_\infty$ algebra $A$ is an $A_\infty$ module $A_{/A}=A$ over itself, where $\be_{k,\ell}=\mu_{k+\ell+1}$; c.f. \cite[Lemma 4.1]{T}.
\item\label{ITEM:M/B=M/A}
If $f:A\to B$ is a morphism of $A_\infty$ algebras, and $M=M_{/B}$ is a $A_\infty$ module over $B$, then $M$ has an induced $A_\infty$ module structure $M_{/A}$ over $A$ given by:
\[
\ul{\be_{k,\ell}{}_{/A}}=\sum_{k_1+\dots k_i=k} \sum_{\ell_1+\dots \ell_j=\ell} 
\ul{\be_{i,j}}\circ (\ul{f_{k_1}}\ot\dots\ot \ul{f_{k_i}}\ot id_{\ul{M}}\ot \ul{f_{\ell_1}}\ot\dots\ot \ul{f_{\ell_j}})
\]
In particular, for a dga map $f:A\to B$, any dg bimodule $M$ over $B$ is also a dg bimodule over $A$.
\end{enumerate}
\end{examples}

Let $(A,\mu)$ and $(B,\nu)$ be two $A_\infty$ algebras, and let $(M,\be)$ be an $A_\infty$ module over $A$, and let $(N,\ga)$ be an $A_\infty$ module over $B$. Let $f:A\to B$ be a morphism of $A_\infty$ algebras. A \emph{morphism of $A_\infty$ modules $F:M\to N$ of degree $|F|$ over $f$} consists of maps $F_{k,\ell}:A^{\ot k}\ot M\ot A^{\ot \ell}\to N$ of degree $|F_{k,\ell}|=|F|-k-\ell$, satisfying
\begin{multline}\label{EQU:F_F}
\sum_{k+\ell-1=r}\,\, \sum_{i+j+1=k} \ul{F_{k,s}}\circ (\underbrace{id\otimes \dots \otimes id}_{i\text{ many}}  \otimes\ul{\mu_\ell}\otimes \underbrace{id\otimes \dots \otimes id}_{j\text{ many}}\ot id_{\ul M}\ot \underbrace{id\otimes \dots \otimes id}_{s\text{ many}})
\\
+\sum_{k+i=r}\,\, \sum_{\ell+j=s} \ul{F_{i,j}}\circ (\underbrace{id\otimes \dots \otimes id}_{i\text{ many}}  \otimes\ul{\be_{k,\ell}}\otimes \underbrace{id\otimes \dots \otimes id}_{j\text{ many}})
\\
+\sum_{k+\ell-1=s}\,\, \sum_{i+j+1=k} \ul{F_{r,k}}\circ (  \underbrace{id\otimes \dots \otimes id}_{r\text{ many}}\ot id_{\ul M}\ot\underbrace{id\otimes \dots \otimes id}_{i\text{ many}}  \otimes\ul{\mu_\ell}\otimes \underbrace{id\otimes \dots \otimes id}_{j\text{ many}})
\\
=(-1)^{|F|}\cdot \sum_{k+k_1+\dots+k_i=r} \sum_{\ell+\ell_1+\dots+\ell_j=s} \ul{\ga_{i,j}}\circ ( {\ul{f_{k_1}}\otimes \dots \otimes \ul{f_{k_i}}} \ot\ul{F_{k,\ell}}\ot  {\ul{f_{\ell_1}}\otimes \dots \otimes \ul{f_{\ell_j}}})
, \quad \forall r,s\geq 0
\end{multline}
Similar to before, $\ul{F_{k,\ell}}:\ul{A}^{\ot k}\ot \ul{M}\ot \ul{A}^{\ot \ell}\to \ul{N}$ are the shifted maps of degree $|\ul{F_{k,\ell}}|=|F|$. We will also write $F_{/f}$ is we want to emphasize that $F$ is over $f$. We note that condition \eqref{EQU:F_F} is equivalent to saying that $F=\{F_{k,\ell}\}_{k,\ell}$ is a morphism of $A_\infty$ modules $F:M\to N_{/A}$ over $id_A$, where $N_{/A}$ is the induced $A_\infty$ module over $A$ from Example \ref{EXA:A-infty-modules}\eqref{ITEM:M/B=M/A}.

\begin{definition}\label{DEF:exact-Inf-Morph}
For any $A_\infty$ algebra $(A,\mu)$ and $A_\infty$ modules $(M,\be)$ and $(N,\ga)$ be  over $A$, the space of all families of maps $F=\{F_{k,\ell}:A^{\ot k}\ot M\ot A^{\ot \ell} \to N$ of degree $|F|-k-\ell\}_{k,\ell\geq 0}$ has a differential given by
\[
\delta(F)=F\circ \be-(-1)^{|F|}\ga\circ F.
\]
Explicitly, the $(r,s)$-component of $\delta(F)$ is given by 
\[
\ul{\delta(F)_{r,s}}= \textrm{(the left-hand side of \eqref{EQU:F_F})} - \textrm{(the right-hand side of \eqref{EQU:F_F})}. 
\]
\end{definition}

One checks that $\delta^2=0$. Note that closed elements are precisely morphisms of $A_\infty$ modules $F:M\to N$ over $id_A$.

\begin{examples}\label{EXA:Morph-of-Ainfty-modules}
\quad
\begin{enumerate}
\item
For a dga map $f:A\to B$, and two given dg bimodules $M$ over $A$ and $N$ over $B$, a dg bimodule $F:M\to N$ map over $f$ is a morphism of $A_\infty$ modules in the above sense.
\item\label{ITEM:f-gives-F}
If $f:A\to B$ is a map of $A_\infty$ algebras, then $A$ is an $A_\infty$ module over $A$, and $B$ is an $A_\infty$ module $B$ (by Example \ref{EXA:A-infty-modules}\eqref{ITEM:A=A/A}), and there is an induced morphism of $A_\infty$ modules  $F:A\to B$ (of degree $0$) over $f$ given by $F_{k,\ell}=f_{k+\ell+1}$.
\item\label{ITEM:f-G-N-N'-gives-GA-NA-N'A}
If $f:A\to B$ is a morphism of $A_\infty$ algebras, and $M$ and $N$ are $A_\infty$ modules over $B$, and $F:M\to N$ is a morphism of $A_\infty$ modules over $id_B$, then there is an induced morphism of $A_\infty$ modules $F_{/A}:M_{/A}\to N_{/A}$ over $id_A$ given by
\[
\ul{F_{k,\ell}{}_{/A}}=\sum_{k_1+\dots k_i=k} \sum_{\ell_1+\dots \ell_j=\ell} 
\ul{F_{i,j}}\circ (\ul{f_{k_1}}\ot\dots\ot \ul{f_{k_i}}\ot id_{\ul{M}}\ot \ul{f_{\ell_1}}\ot\dots\ot \ul{f_{\ell_j}})
\]
\end{enumerate}
\end{examples}

\subsection{Massey products for $A_\infty$ modules}\label{SEC:Massey-modules}

The notions from Section \ref{SEC:Def-system-Massey-A-infty} carry over to $A_\infty$ modules. By appropriately ``replacing one copy of $A$ by $M$'' one can define
\begin{enumerate}
\item Triangular system and staircase product (c.f. Definition \ref{defn:triangsys}), which is analogous to \eqref{EQU:x**} but uses the $A_\infty$ module structure  $\be$, i.e.,
\begin{equation*}
\ul{\St{\xbb}{p,q}}\\
=\sum_{\tiny\begin{matrix}{r+s\geq 1}\\{p=j_0\leq j_1< j_2<\dots<j_{r+s+1}=q}\end{matrix}}
\ul{\be_{r,s}}(\ul{a_{j_0,j_1}},\dots, \ul{m_{j_{r}+1,j_{r+1}}},\dots,\ul{a_{j_{r+s}+1,j_{r+s+1}}})
\end{equation*}
\item Defining system (c.f. Definition \ref{defn;defnsys})
\end{enumerate} 
We note that the defining system has one element of the module $M$ on the diagonal, and then elements of $M$ on its ``cone'' to the left and up, while all other elements are in $A$.
\[
\begin{matrix}
             & \ceco m_{2,6} & \ceco m_{3,6} & \ceco m_{4,6} & a_{5,6} & a_{6,6} \\
\ceco m_{1,5} & \ceco m_{2,5} & \ceco m_{3,5} & \ceco m_{4,5} & a_{5,5} \\
\ceco m_{1,4} & \ceco m_{2,4} & \ceco m_{3,4} & \ceco m_{4,4} \\
 a_{1,3} & a_{2,3}  & a_{3,3}  \\
 a_{1,2} & a_{2,2} \\
 a_{1,1}
\end{matrix}
\]
The proofs from Section \ref{SEC:Def-system-Massey-A-infty} then establish the Massey product set with one module element $\la a,...,a,m,a,...,a\ra$, 
and the fact that the is a well defined Massey product set map 
\[
H(A)\times\dots \times H(A)\times H(M)\times H(A)\times\dots \times H(A)\to \mathcal P(H(M)).
\]

One may also verify the analogue of Proposition \ref{prop;MPnatural}:
\begin{proposition} (Naturality of module Massey products)
Let $f:A\to B$ be a morphism of $A_\infty$ algebras, and $F:M\to N$ a morphism of $A_\infty$ modules over $f$. If $\xbb$ is a defining system in $M$, then there is a defining system $\{y_{\bu\bu}\}_{1,n}=F(\{x_{\bu\bu}\}_{1,n})$ in $N$ defined similarly to \eqref{EQU:y**=f(x**)}, where we apply either $F_{i,j}$ or $f_k$ depending on the output being in $N$ or $A$. Then, the maps $(f_1)_*:H(A)\to H(B)$ and  $(F_{0,0})_*:H(M)\to H(N)$ on cohomology restricts to a set map on the Massey products satisfying:
\[
(F_{0,0})_*\big( \la[a_{1,1}],\dots,[m_{\ell,\ell}],\dots,[a_{n,n}]\ra \big) \subseteq  \la[f_1(a_{1,1})],\dots,[F_{0,0}(m_{\ell,\ell})],\dots, [f_1(a_{n,n})] \ra.
\]

Therefore, $A$-infinity quasi-isomorphisms in both directions yields bijective Massey product sets.
\end{proposition}

We give one clarifying example from the Hopf fibration:
\begin{example}
Assume $R$ is a field of characteristic $0$. The Hopf map $f:S^3\to S^2$ induces a map of de Rham forms $f^*:\Om(S^2)\to \Om(S^3)$. The minimal models (in the sense of rational homotopy theory) are
\begin{align*}
\mathcal{M}_{S^2} &= \bigwedge \langle x,y\rangle, \quad |x|=2, |y|=3, dy = x^2 \\
\mathcal{M}_{S^3} &= \bigwedge \langle z\rangle, \quad |z|=3, 
\end{align*}
and induce a map, also denoted by $f^* : \mathcal{M}_{S^2} \to \mathcal{M}_{S^3} $, satisfying
$f^*(y)=z$ since $f$ represents the generator for $\pi_3(S^2)$.
The map $f^*$ makes $\mathcal{M}_{S^3}$ into an $\mathcal{M}_{S^2}$-module via
\[
a \bullet m = f^*(a) \cdot m.
\]

If we denote by $1\in \mc M_{S^3}$ the unit as an element in the dg module over the dga $\mc M_{S^2}$, then there is a non-trivial Massey product with respect to this module structure, given by
\[
\la x,x,1\ra=\{z\}\subseteq H^\bu(S^3),
\]
as the only choice for defining system is $dy=x^2$ and $d(0) = x \bullet 1 =0$. 

In particular, this non-trivial Massey product shows that the module structure of 
$\Om(S^3)$ over $\Om(S^2)$ is not equivalent to a dg-module and algebra structure with zero differentials, even though both spheres are formal as spaces. 
\end{example}

\section{Massey Products for $A_\infty$ inner products}\label{SEC:Massey-inner-product}

In this section we review $A_\infty$ inner products and introduce a notion of ``Massey products'', establishing all the desirable properties as in the previous discussions of  $A_\infty$ algebras and modules in the earlier sections. 
Briefly, this requires a new notion of ``cyclic defining system'', the ``cyclic straircase product'', and a few novel arguments to establish properties such as naturality with respect to morphisms and pullback of inner products. For the reader who is mainly interested in commutative dgas, these are spelled out explicitly and more simply in the subsequent section.

To begin, we view an inner product as a morphism from a module to its dual module. 

\subsection{Background on $A_\infty$ inner products}\label{SEC:background-A-infty-IP}
If $(M,\beta)$ is an $A_\infty$ module over the $A_\infty$ algebra $A$, then the dual $M^*=Hom(M,R)$ is also an $A_\infty$ module over $A$ whose module structure is given by $\be'_{k,\ell}:A^{\ot k}\ot M^*\ot A^{\ot \ell}\to M^*$,
\begin{equation}\label{EQU:dual-module-structure}
\ul{\be'_{k,\ell}}(\ul{a_1},\dots, \ul{a_k},\ul{m^*},\ul{\wt a_1},\dots, \ul{\wt a_\ell})(\ul{\wt m})=
-(-1)^{\varepsilon}\cdot 
\ul{m^*}\big(\ul{\be_{\ell,k}}(\ul{\wt a_1},\dots, \ul{\wt a_\ell},\ul{\wt m},\ul{a_1},\dots, \ul{a_k})\big)
\end{equation}
for $a_i,\wt a_i\in A$, $m^*\in M^*$, $\wt m\in M$, and the sign $(-1)^\varepsilon=(-1)^{(|\ul{a_1}|+\dots+| \ul{a_k}|)\cdot (|\ul{m^*}|+|\ul{\wt a_1}|+\dots+| \ul{\wt a_\ell}|+|\ul{\wt m}|)+|\ul{m^*}|}$ is given by the Koszul sign rule (the extra minus is needed for $\be^2=0$ to imply $\be'^2=0$).

With this, we define an \emph{$A_\infty$ inner product on $M$} (also called a \emph{homotopy inner product} or \emph{infinity inner product}) to be a morphism $I:M\to M^*$ of $A_\infty$ modules from $M$ to $M^*$ over $id_A$. By Equations \eqref{EQU:F_F} and \eqref{EQU:dual-module-structure}, this means that the functions $I_{k,\ell}:A^{\ot k}\ot M\ot A^{\ot \ell}\to M^*$ satisfy
\begin{align}\label{EQU:F=inner-product}
&\sum_{k+\ell-1=r}\,\, \sum_{i+j+1=k}(-1)^{\varepsilon_1}\cdot \ul{I_{k,s}}\circ (\underbrace{\ul{a_1}, \dots }_{i\text{ many}}  ,\ul{\mu_\ell}(\ul{a_{i+1}},\dots), \underbrace{ \dots,\ul{a_r}}_{j\text{ many}},\ul{m},\underbrace{\ul{\wt a_1},\dots,\ul{\wt a_s}}_{s\text{ many}})(\ul{\wt m})
\\ \nonumber
&+\sum_{k+i=r}\,\, \sum_{\ell+j=s}(-1)^{\varepsilon_2}\cdot \ul{I_{i,j}}\circ (\underbrace{\ul{a_1}, \dots}_{i\text{ many}}  ,\ul{\be_{k,\ell}}(\ul{a_{i+1}},\dots,\ul{a_r},\ul{m},\ul{\wt a_1},\dots, \ul{\wt a_{s-j}}), \underbrace{\dots,\ul{\wt a_s}}_{j\text{ many}})(\ul{\wt m})
\\ \nonumber
&+\sum_{k+\ell-1=s}\,\, \sum_{i+j+1=k} (-1)^{\varepsilon_3}\cdot\ul{I_{r,k}}\circ (  \underbrace{\ul{a_1},\dots,\ul{a_r}}_{r\text{ many}},\ul{m},\underbrace{\ul{\wt a_1},\dots}_{i\text{ many}},\ul{\mu_\ell}(\ul{\wt a_{i+1}},\dots), \underbrace{ \dots,\ul{\wt a_s}}_{j\text{ many}})(\ul{\wt m})
\\ \nonumber
&+\sum_{k+i=r} \sum_{\ell+j=s} (-1)^{\varepsilon_4}\cdot \ul{I_{k,\ell}}(\ul{a_{i+1}},\dots,\ul{a_r},\ul{m},\ul{\wt a_1},\dots, \ul{\wt a_{s-j}})
\Big(\ul{\be_{j,i}}(\underbrace{\dots,\ul{\wt a_s}}_{j\text{ many}},\ul{\wt m}, \underbrace{\ul{a_1}, \dots}_{i\text{ many}})\Big)
=0
\end{align}
for any $a_i,\wt a_i\in A$, $m,\wt m\in M$, where the signs $(-1)^{\varepsilon_j}$ are the Koszul signs when applying $(\ul{a_1},\dots,\ul{a_r}, \ul{m},\ul{\wt{a}_1},\dots,\ul{\wt{a}_s}, \ul{\wt m})$ to \eqref{EQU:F_F}.

\begin{examples}\label{EXA:cdga=>IP}
\quad
\begin{enumerate}
\item\label{ITM:cdga-example}
Our main examples will be of the following form. Let $A$ be a \emph{graded commutative} dga (and take $M=A$ as the dg bimodule over itself), and fix a closed element $x\in A^*$ in the dual of $A$. Then, we have the ($A_\infty$-)inner product $I_x:A\to A^*$ given by $(I_x(a))(\wt{a})=x(a\cdot \wt{a})$ for any $a,\wt{a}\in A$.
\item A generalization of the previous example is given by a \emph{cyclic} $A_\infty$ algebra; see, e.g.,  \cite{Kon93, Kaj07, KS09}. If $A$ is an $A_\infty$ algebra and $\langle , \rangle:A\ot A\to R$ is a graded symmetric map, then $(A,\langle , \rangle)$ is a cyclic $A_\infty$ algebra if and only if the map $I_{0,0}:A\to A^*$ given by $\langle , \rangle$ together with $I_{k,\ell}=0$ for $k+\ell>0$ is an $A_\infty$ inner product on $A$.
\end{enumerate}
\end{examples}

Now, let $A$ be an $A_\infty$ algebra and $M$ be an $A_\infty$ module over $A$. Moreover, let $B$ be an $A_\infty$ algebra, $N$ an $A_\infty$ module over $B$, and let $I:N\to N^*$ be an $A_\infty$ inner product on $N$, which we want to pull back to $M$. If $f:A\to B$ is a morphism of $A_\infty$ algebras, and $F:M\to N$ is a morphism of $A_\infty$ modules over $f$, then there is an induced $A_\infty$ inner product $F^\sharp(I):M\to M^*$ on $M$ given by
\begin{multline}\label{EQU:Fsharp(G)}
\ul{F^\sharp(I)_{k,\ell}}(\ul{a_1},\dots, \ul{a_k},\ul{m},\ul{\wt a_1},\dots, \ul{\wt a_\ell})(\ul{\wt m})=
\sum_{\tiny\begin{matrix}0\leq p\leq k\\ 0\leq q\leq \ell\end{matrix}} \sum_{\tiny\begin{matrix}r+r_1+\dots r_i=p\\s+s_1+\dots+s_j=q\end{matrix}} (-1)^\varepsilon
\\
\cdot\Big(\underbrace{\ul{I_{i,j}}\circ (\ul{f_{r_1}}\ot \dots \ot \ul{f_{r_i}}\ot \ul{F_{r,s}}\ot \ul{f_{s_1}}\ot \dots \ot \ul{f_{s_j}})(\ul{a_{k-p+1}},\dots, \ul{a_k},\ul{m},\ul{\wt a_1},\dots, \ul{\wt a_{q}})}_{\in \ul{N}^*}\Big)
\\
\Big(\underbrace{\ul{F_{\ell-q,k-p}} (\ul{\wt a_{q+1}},\dots, \ul{\wt a_\ell},\ul{\wt m},\ul{a_1},\dots, \ul{a_{k-p}})}_{\in \ul{N}}\Big)\quad \in R
\end{multline}
Here, the sign $(-1)^\varepsilon$ is given by the Koszul sign. An alternative description of $F^\sharp(I)$ is given by noting that it is the composition of morphisms $F^\sharp(I)=F^*\circ I_{/A}\circ F:M\to N_{/A}\to N^*_{/A}\to M^*$ over $id_A$. We also note that Equation \eqref{EQU:Fsharp(G)} may be pictured as follows:
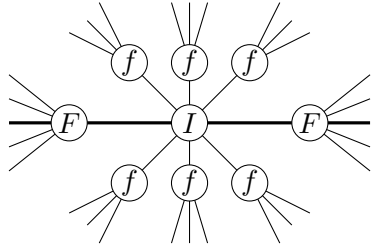
\[
\begin{tikzpicture}[scale=.8]
\draw (-1.7,-1.7) -- (1.7,1.7); \draw (-1.7,1.7) -- (1.7,-1.7); \draw (0,2) -- (0,-2);
\draw (-.3,2) --(0,1) -- (0.3,2); \draw (-.3,-2) --(0,-1) -- (0.3,-2);
\draw (2,1.5) --(1,1) -- (1.5,2); \draw (2,-1.5) --(1,-1) -- (1.5,-2); 
\draw (-2,1.5) --(-1,1) -- (-1.5,2); \draw (-2,-1.5) --(-1,-1) -- (-1.5,-2); 
\draw (3,0.4) -- (2,0) -- (3,-0.4); \draw (3,0.8) -- (2,0) -- (3,-0.8); 
\draw (-3,0.4) -- (-2,0) -- (-3,-0.4); \draw (-3,0.8) -- (-2,0) -- (-3,-0.8); 
 \draw [very thick] (-3,0) -- (3,0);
 \fcirc{(0,0)}{I}  \fcirc{(-2,0)}{F}  \fcirc{(2,0)}{F}
 \fcirc{(0,1)}{f}  \fcirc{(-1,1)}{f}  \fcirc{(1,1)}{f}  \fcirc{(0,-1)}{f}  \fcirc{(-1,-1)}{f}  \fcirc{(1,-1)}{f} 
\end{tikzpicture}
\]
\begin{examples}
\quad
\begin{enumerate}
\item
If $f:A\to B$ is a map of graded commutative dgas, and we fix a closed element $y\in B^*$, then we have the inner product $I_y:A\to A^*$ given by $(I_y(a))(\wt{a}):=y(f(a)\cdot f(\wt{a}))$.
\item
If $f:A\to B$ is a map of $A_\infty$ algebras and we have an $A_\infty$ inner product $I:B\to B^*$ on $B$, then the induced morphism of $A_\infty$ modules $F:A\to B$ over $f$ from Example \ref{EXA:Morph-of-Ainfty-modules}\eqref{ITEM:f-gives-F}, yields and $A_\infty$ inner product $F^\sharp(I):A\to A^*$ on $A$.
\end{enumerate}
\end{examples}

\subsection{Massey products for $A_\infty$ inner products}

We now give an appropriate notion of triangular system, staircase product and defining system for $A_\infty$ inner products. Let $A$ be an $A_\infty$ algebra and $M$ an $A_\infty$ module with an $A_\infty$ inner product $I:M\to M^*$.

\begin{definition}
Fix two non-negative integers $k,\ell\geq 0$.
\begin{enumerate}
\item
A \emph{cyclic triangular system} for the $k+\ell+2$ elements
\begin{equation}\label{EQU:a11-mk+l+2k+l+2}
(\underbrace{a_{1,1},\dots, a_{k,k}}_{k\text{ many}},m_{k+1,k+1},\underbrace{a_{k+2,k+2},\dots,a_{k+\ell+1,k+\ell+1}}_{\ell\text{ many}}, m_{k+\ell+2,k+\ell+2}),
\end{equation}
 where $m_{j,j}\in M$ and $x_{j,j}\in A$, consists of two triangular systems
\begin{align*}
&\xbb \text{ for }(\underbrace{a_{1,1},\dots, a_{k,k}}_{k\text{ many}},m_{k+1,k+1},\underbrace{a_{k+2,k+2},\dots,a_{k+\ell+1,k+\ell+1}}_{\ell\text{ many}})
\\
\text{and }&\xbp \text{ for }(\underbrace{a_{k+2,k+2},\dots,a_{k+\ell+1,k+\ell+1}}_{\ell\text{ many}},m_{k+\ell+2,k+\ell+2},\underbrace{a_{1,1},\dots, a_{k,k}}_{k\text{ many}})
\end{align*}
both of which \emph{include} the top left corner points, and which coincide on their intersection triangular points.

It is convenient to index the elements in $\xbp$ in a cyclic way, so that, for example, in the upper left corner of $m_{k+\ell+2,k+\ell+2}$ and $a_{1,1}$, we have an element indexed by $m_{k+\ell+2,1}$. We picture the defining systems by placing $\xbb$ below $\xbp$, where their common $a_{i,j}$ elements overlap. For example, a cyclic triangular system for $(a_{1,1},a_{2,2},m_{3,3},a_{4,4},a_{5,5},a_{6,6},m_{7,7})$ can be written as follows:
\begin{equation}\label{EQU:1-7-triangular-system}
\begin{matrix}
&&& \ceco m_{4,2} & \ceco m_{5,2} &  \ceco m_{6,2} & \ceco m_{7,2} & a_{1,2} & a_{2,2} \\
&&& \ceco m_{4,1} & \ceco m_{5,1} &  \ceco m_{6,1} & \ceco m_{7,1} & a_{1,1} \\
&&&\ceco m_{4,7} &\ceco m_{5,7} &  \ceco m_{6,7}  & \ceco m_{7,7}  \\ 
 \ceco m_{1,6} &  \ceco m_{2,6} & \ceco  m_{3,6} &  a_{4,6} & a_{5,6} & a_{6,6}   \\
 \ceco m_{1,5} &  \ceco m_{2,6} & \ceco m_{3,5} &  a_{3,5} &  a_{5,5}&&& \xbp \\
 \ceco m_{1,4} & \ceco m_{2,4} & \ceco m_{3,4} & a_{4,4} \\
 \ceco m_{1,3} & \ceco m_{2,3} &  \ceco m_{3,3}\\
 a_{1,2}  &  a_{2,2} \\
 a_{1,1}   &&& \xbb
\end{matrix}
\end{equation}
Here the lower $a_{i,j}$ of $\xbb$ are also identified with the upper $a_{i,j}$ of $\xbp$, so that this could be thought of as a cyclic arrangement on a cylinder, where the diagonal elements $x_{j,j}$ are placed at the base of the cylinder.

\item
A cyclic triangular system is called a \emph{cyclic defining system} if it satisfies
\begin{equation}\label{EQU:dx=-Sx,dx'=-Sx'}
d(x_{i,j})=- \St\xbb{i,j}\quad\text{and}\quad d(x'_{i,j})=- \St\xbp{i,j}\quad\text{for all $i,j$},
\end{equation}
where $\mc S^{i,j}$ is the staircase product from Section \ref{SEC:Massey-modules}, which uses both the $A_\infty$ algebra structure and the $A_\infty$ module structure. We note that the above equations have to be true for all $i,j$, including the top left corner points of the triangular systems.

\item
Each cyclic defining system $\xbb,\xbp$ associated to elements $(a_{1,1},\dots, m_{k+\ell+2,k+\ell+2})$ as in \eqref{EQU:a11-mk+l+2k+l+2} determines an element in $R$ via the \emph{staircase inner product}, which is defined to be the element $\mc I =\StIP{\xbb,\xbp} \in R$ given by
\begin{multline}\label{EQU:F(x**,x'**)}
\mc I=
\sum_{\tiny\begin{matrix} r\geq 0\\ 1\leq i_0\leq i_1< \dots<i_r<k+1\end{matrix}} \quad
\sum_{\tiny\begin{matrix} s\geq 0\\ k+1\leq j_0< j_1<\dots<j_s<k+\ell+2\end{matrix}} \quad
\\
(-1)^\varepsilon \cdot \ul{I_{r,s}}\Big(
\underbrace{\ul{x_{i_0,i_1}},\dots, \ul{x_{i_{r-1}+1,i_r}}}_{r \text{ elements in }\ul A},
\underbrace{\ul{x_{i_r+1,j_0}}}_{\in \ul{M}},
\underbrace{\ul{x_{j_0+1,j_1}},\dots, \ul{x_{j_{s-1}+1,j_s}}}_{s \text{ elements in }\ul A}
\Big)\big(\underbrace{\ul{x'_{j_{s}+1,i_0-1}}}_{\in \ul{M}}\big),
\end{multline}
where $(-1)^\varepsilon$ is the Koszul sign coming from moving $\ul{a_{1,1}},\dots, \ul{a_{i_0-1,i_0-1}}$ to the end. 
If we wish to emphasize the dependence on the $A_\infty$ inner product $I$ we'll write $\mc I_I$.

For $r=0$ the first sum is understood to be vacuous, while for $s=0$ the second sum is understood to be vacuous, and in the last element, if $i_0=1$, then $i_0-1$ is understood as $k+\ell+2$. 

The above sum can be thought of as summing over ``cyclic staircases'', and all calculations could as well be done in a graphical way on a cylinder.
\[
\begin{tikzpicture}
\draw (0,0) ellipse (6 and .5);
\foreach \n in {1,...,11}{
\fill ({6*cos(18*\n)},{-.5*sin(18*\n)}) circle (.05);
\draw ({6.45*cos(18*\n)},{-.8*sin(18*\n)}) node {\scalebox{0.8}{$a_{\n,\n}$}};}
\fill ({6*cos(18*12)},{-.5*sin(18*12)}) circle (.05);
\draw ({6.45*cos(18*12)},{-.8*sin(18*12)}) node {\scalebox{0.8}{$m_{12,12}$}};
\foreach \n in {13,...,19}{
\fill ({6*cos(18*\n)},{-.5*sin(18*\n)}) circle (.05);
\draw ({6.45*cos(18*\n)},{-.8*sin(18*\n)}) node {\scalebox{0.8}{$a_{\n,\n}$}};}
\fill ({6*cos(18*20)},{-.5*sin(18*20)}) circle (.05);
\draw ({6.45*cos(18*20)},{-.8*sin(18*20)}) node {\scalebox{0.8}{$m_{20,20}$}};

\draw [thick] ({6*cos(18*3)},{-.5*sin(18*3)})  -- ({6*cos(18*20)},{2-.5*sin(18*20)}) -- ({6*cos(18*18)},{-.5*sin(18*18)});  \draw ({.6+6*cos(18*20)},{2.1-.5*sin(18*20)}) node {\scalebox{0.8}{$m_{18,3}$}};
\fill ({6.1*cos(18*20)},{2.1-.5*sin(18*20)})  circle (.09);

\draw [thick] ({6*cos(18*4)},{-.5*sin(18*4)})  -- ({6*cos(18*6)},{2.7-.5*sin(18*6)}) -- ({6*cos(18*8)},{-.5*sin(18*8)});  \draw ({.5+6*cos(18*6)},{2.7-.5*sin(18*6)}) node {\scalebox{0.8}{$a_{4,8}$}};
\fill ({6*cos(18*6)},{2.7-.5*sin(18*6)})  circle (.09);

\draw [thick] ({6*cos(18*9)},{-.5*sin(18*9)})  -- ({6*cos(18*10)},{1.5-.5*sin(18*10)}) -- ({6*cos(18*13)},{-.5*sin(18*13)});  \draw ({.6+6*cos(18*10)},{1.5-.5*sin(18*10)})  node {\scalebox{0.8}{$m_{9,13}$}};
\fill ({6*cos(18*10)},{1.5-.5*sin(18*10)})  circle (.09);

\fill ({6*cos(18*14)},{-.5*sin(18*14)})  circle (.09);

\draw [thick] ({6*cos(18*15)},{-.5*sin(18*15)})  -- ({6*cos(18*16)},{1-.5*sin(18*16)}) -- ({6*cos(18*17)},{-.5*sin(18*17)});  \draw ({.6+6*cos(18*16)},{1-.5*sin(18*16)})  node {\scalebox{0.8}{$a_{15,17}$}};
\fill  ({6*cos(18*16)},{1-.5*sin(18*16)})  circle (.09);

\end{tikzpicture}
\]
The diagram above represents the term $I_{1, 2}(a_{4,8}, m_{9,13}, a_{14, 14}, a_{15, 17})(m_{18,3})$ in a staircase inner product associated to  elements $(a_{1,1}, \dots, a_{11,11}, m_{12, 12}, a_{13, 13}, \dots, a_{19,19}, m_{20,20})$.
For the benefit of clarity, we will stick to the algebraic expressions.
\end{enumerate}
\end{definition}
\begin{example}
For example, when $k=\ell=1$, the staircase inner product for the elements $(a_{1,1},m_{2,2},a_{3,3},m_{4,4})$ with the cyclic defining system
\begin{equation}\label{EQU:11-staircaseIP}
\begin{matrix}
&  & \ceco m_{3,1} & \ceco m_{4,1} &  a_{1,1}\\
&  & \ceco m_{3,4} &\ceco m_{4,4} \\
 \ceco m_{1,3} & \ceco m_{2,3} &   a_{3,3}&&\xbp\\
 \ceco m_{1,2}  & \ceco m_{2,2} \\
 a_{1,1}  && \xbb
\end{matrix},
\end{equation}
is given by the following sum
\begin{align}\label{EQU:staircase-product-for-11}
\mc I =&  
\ul{I_{0,0}}(\ul{m_{1,3}})(\ul{m_{4,4}})+(-1)^\varepsilon \ul{I_{0,0}}(\ul{m_{2,2}})(\ul{m_{3,1}})
+\ul{I_{0,0}}(\ul{m_{1,2}})(\ul{m_{3,4}})+(-1)^\varepsilon \ul{I_{0,0}}(\ul{m_{2,3}})(\ul{m_{4,1}})
\\ \nonumber
&+\ul{I_{1,0}}(\ul{a_{1,1}},\ul{m_{2,2}})(\ul{m_{3,4}})+\ul{I_{1,0}}(\ul{a_{1,1}},\ul{m_{2,3}})(\ul{m_{4,4}})
\\ \nonumber
&+\ul{I_{0,1}}(\ul{m_{1,2}},\ul{a_{3,3}})(\ul{m_{4,4}})+(-1)^\varepsilon \ul{I_{0,1}}(\ul{m_{2,2}},\ul{a_{3,3}})(\ul{m_{4,1}})
\\ \nonumber
&+\ul{I_{1,1}}(\ul{a_{1,1}},\ul{m_{2,2}},\ul{a_{3,3}})(\ul{m_{4,4}})
\end{align}
 where $(-1)^\varepsilon=(-1)^{|\ul{a_{1,1}}|\cdot (|\ul{m_{2,2}}|+|\ul{a_{3,3}}|+|\ul{m_{4,4}}|)}$.
\end{example}
We then define the Massey inner product analogously to the Massey product in Section \ref{SEC:Massey-A-infty}.
\begin{definition}
For $(x_{1,1},\dots,x_{k+\ell+2,k+\ell+2})\in A^{\times k}\times M\times A^{\times \ell}\times M$, $k+\ell \geq 1$, we define the \emph{$(k,\ell)^{th}$ Massey inner product} (or \emph{$(k,\ell)^{th}$ Massey inner product set}) to be the subset of the ground ring $R$ given by 
\begin{equation*} 
\la x_{1,1},\dots,x_{k+\ell+2,x+\ell+2} \ra_I 
 :=\Big\{ \StIP{\xbb,\xbp} \Big| \xbb, \xbp \text{ is a cyclic defining system} \Big\} \subseteq R.
\end{equation*}
We say that the Massey inner product is trivial, if $\la x_{1,1},\dots,x_{k+\ell+2,k+\ell+2} \ra_I$ is either the empty set or contains $0\in R$.
\end{definition}

\begin{lemma} \label{lem;MIPwelldef}
Just as for the Massey product, this induces a well-defined map from the cohomologies of $A$ and $M$ respectively to the power set of $R$:
\begin{align}\label{EQU:Massey-IP-DEF}
H(A)\times\dots \times H(A)\times H(M)\times H(A)\times \dots \times H(A)\times H(M)\to \mathcal P(R)
\\ \nonumber
\la [x_{1,1}],\dots,[x_{k+\ell+2,k+\ell+2}]\ra_I:= \la x_{1,1},\dots,x_{k+\ell+2,x+\ell+2} \ra_I
\end{align}
\end{lemma}

\begin{proof}[Proof of \eqref{EQU:Massey-IP-DEF}]
If we add $dz$ to a term $x_{p,q}$ (or $x'_{p,q}$) in a cyclic defining system $\xbb, \xbp$ then by Lemma \ref{LEM:Zs} we can add terms $Z_{i,j}$ to $\xbb, \xbp$ to get a cyclic new defining system $\xbbt, \xbpt$ (c.f. Equations \eqref{EQU:xtilde=x+Z} and \eqref{EQU:dxtilde=-sxtilde}). Note that $\xbb, \xbp$ and $\xbbt, \xbpt$ only differ in the cone to the left and up of the given $(p,q)$-spot (--for example adding $dz$ to $a_{3,3}$ in \eqref{EQU:11-staircaseIP} will add terms $Z_{i,j}$ to $m_{1,3}, m_{2,3}, m_{3,4}, m_{3,1}$). We will show that 
\begin{equation}\label{EQU:F(x,x')=F(wt-x,wt-x')}
\StIP{\xbb,\xbp}=\StIP{\xbbt,\xbpt}
\end{equation}
This shows that the Massey inner product is well-defined on cohomology.

The argument is similar to the one in \eqref{EQU:S(xtilde)=...}. First, we observe that $\StIP{\xbbt,\xbpt}$ is given by a sum over all cyclic staircases from \eqref{EQU:F(x**,x'**)}, which uses either elements purely from $\StIP{\xbb,\xbp}$ or include at most one of the terms $dz$ or $Z_{i,j}$. We need to check that this sum of terms with $dz$ or $Z_{i,j}$ is zero. This follows, because this sum is precisely the left-hand side of the condition \eqref{EQU:F=inner-product} of the morphism $I$ of $A_\infty$ modules applied to all cyclic staircases of a new system, which is given by adding $z$ at the $(p,q)$-spot to $\xbb,\xbp$. Indeed, applying \eqref{EQU:F=inner-product} to cyclic staircases of ``$z$ added to $\xbb,\xbp$ at the $(p,q)$-spot'' gives zero whenever $z$ is not inside a product $\mu_r$ or $\beta_{r,s}$ because both $\xbb$ and $\xbp$ are defining systems, and so, for example, a $\mu_1(x_{i,j})$ will cancel with the $\mu_r(\dots)$ over ``lower staircases'' spanning over $i,j$:
\[
\sum \pm I_{i,j}\big(\dots, z, \dots, \mu_r(\dots),\dots\big)(\dots)=0,\quad \text{etc.}
\]
The only remaining terms are where $z$ is inside a $\mu_r$ or $\beta_{r,s}$ (compare again with \eqref{EQU:S(xtilde)=...}):
\[
\sum \pm I_{i,j}\big(\dots, \mu_r(\dots,z,\dots),\dots\big)(\dots)=0,\quad \text{etc.}
\]
This is precisely the sum with one $\mu_1(z)$ or a $\mu_r(\dots, z,\dots)$ which is a $Z_{u,v}$. As argued above, this is also equal to the difference of $\StIP{\xbbt,\xbpt}-\StIP{\xbb,\xbp}$. By \eqref{EQU:F=inner-product} the sum over all these vanishes, and we get the claim.
\end{proof}

The next definition and lemma show that an $A_\infty$ inner product that is \emph{exact}, in an appropriate complex of maps for  given $A_\infty$ modules, always has trivial Massey inner products. The definition will also be used below to define morphisms of $A_\infty$ algebras with $A_\infty$ inner products, c.f. Definition \ref{DEF:morphism-of-Ainf+AIP}.

\begin{definition}\label{DEF:exact-HIP}
An $A_\infty$ algebra with $A_\infty$ inner product $I:M\to M^*$, for 
 $A_\infty$ module $M$ over $A$, is called \emph{exact} if $I = \delta(G)$, for some family of maps 
 $G=\{G_{k,\ell} \} :A^{\ot k}\ot M\ot A^{\ot \ell} \to M^*$, c.f. Definition \ref{DEF:exact-Inf-Morph}.
\end{definition}

\begin{lemma}\label{LEM:exactF=>0MIP}
Let $A$ be an $A_\infty$ algebra and let $M$ be an $A_\infty$ module over $A$. Then any Massey inner product associated to an \emph{exact} $A_\infty$ inner product $I:M\to M^*$ is either the empty set or the one-element set $\{0\}$.
\end{lemma}
\begin{proof}
Assuming that a Massey inner product is not empty, let $\xbb, \xbp$ be a cyclic defining system. We plug this into the expression \eqref{EQU:F(x**,x'**)} for the Massey inner product for $I=\delta(G)$, which means we take $\delta(G)$ applied to all ``cyclic staircases''. Now $\delta(G)$ applies exactly one differential $\mu_r$ (with $r\geq 1$) or $\be_{r,s}$ (with $r,s\geq 0$) to these $x_{i,j}$ (c.f. \eqref{EQU:F=inner-product}), which all cancel due to the conditions $d(x_{i,j})+\St\xbb{i,j}=0$ and $d(x'_{i,j})+ \St\xbp{i,j}=0$ from \eqref{EQU:dx=-Sx,dx'=-Sx'} of the cyclic defining system $\xbb, \xbp$. Thus, the staircase inner product is $0$.
\end{proof}

\subsection{Morphisms of $A_\infty$ algebras with $A_\infty$ modules and $A_\infty$ inner products}

We now establish functoriality properties of Massey inner product sets.

\begin{definition}\label{DEF:morphism-of-Ainf+AIP}
For two $A_\infty$ algebras with $A_\infty$ modules and $A_\infty$ inner products $(A,M,I:M\to M^*)$ and $(B,N,J:N\to N^*)$, a \emph{morphism} $(f,F):(A,M,I)\to (B,N,J)$ consists of a morphism of $A_\infty$ algebras $f:A\to B$ and a morphism of $A_\infty$ module $F:M\to N$ over $f$, such that the difference $F^\sharp(J)-I$ is an exact $A_\infty$ inner product (see Definition \ref{DEF:exact-HIP}).
\end{definition}

We now check naturality of Massey inner products under these morphism.
\begin{proposition}\label{PROP:MIP-under-morphism}
Let $(f,F):(A,M,I)\to (B,N,J)$ be a morphism of $A_\infty$ algebras with $A_\infty$ modules and $A_\infty$ inner products. Then, we get an inclusion of Massey inner products; more precisely, for $(a_{1,1},\dots, m_{k+1,k+1},a_{k+2,k+2},\dots, m_{k+\ell+2,k+\ell+2})\in A^{\times k}\times M\times A^{\times \ell}\times M$, we have the inclusion
\begin{multline}\label{EQU:MIP-morph-inclusion}
 \la a_{1,1},\dots, m_{k+1,k+1},a_{k+2,k+2},\dots, m_{k+\ell+2,k+\ell+2}\ra_{I}
 \\ \subseteq 
 \la f_1(a_{1,1}),\dots, F_{0,0}(m_{k+1,k+1}),f_1(a_{k+2,k+2}),\dots, F_{0,0}(m_{k+\ell+2,k+\ell+2}) \ra_{J}\quad\subseteq R
\end{multline}
\end{proposition}
\begin{proof}
First, we note that we get the same Massey inner product set for $I$ and $F^\sharp(J)$. This follows since a cyclic defining system $\xbb,\xbp$ on $M$ is independent of the given $A_\infty$ inner product, so that any cyclic defining system can be applied to any $A_\infty$ inner product, showing that their Massey inner products are either all empty or are all non-empty. Now, the difference in the staircase inner products \eqref{EQU:F(x**,x'**)} for $I$ and $F^\sharp(J)$ is the staircase inner product for an exact $A_\infty$ inner product, which by Lemma \ref{LEM:exactF=>0MIP} is zero. Thus, any Massey inner product sets for $I$ and $F^\sharp(J)$ are equal, $\la\dots\ra_{I}=\la\dots\ra_{F^\sharp(J)}$.

Next, we show that there are inclusions of Massey inner product sets $\la\dots\ra_{F^\sharp(J)}\subseteq \la\dots\ra_{J}$. If $\xbb, \xbp$ is a cyclic defining system for $(a_{1,1},\dots, m_{k+\ell+2,k+\ell+2})\in A^{\times k}\times M\times A^{\times \ell}\times M$, then, Lemma \ref{LEM:f(defining-system)} shows that there is a cyclic defining system $\ybb, \ybp$ given by Equation \eqref{EQU:y**=f(x**)}, whose diagonal elements are $(f_1(a_{1,1}),\dots, F_{0,0}(m_{k+\ell+2,k+\ell+2}))\in  B^{\times k}\times N\times B^{\times \ell}\times N$. (Note, that we apply $f_r$ in \eqref{EQU:y**=f(x**)} if all $x_{i,j}$ are in $A$ or else $F_{r,s}$ is one of the $x_{i,j}$ is in $M$.) Now, we claim that $\mc I_{F^\sharp (J)}(\xbb,\xbp)=\mc I_{J}(\ybb,\ybp)$,  which follows since the cyclic staircases of $J$ evaluated at the $y_{i,j}$ are precisely the cyclic staircases of $F^\sharp(J)$ evaluated at the $x_{i,j}$, i.e., they are both equal to
\begin{multline*}
\sum (-1)^\varepsilon\cdot\Big(J_{i,j} \big(\ul{f_{r_1}}(\ul{x_{u_0,u_1}},\dots), \dots , \ul{f_{r_i}}(\dots), \ul{F_{r,s}}(\dots), \ul{f_{s_1}}(\dots), \dots , \ul{f_{s_j}}(\dots, \ul{x_{v_{n-1}+1,v_n}})\big)\Big)
\\
\Big(\ul{F_{\ell-q,k-p}} (\ul{x'_{v_{n+1},w_1}}, \dots, \ul{x'_{w_m,u_0-1}})\Big)
\end{multline*}
Here, the sum is over all cyclic staircases of the $x_{i,j}$, which, after applying $f_r$ and $F_{r,s}$, give the cyclic staircases of the $y_{i,j}$. Thus, we get that $\mc I_{F^\sharp (J)}(\xbb,\xbp)=\mc I_{J}(\ybb,\ybp)$, which shows that we have the inclusion $\la\dots\ra_{F^\sharp(J)}\subseteq \la\dots\ra_{J}$.
\end{proof}

\begin{definition}\label{DEF:AAIP-equivalence}
Two $A_\infty$ algebras with $A_\infty$ modules and $A_\infty$ inner products $(A,M,I:M\to M^*)$ and $(B,N,J:N\to N^*)$ are called \emph{equivalent}, if there are morphisms $(f,F):(A,M,I)\to (B,N,J)$ and  $(g,G):(B,N,J)\to (A,M,I)$ which are quasi-inverses of each other, i.e., their lowest components are inverses on homology,
$$(f_1)_*=(g_1)_*^{-1}:H(A)\to H(B)\quad \text{ and }\quad (F_{0,0})_*=(G_{0,0})_*^{-1}:H(M)\to H(N)$$
\end{definition}

Proposition \ref{PROP:MIP-under-morphism} implies the following corollary,  which gives a sufficient condition for a bijection of Massey inner product sets.

\begin{corollary}\label{COR:equiv-AAIP-gives-equal-Massey}
If $(A,M,I:M\to M^*)$ and $(B,N,J:N\to N^*)$ are equivalent, then their Massey inner product sets are equal, i.e. for all $(a_{1,1},\dots, m_{k+\ell+2,k+\ell+2})\in A^{\times k}\times M\times A^{\times \ell}\times M$,
\begin{equation}
 \la a_{1,1},\dots, m_{k+\ell+2,k+\ell+2}\ra_{I}= 
 \la f_1(a_{1,1}),\dots, F_{0,0}(m_{k+\ell+2,k+\ell+2}) \ra_{J}
\end{equation}
where $(f,F):(A,M,I)\to (B,N,J)$ is a map inducing the equivalence.
\end{corollary}

\begin{definition}\label{DEF:AAIP-formal}
An $A_\infty$ algebra with $A_\infty$ module and $A_\infty$ inner product is called \emph{formal}, if it is equivalent to one whose $A_\infty$ algebra has only non-vanishing product map $\mu_2$, the $A_\infty$ module has only non-vanishing bimodule maps $\be_{1,0}$ and $\be_{0,1}$, and the $A_\infty$ inner product  has only non-vanishing inner product $I_{0,0}$.
\end{definition}

\begin{corollary} \label{cor;formalMIP}
If $(A,M,I)$ is formal, then all Massey inner products are trivial (i.e. they are empty or contain $0$).
\end{corollary}
\begin{proof}
By Corollary \ref{COR:equiv-AAIP-gives-equal-Massey} we only need to show that if $(A,M,I)$ only has non-vanishing components $\mu_2$, $\be_{1,0}$, $\be_{0,1}$, and $I_{0,0}$, then it has trivial Massey inner product sets. Assume that a Massey inner product set is non-empty. Then it has a cyclic defining system. One can choose the cyclic defining system $\xbb,\xbp$ with $x_{i,j}=0$ and $x'_{i,j}=0$ for $i\neq j$, since there are no higher multiplications or module maps, and for this we get $\mc I(\xbb,\xbp)=0$, so that $\la x_{1,1},\dots,x_{k+\ell+2,k+\ell+2} \ra_{I}$ contains $0$.
\end{proof}

It is possible that a dga is formal, but the dga with inner product is not formal, i.e. there is in general more information in a homotopy inner product, beyond the information of the underlying algebra.

\begin{example}
Let $A$ be a graded module with $A_\infty$ structure $0$, and let $M=A$ be its $A_\infty$ module over $A$, i.e., all module maps are zero as well. For some fixed $k,\ell\geq 0$, let $I_{k,\ell}:A^{\ot k}\ot M\ot A^{\ot \ell}\to M^*$ be a non-zero map. Then $I$ is an $A_\infty$ inner product, as it satisfies \eqref{EQU:F=inner-product}.

Now, any $(x_{1,1},\dots, x_{k+\ell+2,k+\ell+2})$ has a cyclic defining system $\xbb,\xbp$ with $x_{i,j}=0$ and $x'_{i,j}=0$ for $i\neq j$, and with this $\mc I(\xbb,\xbp)=I_{k,\ell}(x_{1,1},\dots, x_{k+\ell+2,k+\ell+2})$. If we choose the $x_{j,j}$ in such degrees that this defining system is the only possible one according to the degrees of the $x_{i,j}$ and $x'_{i,j}$, then $\la x_{1,1},\dots, x_{k+\ell+2,k+\ell+2}\ra_I=\{I_{k,\ell}(x_{1,1},\dots)(x_{k+\ell+2,k+\ell+2})\}$. As $I_{k,\ell}$ is non-zero, there exists $x_{j,j}$ such that the Massey inner product set is non-trivial, and thus $(A,A,I)$ is not formal. However, the zero algebra (and module) is certainly formal as a dga (and dg module).
\end{example}

We will compute more interesting non-trivial Massey inner product sets in the next section.

\begin{remark}
The precise chain level version of the inner product coming from Poincar\'e duality plays an important role for computations in string topology. For example, $S^2$ is formal as a dga, but the inner product of $S^2$ is not formal, which has non-trivial implications when giving algebraic models for string topology operations; see \cite{Men, PT}. Nevertheless, it seems that the Massey inner products for $S^2$ are trivial. 

Another question is whether the lens spaces $L(p,q)$ have any non-trivial Massey inner products, as their string topology coproducts differ for different values of $q$; see \cite{Naef}. To see whether there are any interesting Massey inner products for $L(p,q)$, one needs to use a chain level version of the Poincar\'e duality $A_\infty$ inner product (e.g. using the construction from \cite{FRT}), which then could also be transferred to cohomology using the formulas from Appendix \ref{APP:L(p,q)-mimimal}, and the results of the following subsection.
\end{remark}

\subsection{Massey inner product sets and the homotopy transfer theorem}

In practice it may be difficult to check if two $A_\infty$ structures are equivalent in the sense of Definition \ref{DEF:AAIP-equivalence}. Fortunately, in some cases one can check directly that the Massey inner product sets are equal. One such case is given by the homotopy transfer theorem for $A_\infty$ algebras, see \cite{Kad82}.

We fix some notation for this section as we review the explicit formulas for the transfer theorem from \cite{M}.

\begin{setup}\label{SETUP:transfer-theorem}
Let $R$ be any ground ring. Let $(A,\mu_1)$ be a graded chain complex and let $(B,\nu)$ be an $A_\infty$ algebra, and we also write $d_A=\mu_1$ and $d_B=\nu_1$ for the differentials on $A$ and $B$, respectively. Assume there are chain maps $f:B\to A$ and $g:A\to B$ such that $g\circ f$ is chain homotopic to the identity via the homotopy $h:B\to B$, i.e., $g\circ f-id_B=d_B\circ h+h\circ d_B$. For example, if $(A,\mu_1)=(H(B),0)$ is the cohomology of $B$, then, over a field $R$ (of any characteristic), there are always such maps $f$, $g$, and $h$; see, for example, \cite[Lemma 9.4.4.]{LV}.

With this, define the p-kernel to be $p_n:B^{\ot n}\to B$ given by $p_2=\nu_2$, and for $n>2$,
\[
\ul{p_n}=\sum_{\tiny\begin{matrix} \ell\geq 2, r_1,\dots, r_\ell\geq 1\\ r_1+\dots+r_\ell=n\end{matrix}} \,\,
\ul{\nu_\ell}\Big((\ul{h}\circ \ul{p_{r_1}})\ot \dots \ot (\ul{h}\circ \ul{p_{r_\ell}})\Big),
\]
where, by formal convention, the expression $h\circ p_1$ is meant to be $h\circ p_1=id_B$. Note, that this means that $p_n$ is given by a sum over all planar rooted trees, where we apply multiplications $h\circ \nu_r$ at the internal vertices, except for the one close to the root which only has $\nu_r$; see \eqref{EQU:p-kernel-of-y}. Moreover,
\begin{align*}
&\text{we already have }&& \mu_1  &&\text{ and define} && \mu_n:=f\circ p_n\circ g^{\ot n}:A^{\ot n}\to A \text{ for }n>1,\\
&\text{as well as }  && k_1:=g && \text{ and}  && k_n:=h\circ p_n\circ g^{\ot n}:A^{\ot n}\to B\text{ for }n>1.
\end{align*}
\end{setup}
\begin{theorem}[\cite{K, Kad82, M}]\label{THM:transfer-theorem}
In the notation from Setup \ref{SETUP:transfer-theorem}, we get that $(A,\mu)=(A,\mu_1,\mu_2,\dots)$ is an $A_\infty$ algebra, and $k:(A,\mu)\to (B,\nu)$ is a morphism of $A_\infty$ algebras.
\end{theorem}

Now, $A$ is an $A_\infty$ module over $A$, and $B$ is an $A_\infty$ module over $B$, and by Example \ref{EXA:Morph-of-Ainfty-modules}\eqref{ITEM:f-gives-F} $k$ gives an induced morphism of $A_\infty$ modules $K:A\to B$ over the $A_\infty$ algebra map $k:A\to B$. Then, we have the following theorem.
\begin{theorem} \label{thm;pullbacknat}
Using the above notation, if $I$ is an $A_\infty$ inner product on $B$, then $I$ has the same Massey inner product sets as the pull back $A_\infty$ inner product $K^\sharp(I)$ on $A$. In other words, for fixed $s,t$, and for any $(a_{1,1},\dots, a_{s+t+2,s+t+2})\in A^{\times s}\times A\times A^{\times t}\times A$, we have equality of $(s,t)^{th}$ Massey inner product sets
\begin{equation}\label{EQU:transfer-pullback-equal}
 \la a_{1,1},\dots, a_{s+t+2,s+t+2}\ra_{K^\sharp(I)}
 \\ = 
 \la g(a_{1,1}),\dots, g(a_{s+t+2,s+t+2}) \ra_{I}
\end{equation}
\end{theorem}
\begin{proof}\quad

``$\subseteq$": Since $(k,K):(A,A,K^\sharp(I))\to (B,B,I)$ is a morphism (and recall that $k_1=K_{0,0}=g$), Equation \eqref{EQU:MIP-morph-inclusion}  in Proposition \ref{PROP:MIP-under-morphism} shows that there is an inclusion from left to right of \eqref{EQU:transfer-pullback-equal}.

``$\supseteq$": To check the inclusion from right to left of \eqref{EQU:transfer-pullback-equal}, assume that $r\in R$ is an element in the right-hand side of \eqref{EQU:transfer-pullback-equal}.

\noindent{\bf Step 1:} There exists a cyclic defining system for $B$ with $\mc I(\ybb,\ybp)=r$ and a cyclic triangular system $\xbb,\xbp$ for $A$ such that for all $i,j$ we have $d_A(x_{i,j})=f(d_B(y_{i,j}))$ and
\begin{equation}\label{EQU:inductive-x-y-relation}
y_{i,j}+h\circ d_B(y_{i,j})=g(x_{i,j})\quad \text{ and } \quad y'_{i,j}+h\circ d_B(y'_{i,j})=g(x'_{i,j})
\end{equation}
\begin{itemize}
\item[\it Proof.\hspace{0.9mm}]
We start with any cyclic defining system $\{\wt {y_{i,j}}\},\{\wt{y'_{i,j}}\}$ for $B$ with $\mc I(\{\wt {y_{i,j}}\},\{\wt{y'_{i,j}}\})=r$. From the homotopy relation, we have $g(f(\wt{y_{i,j}}))-\wt{y_{i,j}}=d_B(h(\wt {y_{i,j}}))+h(d_B(\wt {y_{i,j}}))$. We will inductively change the defining system by adding the exact terms $y_{i,j}:=\wt{y_{i,j}}+d_B(h(\wt {y_{i,j}}))$ at the $(i,j)^{th}$ spot, which changes the whole defining system above and to the left of the $(i,j)^{th}$ spot, but in such a way that the $\mc I$ output stays the same. (This can be done according to Equation \eqref{EQU:F(x,x')=F(wt-x,wt-x')}). Moreover, we define $x_{i,j}=f(\wt{y_{i,j}})$, so that $d_A(x_{i,j})=f(d_B(y_{i,j}))$. Note, that then \eqref{EQU:inductive-x-y-relation} also holds, since $y_{i,j}=\wt{y_{i,j}}+d_B(h(\wt {y_{i,j}}))=g(f(\wt{y_{i,j}}))-h(d_B(\wt {y_{i,j}}))=g(x_{i,j})-h\circ d_B(y_{i,j})$ and similarly for $y'_{i,j}$.

We keep changing the cyclic defining system (and definions of $x_{i,j}$) starting from the diagonal elements $\wt{y_{i,i}}$ and $\wt{y'_{i,i}}$, and keep moving away from the diagonal, so that we end up with a cyclic defining system $\ybb,\ybp$ for $B$ with $\mc I(\ybb,\ybp)=r$, and a cyclic triangular system $\xbb, \xbp$ for $A$ such that \eqref{EQU:inductive-x-y-relation} holds.
\end{itemize}

\noindent{\bf Step 2:} $\xbb,\xbp$ from Step 1 is a cyclic defining system for $A$.
\begin{itemize}
\item[\it Proof.\hspace{0.9mm}]
We will check that $d_A(x_{i,j})=-\St{\xbb}{i,j}$. (The proof for $x'_{i,j}$ is similar.)

Since $\ybb$ in Step 1 is a defining system, we get for $d_A(x_{i,j})$ (c.f. \eqref{EQU:x**}):
\begin{equation}\label{EQU:d(xij)-in-transfer}
\ul{d_A}(\ul{x_{i,j}})=\ul{f}(\ul{d_B}( \ul{y_{i,j}}))=-\ul{f}\Big(\ul{\St{\ybb}{i,j}}\Big)
=-\ul{f}\Big(\sum_{\ell,j_0,\dots,j_\ell} \ul{\nu_\ell}(\ul{y_{j_0,j_1}},\dots,\ul{y_{j_{\ell-1}+1,j_\ell}}) \Big)
\end{equation}
Next we compute $\St{\xbb}{i,j}$. From the definition of $\mu_\ell$ and $p_\ell$ in Setup \ref{SETUP:transfer-theorem}, we get
\begin{multline*}
\ul{\St{\xbb}{i,j}}=\sum_{\ell,j_0,\dots,j_\ell} \ul{\mu_\ell}(\ul{x_{j_0,j_1}},\dots,\ul{x_{j_{\ell-1}+1,j_\ell}}) 
=\sum_{\ell,j_0,\dots,j_\ell} \ul{f}\circ \ul{p_\ell}\Big(\ul{g}(\ul{x_{j_0,j_1}}),\dots,\ul{g}(\ul{x_{j_{\ell-1}+1,j_\ell}})\Big) \\
=\ul{f}\circ \big( \sum_{\tiny\begin{matrix}\ell,j_0,\dots,j_\ell\\ m,r_1,\dots,r_m\end{matrix}}  \ul{\nu_m}\big((\ul{h}\circ \ul{p_{r_1}})\ot \dots \ot (\ul{h}\circ \ul{p_{r_m}})\big)\Big) \big(\ul{y_{j_0,j_1}}+\ul{h}\ul{d_B}(\ul{y_{j_0,j_1}}),\dots,\ul{y_{j_{\ell-1}+1,j_\ell}}+\ul{h}\ul{d_B}(\ul{y_{j_{\ell-1}+1,j_\ell}})\big)
\end{multline*}
Now, for the $q^{th}$ input on the right, we use $d_B({y_{j_{q-1}+1,j_q}})=-\St{\ybb}{j_{q-1}+1,j_q}$ to get
\begin{equation}\label{EQU:y+hby}
\ul{y_{j_{q-1}+1,j_q}}+\ul{h}\ul{d_B}(\ul{y_{j_{q-1}+1,j_q}})
=\ul{y_{j_{q-1}+1,j_q}}-\ul{h}\Big(\sum \ul{\nu_t}(\ul{y_{i_0,i_1}},\dots,\ul{y_{i_{t-1}+1,i_t}})\Big)
\end{equation}
Thus, we either apply $\ul{y_{j_{q-1}+1,j_q}}$ or $-\ul{h}\big(\sum \ul{\nu_t}(\dots)\big)$, which is just the output of a different term in the sum for $\ul{\St{\xbb}{i,j}}$.

An example may be illuminating here. Recall that the p-kernel are given by planar rooted trees which apply some $h\circ \nu_i$ at some internal vertex except at the one closest to the root where we only apply some $\nu_i$. For example:
\begin{equation}\label{EQU:p-kernel-of-y}
\scalebox{0.9}[1]{\begin{tikzpicture}
\draw (0,2) node {\scalebox{.8}{$y_{1,2}$}};
\draw (1,2) node {\scalebox{.8}{$y_{3,3}$}};
\draw (2,2) node {\scalebox{.8}{$y_{4,6}$}};
\draw (3,2) node {\scalebox{.8}{$y_{7,8}$}};
\draw (4,2) node {\scalebox{.8}{$y_{9,11}$}};
\draw (5,2) node {\scalebox{.8}{$y_{12,14}$}};
\draw (6,2) node {\scalebox{.8}{$y_{15,15}$}};
\draw (7,2) node {\scalebox{.8}{$y_{16,18}$}};
\draw (8,2) node {\scalebox{.8}{$y_{19,22}$}};

\draw (5,1.7)--(5,-.7); \fill (5,0) circle (.1); \draw (4.7,-.2) node {\scalebox{0.8}{$\nu_{3}$}};
\draw (4,1.7)--(5,1.2)--(6,1.7); \fill (5,1.2) circle (.1); \draw (5.5,1.1) node {\scalebox{0.8}{$h\nu_{3}$}};
\draw (3,1.7)--(5,0.6); \fill (5,0.6) circle (.1); \draw (5.4,0.6) node {\scalebox{0.8}{$h\nu_{2}$}};
\draw (1,1.7)--(1.5,1.2)--(5,0); \draw (0,1.7)--(1.5,1.2)--(2,1.7);
\fill (1.5,1.2) circle (.1); \draw (1.1,1.1) node {\scalebox{0.8}{$h\nu_{3}$}};
\draw (7,1.7)--(7.4,1.2)--(8,1.7); \draw (7.4,1.2)--(5,0); 
\fill (7.4,1.2) circle (.1); \draw (7.8,1.1) node {\scalebox{0.8}{$h\nu_{2}$}};
\end{tikzpicture}}
\end{equation}
Then, the term corresponding to the above tree can either come from an expression
\[
\nu_3\circ ((h \circ p_3)\ot (h\circ p_4)\ot (h \circ p_2))\Big(y_{1,2},y_{3,3},y_{4,6},y_{7,8},y_{9,11},y_{12,14},
y_{15,15},y_{16,18},y_{19,22}\Big),
\]
or it could also come from
\[
\nu_3\circ ((h \circ p_3)\ot (h\circ p_2)\ot (h \circ p_2))\Big(y_{1,2},y_{3,3},y_{4,6},y_{7,8},(h\circ \nu_3)(y_{9,11},y_{12,14},y_{15,15}),y_{16,18},y_{19,22}\Big).
\]
Note that the latter comes with a minus sign, since it came from $-\ul{h}\big(\sum \ul{\nu_t}(\dots)\big)$ in the input.

Thus, for each term corresponding to a tree with at least one internal vertex other than the one close to the root, we always get this precise term canceling, since we either apply the $y_{*,*}$ all the way at the end, or we apply some of the inputs $y_{*,*}$ at exactly one level lower with $-\ul{h}\big(\sum \ul{\nu_t}(y_{*,*},\dots, y_{*,*})\big)$. When this can be done in several places, we get the same term $2^{\text{\# of places}}$ many times, each time canceling in pairs by signs. It follows that all terms in the sum for $\ul{\St{\xbb}{i,j}}$ cancel, except those which correspond to tress with only one internal vertex, for which we apply only $y_{*,*}$; the corresponding $\ul{h}\big(\sum \ul{\nu_t}(\dots)\big)$ terms cancel with higher level trees. The remaining terms are precisely the ones with all $h\circ p_1=id_B$ inside the $\nu_\ell$, and so we get $\ul{\St{\xbb}{i,j}}=\ul{f}\Big(\sum_{\ell,j_0,\dots,j_\ell} \ul{\nu_\ell}(\ul{y_{j_0,j_1}},\dots,\ul{y_{j_{\ell-1}+1,j_\ell}}) \Big)$. Putting this together with \eqref{EQU:d(xij)-in-transfer} gives the claimed identity $d_A(x_{i,j})=-\St{\xbb}{i,j}$.
\end{itemize}

\noindent{\bf Step 3:} We have the equality $\mc I(\xbb,\xbp)=\mc I(\ybb,\ybp)=r$.
\begin{itemize}
\item[\it Proof.\hspace{0.9mm}]
The argument is similar to the one in Step 2. We compute $\mc I(\xbb,\xbp)$ from \eqref{EQU:F(x**,x'**)} applied to the pull back $A_\infty$ inner product $K^\sharp(G)$ (c.f. \eqref{EQU:Fsharp(G)}) as follows:
\begin{multline*}
\mc I(\xbb,\xbp)=\sum_{u,v,i_0,\dots,j_0,\dots} (-1)^\epsilon\cdot  \ul{(K^\sharp(I))_{u,v}}(\ul{x_{i_0,i_1}},\dots, \ul{x_{j_{v-1}+1,j_v}})(\ul{x'_{j_{v}+1,i_0-1}})
\\
=\sum_{\tiny\begin{matrix}u,v\\ m,m_1,\dots,m_u\\ n,n_1,\dots,n_v\\ i_0,\dots,j_0,\dots\end{matrix}} (-1)^\epsilon  \ul{I_{u,v}}\Big(\ul{k_{m_1}}(\dots),\dots,\ul{K_{m,n}}(\dots),\dots, \ul{k_{m_v}}(\dots)\Big)\big(\ul{K_{m',n'}}(\dots,\ul{x'_{j_{s}+1,i_0-1}},\dots)\big),
\end{multline*}
where we apply appropriate $x_{i,j}$ or $x'_{i,j}$ from the cyclic defining system in the dotted spaces. Now, the $k_n$ for $n>1$ (and similarly the $K_{n,m}=k_{n+m+1}$ for $n+m>0$) are given by the p-kernel $k_\ell=h\circ p_\ell\circ g^{\ot \ell}$, and so, applying them to the $x_{i,j}$ or $x'_{i,j}$ gives
\begin{multline*}
\ul{k_\ell}(\ul{x_{i'_0,i'_1}},\dots,\ul{x_{i'_{\ell-1}+1,i'_\ell}})
=\ul{h}\circ \ul{p_\ell}\Big(\ul{g}(\ul{x_{i'_0,i'_1}}),\dots,\ul{g}(\ul{x_{i'_{\ell-1}+1,i'_\ell}})\Big) \\
= \Big( \sum \ul{h}\circ\ul{\nu_m}\big((\ul{h}\circ \ul{p_{r_1}})\ot \dots \ot (\ul{h}\circ \ul{p_{r_m}})\big)\Big) \big(\ul{y_{i'_0,i'_1}}+\ul{h}\ul{d_B}(\ul{y_{i'_0,i'_1}}),\dots,\ul{y_{i'_{\ell-1}+1,i'_\ell}}+\ul{h}\ul{d_B}(\ul{y_{i'_{\ell-1}+1,i'_\ell}})\big).
\end{multline*}
Using \eqref{EQU:y+hby} from Step 2, we see that we apply either some $\ul{y_{*,*}}$ or $\ul{h}\ul{d_B}(\ul{y_{*,*}})=-\ul{h}\big(\sum \ul{\nu_t}(\ul{y_{*,*}},\dots,\ul{y_{*,*}})\big)$, which is just the output of a different term in the sum for $\mc I(\xbb,\xbp)$.

Thus, just as in Step 2, all terms corresponding to trees with more than one internal vertex cancel, and what remains are the terms $y_{*,*}$ from $k_1(x_{*,*})$ or $K_{0,0}(x_{*,*})$, i.e.,
\[
\mc I(\xbb,\xbp)
=\sum (-1)^\varepsilon \ul{I_{u,v}}(\ul{y_{i_0,i_1}},\dots, \ul{y_{j_{v-1}+1,j_v}})(\ul{y'_{j_{v}+1,i_0-1}})
=\mc I(\ybb,\ybp)=r
\]
\end{itemize}
Step 3 shows that $r$ is in the left side of \eqref{EQU:transfer-pullback-equal}, which completes the proof of the theorem.
\end{proof}

\section{Cyclic Massey products for commutative dgas}\label{SEC:Massey-cdga}

In this section, we consider graded commutative dgas with an inner product given by multiplication and evaluation on a dual element, which is an example of an $A_\infty$ inner product as described in the previous section (Example \ref{EXA:cdga=>IP}\eqref{ITM:cdga-example}). However, in this section, we will mainly use the notation from Section \ref{SEC:Massey-A-infty}, and only refer back to Sections \ref{SEC:Massey-modules} and  \ref{SEC:Massey-inner-product} to avoid duplication of figures, etc.

\subsection{Definition and naturality of cyclic Massey products}

Let $A$ be graded commutative dga, regarded as a dg bimodule over itself. Recall from Definition \ref{defn:triangsys}, that for a defining system $\xbb=\{x_{i,j}\}_{p\leq i\leq j\leq q}$, the staircase product is
\begin{equation}\label{EQU:S-expanded-for-cyc}
 \St\xbb{p,q}= -(-1)^{|x_{p,p}|}x_{p,p}x_{p+1,q}-(-1)^{|x_{p,p+1}|}x_{p,p+1}x_{p+2,q}-\dots-(-1)^{|x_{p,q-1}|}x_{p,q-1}x_{q,q}
\end{equation}
\begin{definition}Fix two non-negative integers $k$ and $\ell$.
\begin{enumerate}
\item
A \emph{cyclic defining system} for the $k+\ell+2$ elements $(a_{1,1},\dots, a_{k+\ell+2,k+\ell+2})\in A^{\times k+\ell+2}$ consists of two triangular systems (c.f. Definition \ref{defn:triangsys}), $\xbb$ for $(a_{1,1},\dots, a_{k+\ell+1,k+\ell+1})$, and $\xbp$ for $(a_{k+2,k+2},\dots,a_{k+\ell+2,k+\ell+2},a_{1,1},\dots, a_{k,k})$ both of which \emph{include} the top left corner points, and which coincide on their intersection triangular points, satisfying
\begin{equation}\label{EQU:cyclic:dx=-Sx,dx'=-Sx'}
d(x_{i,j})=- \St\xbb{i,j}\quad\text{and}\quad d(x'_{i,j})=- \St\xbp{i,j}\quad\text{for all $i,j$}
\end{equation}
It is convenient to use cyclic indexing notation for $\xbp$, so that after the index $k+\ell+2$ we continue with the index $1$ (see \eqref{EQU:1-7-triangular-system} and \eqref{EQU:11-staircaseIP}, for example, but now all $a_{i,j}$ and $m_{i,j}$ are in $A$). We note that $a_{k+1,k+1}$ and $a_{k+\ell+2,k+\ell+2}$ play a special role as the first is missing $\xbp$ and the second is missing in $\xbb$.
\item
Each cyclic defining system $\xbb,\xbp$ associated to elements $(a_{1,1},\dots, a_{k+\ell+2,k+\ell+2})$ determines an element in $A$ via the \emph{cyclic staircase product}, which is defined to be the element $\mc C=\StC{\xbb,\xbp}\in A$ given by
\begin{align}\label{EQU:C(x**,x'**)}
\mc C=& \sum_{1\leq i\leq k+1\leq j\leq k+\ell+1} \quad (-1)^{|{x_{i,j}}|+\varepsilon_i} \cdot x_{i,j}\cdot x'_{j+1,i-1},
\\ \nonumber
&\text{where }(-1)^{\varepsilon_i}=(-1)^{1+(|\ul{a_{1,1}}|+\dots+|\ul{a_{i-1,i-1}}|)\cdot (|\ul{a_{i,i}}|+\dots+|\ul{a_{k+\ell+2,k+\ell+2}}|)}
\end{align}
Here, if $i=1$ we interpret $i-1$ to mean $k+\ell+2$, and we recall that the underline means shifting down by $1$, i.e., $|\ul{a_{i,j}}|=|a_{i,j}|-1$.
\item
For $(a_{1,1},\dots,a_{k+\ell+2,k+\ell+2})\in A^{\times k+\ell+2}$, we define the \emph{$(k,\ell)^{th}$ cyclic Massey product} (or \emph{$(k,\ell)^{th}$ cyclic Massey product set}) to be the subset of $A$ given by 
\begin{equation*} 
\la\la a_{1,1},\dots,a_{k+\ell+2,x+\ell+2} \ra\ra 
 :=\Big\{ \StC{\xbb,\xbp} \Big| \xbb, \xbp \text{ is a cyclic defining system} \Big\} \subseteq A
\end{equation*}
We say that the cyclic Massey product is trivial, if $\la\la a_{1,1},\dots,a_{k+\ell+2,k+\ell+2} \ra\ra$ is either the empty set or contains $0\in A$.
\end{enumerate}
\end{definition}

\begin{example} \label{ex;cyclicstrair11}
For example, if $k=\ell=1$, a staircase inner product for $(a_{1,1},m_{2,2},a_{3,3},m_{4,4})\in A^{\times 4}$ with the cyclic defining system $\xbb, \xbp$ shown in \eqref{EQU:11-staircaseIP} (where all $a_{i,j}$ and $m_{i,j}$ are in $A$) is given by the following sum (where $(-1)^\varepsilon=(-1)^{1+(|{a_{1,1}}|-1)\cdot (|{m_{2,2}}|+|{a_{3,3}}|+|{m_{4,4}}|-3)}$):
\begin{align*}
\mc C (\xbb,\xbp)=&
(-1)^{|{m_{1,3}}|+1} \cdot m_{1,3}\cdot m_{4,4}+(-1)^{|{m_{2,2}}|+\varepsilon}\cdot m_{2,2}\cdot m_{3,1}
\\ &
+(-1)^{|{m_{1,2}}|+1}\cdot m_{1,2}\cdot m_{3,4}+(-1)^{|{m_{2,3}}|+\varepsilon}\cdot m_{2,3}\cdot m_{4,1}
\end{align*}
\end{example}

\begin{proposition}
The cyclic staircase product is always closed, i.e., $d(\StC{\xbb,\xbp})=0$. Moreover, the cyclic Massey product induces a well-defined map from the cohomologies $H(A)$ into the power set of $H(A)$,
\begin{align}\label{EQU:cyclic-Massey-on-H(A)}
 H(A)^{\times k+\ell+2}&\to \mathcal P(H(A)),\\ \nonumber
\la \la [a_{1,1}],\dots,[a_{k+\ell+2,x+\ell+2}] \ra\ra 
 &:=\Big\{ \big[\StC{\xbb,\xbp}\big] \Big| \xbb, \xbp \text{ is a cyclic defining system} \Big\}
\end{align}
\end{proposition}
\begin{proof}
We first show that $d(\mc C)=0$. We compute for some $1\leq i\leq k+1\leq j\leq k+\ell+1$:
$$d((-1)^{|{x_{i,j}}|+\varepsilon_i} \cdot x_{i,j}\cdot x'_{j+1,i-1})=(-1)^{|{x_{i,j}}|+\varepsilon_i} \cdot d(x_{i,j})\cdot x'_{j+1,i-1}+(-1)^{\varepsilon_i}\cdot x_{i,j}\cdot d(x'_{j+1,i-1})$$ where $d(x_{i,j})$ and $d(x'_{j+1,i-1})$ can be expanded via \eqref{EQU:cyclic:dx=-Sx,dx'=-Sx'} and \eqref{EQU:S-expanded-for-cyc}. The first term becomes a sum $d(x_{i,j})=\dots+(-1)^{|x_{i,\ell}|} \cdot x_{i,\ell}\cdot x_{\ell+1,j}+\dots$, so that we get a sum with a term
\begin{equation}\label{EQU:proof-d(C)=0}
(-1)^{|\ul{x_{\ell+1,j}}|+\varepsilon_i} \cdot x_{i,\ell}\cdot x_{\ell+1,j}\cdot x'_{j+1,i-1} \text{ for some } i\leq \ell\leq j-1
\end{equation}
where we used that $(-1)^{|{x_{i,j}}|+|x_{i,\ell}|}=(-1)^{|\ul{x_{i,j}}|+|\ul{x_{i,\ell}}|}=(-1)^{|\ul{x_{i,j}}|+\dots+|\ul{x_{i,\ell}}|}=(-1)^{|\ul{x_{\ell+1,j}}|}$ due to $|\ul{x_{p,q}}|=|\ul{x_{p,p}}|+|\ul{x_{p+1,p+1}}|+\dots +|\ul{x_{q,q}}|$. This will cancel with the other term, but we need to distinguish two cases, namely, whether $k+1\leq \ell$ or $\ell<k+1$.
\begin{enumerate}
\item[Case 1:] $k+1\leq \ell$. We then consider the term $(-1)^{\varepsilon_i}\cdot x_{i,\ell}\cdot d(x'_{\ell+1,i-1})=\dots+(-1)^{\varepsilon_i +|x_{\ell+1,j}|}\cdot  x_{i,\ell}\cdot x_{\ell+1,j}\cdot x'_{j+1,i-1}+\dots$, which after adding it to \eqref{EQU:proof-d(C)=0} becomes zero. 
\item[Case 2:] $\ell<k+1$. We consider the term  $(-1)^{\varepsilon_{\ell+1}}\cdot x_{\ell+1,j}\cdot d(x'_{j+1,\ell})=\dots+(-1)^{\varepsilon_{\ell+1} +|x'_{j+1,i-1}|}\cdot  x_{\ell+1,j}\cdot x'_{j+1,i-1}\cdot x_{i,\ell}+\dots=\dots+(-1)^{\varepsilon}\cdot x_{i,\ell}\cdot x_{\ell+1,j}\cdot x'_{j+1,i-1} +\dots$, where
\begin{align*}
(-1)^{\varepsilon}
=&(-1)^{\varepsilon_{\ell+1} +|x'_{j+1,i-1}|+|x_{i,\ell}|\cdot (|x_{\ell+1,j}|+|x'_{j+1,i-1} |)}
\\
=&(-1)^{1+|\ul{x_{1,\ell}}|\cdot |\ul{x_{\ell+1,k+\ell+2}}|+|x'_{j+1,i-1}|+(|\ul{x_{i,\ell}}|+1)\cdot (|\ul{x_{\ell+1,j}}|+|\ul{x'_{j+1,i-1} }|)}
\\
=&(-1)^{1+|\ul{x_{1,i-1}}|\cdot |\ul{x_{\ell+1,k+\ell+2}}|+|\ul{x_{i,\ell}}|\cdot |\ul{x_{\ell+1,k+\ell+2}}|+|\ul{x_{i,\ell}}|\cdot (|\ul{x_{\ell+1,j}}|+|\ul{x'_{j+1,i-1} }|)+|x_{\ell+1,j}|}
\\
=&(-1)^{(\varepsilon_i-|\ul{x_{1,i-1}}|\cdot |\ul{x_{i,\ell}}|) +|\ul{x_{i,\ell}}|\cdot |\ul{x_{\ell+1,k+\ell+2}}|+|\ul{x_{i,\ell}}|\cdot (|\ul{x_{\ell+1,\ell+1}}|+\dots+|\ul{x_{i-1,i-1} }|)+|x_{\ell+1,j}|}
\\
=&(-1)^{\varepsilon_i+|x_{\ell+1,j}|}
\end{align*}
Thus, adding this term to \eqref{EQU:proof-d(C)=0}, we get zero.
\end{enumerate}

It remains to show that the cyclic Massey product does not change when adding an exact term $dz$ to some diagonal element $a_{r,r}$ (of the same degree) in $(a_{1,1},\dots,a_{k+\ell+2,k+\ell+2})$. Let $\xbb, \xbp$ be a cyclic defining system for these $a_{j,j}$. To simplify notation, we will adopt the convention, that any $x_{i,j}$ with $i>j$ is meant to be $x'_{i,j}$. Then from Equations \eqref{EQU:Def-of-Z-explicit} and \eqref{EQU:dxtilde=-sxtilde} in Lemma \ref{LEM:Zs} we know that adding $Z_{p,r}=x_{p,r-1}\cdot z$ to $x_{p,r}$ and $Z_{r,q}=(-1)^{|dz|}\cdot z\cdot x_{r+1,q}$ to $x_{r,q}$ for $p<r<q$, since $Z_{i,j}=0$ unless $i=r$ or $j=r$, gives a new cyclic defining system, now for $(a_{1,1},\dots,a_{r,r}+dz,\dots,a_{k+\ell+2,k+\ell+2})$. Below, we list all differences between the staircase products \eqref{EQU:C(x**,x'**)} of the old cyclic defining system $\xbb, \xbp$ and new cyclic defining systems, coming from adding these new terms $Z_{i,j}$ to $x_{i,j}$.

If we added $dz$ to any $a_{r,r}, r\neq k+1,k+\ell+2$, then $\mc C$ does not change:
\begin{itemize}
\item
For $1\leq r\leq k$, the extra term $(-1)^{|Z_{r,q}|+\varepsilon_r} \cdot Z_{r,q}\cdot x_{q+1,r-1}$ in the cyclic staircase product \eqref{EQU:C(x**,x'**)} cancels with the extra term $(-1)^{|x_{r+1,q}|+\varepsilon_{r+1}}\cdot x_{r+1,q}\cdot Z_{q+1,r}$.
\item
For $k+2\leq r\leq k+\ell+1$, the extra term $(-1)^{|Z_{p,r}|+\varepsilon_p}\cdot Z_{p,r}\cdot x_{r+1,p-1}$ cancels with the extra term $(-1)^{|x_{p,r-1}|+\varepsilon_p}\cdot x_{p,r-1}\cdot Z_{r,p-1}$.
\end{itemize}

If we added $dz$ to $a_{r,r}$ for $r=k+1$ or $r=k+\ell+2$, then $\mc C$ changes by an exact term:
\begin{itemize}
\item
For $r=k+1$, all extra terms are given by $d((-1)^{|dz|+\varepsilon_r}\cdot z\cdot x_{k+2,k})=(-1)^{|dz|+\varepsilon_{k+1}}\cdot dz\cdot x_{k+2,k}+ z\cdot (\pm x_{k+2,k+2}\cdot x_{k+3,k}\pm\dots\pm x_{k+2,k-1}\cdot x_{k,k})$, where the term $z\cdot (\dots)$ is
\begin{align*}
&(-1)^{|Z_{r,k+2}|+\varepsilon_{r}}\cdot Z_{r,k+2}\cdot x_{k+3,k}
+\dots
+(-1)^{|Z_{r,k++\ell+1}|+\varepsilon_{r}}\cdot Z_{r,k+\ell+1}\cdot x_{k+\ell+2,k}
\\
& +(-1)^{|Z_{1,r}|+\varepsilon_{1}}\cdot Z_{1,r}\cdot x_{k+2,k+\ell+2}
+\dots
+(-1)^{|Z_{k,r}|+\varepsilon_{k}}\cdot Z_{k,r}\cdot x_{k+2,k-1}
\end{align*}
\item
For $r=k+\ell+2$, all extra terms are given by $d((-1)\cdot x_{1,k+\ell+1}\cdot z)=(\pm x_{1,1}\cdot x_{2,k+\ell+1}\pm\dots\pm x_{1,k+\ell}\cdot x_{k+\ell+1,k+\ell+1})\cdot z +(-1)^{|x_{1,k+\ell+1}|+1}\cdot x_{1,k+\ell+1}\cdot dz$, and $(\dots)\cdot z$ equals
\begin{align*}
&(-1)^{|x_{2,k+\ell+1}|+\varepsilon_{2}}\cdot x_{2,k+\ell+1}\cdot Z_{r,1}
+\dots
+(-1)^{|x_{k,k+\ell+1}|+\varepsilon_{k}}\cdot x_{k,k+\ell+1}\cdot Z_{r,k-1}
\\
& +(-1)^{|x_{1,k}|+\varepsilon_{1}}\cdot x_{1,k}\cdot Z_{k+1,r}
+\dots
+(-1)^{|x_{1,k+\ell}|+\varepsilon_{1}}\cdot x_{1,k+\ell}\cdot Z_{k+\ell+1,r}
\end{align*}
\end{itemize}

This completes the proof of the proposition.
\end{proof}

\begin{proposition}\label{PROP:cyclic-Massey-naturality}
If $f:A\to B$ is a map of graded commutative dgas, then there is an induced inclusion of cyclic Massey products
\begin{equation}
f_*\big(\,\,\la \la [a_{1,1}],\dots,[a_{k+\ell+2,x+\ell+2}] \ra\ra\,\,\big)
\subseteq
\la \la \,\,[f(a_{1,1})],\dots,[f(a_{k+\ell+2,x+\ell+2})] \,\,\ra\ra
\end{equation}
\end{proposition}
\begin{proof}
If $\xbb$, $\xbp$ is a cyclic defining system for $A$, then $\ybb$, $\ybp$ given by $y_{i,j}=f(x_{i,j})$ and $y'_{i,j}=f(x'_{i,j})$ is a cyclic defining system for $B$. This follows by direct inspection or from Definition \ref{DEF:f(x**)}, since \eqref{EQU:y**=f(x**)} is a defining system and all higher $f_r=0$ for $r>1$ vanish. Then, applying the dga map $f$ to \eqref{EQU:C(x**,x'**)} shows that $f(\StC{\xbb,\xbp})=\StC{\ybb,\ybp}$. Thus, every element in a cyclic Massey product in $A$ maps to the corresponding cyclic Massey product in $B$, which implies the claim.
\end{proof}

\begin{lemma}
Fix a closed element $x\in A^*$. We can compose the evaluation map $ev_x:H(A)\to R, [a]\mapsto x(a)$ with the cyclic Massey product \eqref{EQU:cyclic-Massey-on-H(A)} to give a map
$$\la\dots\ra_x:=ev_x\circ \la\la\dots\ra\ra: H(A)^{\times k+\ell+2}\to \mathcal P(H(A))\to \mc P(R)$$
Then, $\la\dots \ra_x$ is the Massey inner product $\la\dots \ra_{I_x}$ from \eqref{EQU:Massey-IP-DEF} applied to the $A_\infty$ inner product $I_x:A\to A^*, (I_x(a))(\wt{a})=x(a\cdot \wt{a})$, from Example \ref{EXA:cdga=>IP}\eqref{ITM:cdga-example}.
\end{lemma}

\begin{corollary}
Let $f:A\to B$ be a map of graded commutative dgas, and let $y\in B^*$ so that $f^*(y)\in A^*$. Then, for all $[a_{1,1}],\dots,[a_{k+\ell+2,k+\ell+2}]\in H(A)$, we have the inclusion:
\begin{equation}\label{EQU:cdga-IP-inclusion}
\la [a_{1,1}],\dots,[a_{k+\ell+2,k+\ell+2}]\ra_{f^*(y)}\subseteq \la [f(a_{1,1})],\dots,[f(a_{k+\ell+2,k+\ell+2)}]\ra_{y}
\end{equation}
\end{corollary}
\begin{proof}
If $r$ is in the set on the left side of \eqref{EQU:cdga-IP-inclusion}, then, for some cyclic defining system  $\xbb,\xbp$ for the $a_{j,j}$, we have $r=f^*(y)(\mc C(\xbb,\xbp))=y(f(\mc C(\xbb,\xbp))$, which is in the set on the right side of \eqref{EQU:cdga-IP-inclusion} by Proposition \ref{PROP:cyclic-Massey-naturality}.
\end{proof}

\subsection{Computations}

In this section we give some illustrative examples of nontrivial Massey inner product sets. Although the constructions up to here apply to the cochains of any space, we're primarily interested in manifolds. One knows that for closed manifolds, over a field of characteristic zero, Massey products landing in top degree are all trivial; see \cite{CFM}. Similar arguments (adding an appropriate class to a defining system to kill the resulting top degree form) show that there are no non trivial Massey inner products (in characteristic zero) coming from the top degree Poincar\'e duality pairing. On the other hand, this condition is an obstruction to a pseudo-manifold being a Poincar\'e duality space.

Here we have our first nontrivial Massey inner product set, defined on a $4$-manifold, with pairing supported in degree $2$.

\begin{example}(Filiform example)\label{EXA:filiform-example}
Consider the any compact $4$-manifold obtained by the quotient $M= \Gamma \backslash G$ of the simply connected nilpotent Lie group $G$ by some lattice $\Gamma$, where the Lie algebra of $G$ has generators $v_1,v_2,v_3,v_4$ and nontrivial brackets $[v_1,v_2 ]=v_3$ and $[v_1,v_3]=v_4$. 

This is the so-called filiform Lie algebra $\mathfrak g$ over $\R$ whose dual commutative dga is the minimal model of $M$, whose dual generators are  $e_1, e_2, e_3, e_4$ with relations $d e_3 = e_2 e_1$ and $d e_4 = e_3 e_1$.

One can compute $H^1 =\R.[e_1]\oplus \R.[e_2]$ and $H^2 = \R.[e_1e_4]\oplus \R.[e_2e_3]$. 
There is a nontrivial $(1,1)$ cyclic Massey product set
\[
\la\la e_1,e_1,e_1,e_2\ra\ra=\{2\cdot [e_1e_4]+c\cdot [e_2 e_3]:c \in \R\}\subseteq H(A),
\]
according to \ref{ex;cyclicstrair11}. To see this, note (using notation from \eqref{EQU:11-staircaseIP} with $a_{1,1}=m_{2,2}=a_{3,3}=e_1$ and $m_{4,4}=e_2$) there is a cyclic defining system using
$m_{3,4} = -m_{4,1} = -e_3$, since $d e_3 = e_2 e_1$, and $m_{3,1} = 2 e_4$ since
\[
d(2 e_4) = e_1 e_3 + (-e_3)e_1,
\]
while any other cyclic defining system differs only by adding multiples of the closed elements $e_1$ and $e_2$ at admissible entries, from which it is easy to check this gives the cyclic Massey product set above.

In particular, using the pairing which evaluates on  $[v_1v_4]$, we see there is a nontrivial Massey inner product set, which is a line that does not contain zero.
\end{example}

In our final example, we compute a non-trivial $(1,1)$ Massey inner product for a 4-component link complement in $S^3$. We compute this via two different approaches. The first one follows as in computations given by Massey \cite{Mas68} and others in the literature, e.g. \cite{On}.

\begin{example}[4-component link; approach \#1]\label{EXA:4-component-link}
Let $R=\Z$. Consider the 4-component link $L=L_1\cup L_2\cup L_3\cup L_4\subseteq S^3$ depicted below:
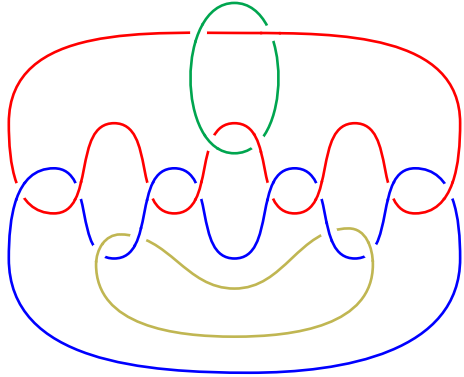
\[
\scalebox{0.4}{
\begin{tikzpicture}
\begin{knot}[clip width=8pt,clip radius=8pt, 
]
\strand [line width=2.5pt,blue]
(-.5,3) to [out=up, in=left] (1,6) to [out=right, in=left] (3,3) to [out=right, in=left] (5,6)
 to [out=right, in=left] (7,3)  to [out=right, in=left] (9,6)  to [out=right, in=left] (11,3)
 to [out=right, in=left] (13,6) to [out=right, in=up] (14.5,3) to [out=down, in=right] (7.5,-.8)
 to [out=left, in=down] (-.5,3);
 \strand [line width=2.5pt,red]
(-.5,7.5) to [out=down, in=left] (1,4.5) to [out=right, in=left] (3,7.5) to [out=right, in=left] (5,4.5)
 to [out=right, in=left] (7,7.5)  to [out=right, in=left] (9,4.5)  to [out=right, in=left] (11,7.5)
 to [out=right, in=left] (13,4.5) to [out=right, in=down] (14.5,7.5) to [out=up, in=right] (7.5,10.5)
  to [out=left, in=up] (-.5,7.5);
 \strand [line width=2.5pt,Green]
(7,6.5) to [out=right, in=right] (7,11.5) to [out=left, in=left] (7,6.5) ;
 \strand [line width=2.5pt,Yellow!70!black]
(3.2,3.8) to [out=right, in=left] (7,2) to [out=right, in=left] (10.8,4) to [out=right, in=up] (11.6,2.8)
to [out=down, in=right] (7,.4) to [out=left, in=down] (2.4,2.8) to [out=up, in=left] (3.2,3.8);
\flipcrossings{2,4, 6, 8, 9, 12, 13, 16}
\end{knot}
\end{tikzpicture}}
\]

We'll compute using relative chains and intersection, which we recall is related to cohomology with cup product in the following way. For each $i$, let $U_i\subseteq U'_i\subseteq S^3$ be open regular neighborhoods of $L_i$ with $\overline{U}_i\subseteq U'_i$, and assume these $U_i$ are sufficiently small so that their closures are disjoint (and similarly for the $U'_i$). Set $U=U_1\cup U_2\cup U_3\cup U_4$ (and similarly for $U'$) and let $X=S^3\setminus U$ be the complement, which is a manifold with boundary $\del X=\del U =\overline U\setminus U$. Then,
\[
H^k(X)\cong H_{3-k}(X,\del X)\cong H_{3-k}(S^3\setminus U,\overline{U'}\setminus U)\cong H_{3-k}(S^3,\overline {U'})\cong H_{3-k}(S^3,L)
\]
by Lefschetz duality, homotopy invariance, and excision. Under this isomorphism, the cup product is identified with the intersection product.  Finally, using the long exact sequence of $(S^3,L)$, we see that $H_1(S^3,L)\cong \Z^{\oplus 3}$ and $H_2(S^3,L)\cong \Z^{\oplus 4}$. 

Now, consider the closed relative chains $a,m,b,n\in C_2(S^3,L)$ and $r,s,t,u,v,w\in C_1(S^3,L)$ and also two more chains $p, q\in C_2(S^3,L)$ as shown below.
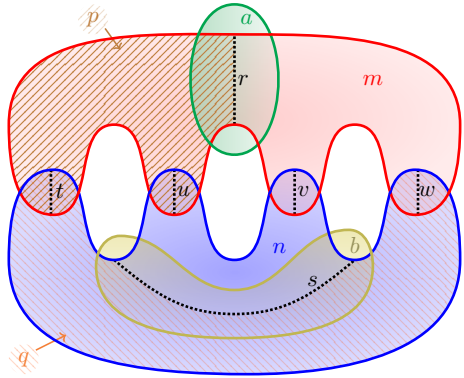
\[
\scalebox{0.4}{
\begin{tikzpicture}
 \filldraw [line width=0, fill=blue!60, opacity=0.6, path fading=fade out]
(-.5,3) to [out=up, in=left] (1,6) to [out=right, in=left] (3,3) to [out=right, in=left] (5,6) 
 to [out=right, in=left] (7,3)  to [out=right, in=left] (9,6)  to [out=right, in=left] (11,3)
 to [out=right, in=left] (13,6) to [out=right, in=up] (14.5,3) to [out=down, in=right] (7.5,-.8)
 to [out=left, in=down] (-.5,3);
 \filldraw [line width=0, fill=red!40, opacity=0.5, path fading=fade out]
(-.5,7.5) to [out=down, in=left] (1,4.5) to [out=right, in=left] (3,7.5) to [out=right, in=left] (5,4.5) 
 to [out=right, in=left] (7,7.5)  to [out=right, in=left] (9,4.5)  to [out=right, in=left] (11,7.5)
 to [out=right, in=left] (13,4.5) to [out=right, in=down] (14.5,7.5) to [out=up, in=right] (7.5,10.5)
  to [out=left, in=up] (-.5,7.5);
 \filldraw [line width=0, fill=Green!40, opacity=0.8, path fading=east]
(7,6.5) to [out=right, in=right] (7,11.5) to [out=left, in=left] (7,6.5) ;
 \filldraw [line width=0, fill=Yellow!80!black, opacity=0.4, path fading=south]
(3.2,3.8) to [out=right, in=left] (7,2) to [out=right, in=left] (10.8,4) to [out=right, in=up] (11.6,2.8)
to [out=down, in=right] (7,.4) to [out=left, in=down] (2.4,2.8) to [out=up, in=left] (3.2,3.8);

 \filldraw [draw=none, pattern color=brown, pattern=north east lines]
(-.5,7.5) to [out=down, in=left] (1,4.5) to [out=right, in=left] (3,7.5) to [out=right, in=left] (5,4.5) 
 to [out=right, in=left] (7,7.5) to [out=up, in=down] (7,10.5)
  to [out=left, in=up] (-.5,7.5);
 \filldraw [draw=none, pattern color=Orange, opacity=0.5, pattern=north west lines]
(-.5,3) to [out=up, in=left] (1,6) to [out=right, in=left] (3,3) to [out=-40, in=left] (7,1.2)
 to [out=right, in=220] (11,3) to [out=right, in=left] (13,6) to [out=right, in=up] (14.5,3) 
 to [out=down, in=right] (7.5,-.8) to [out=left, in=down] (-.5,3);

\draw [line width=2.5pt, dotted] (7,7.5) to [out=up, in=down] (7,10.5);
\draw [line width=2.5pt, dotted](3,3) to [out=-40, in=left] (7,1.2) to [out=right, in=220] (11,3);
\draw [line width=2.5pt, dotted] (.9,4.5) to [out=up, in=down] (.9,6);
\draw [line width=2.5pt, dotted] (5,4.5) to [out=up, in=down] (5,6);
\draw [line width=2.5pt, dotted] (9,4.5) to [out=up, in=down] (9,6);
\draw [line width=2.5pt, dotted] (13.1,4.5) to [out=up, in=down] (13.1,6);

\draw (7.3,9) node {\scalebox{2}{$r$}}; \draw (9.6,2.3) node {\scalebox{2}{$s$}};
\draw (1.2,5.3) node {\scalebox{2}{$t$}}; \draw (5.3,5.3) node {\scalebox{2}{$u$}};
\draw (9.3,5.3) node {\scalebox{2}{$v$}}; \draw (13.4,5.3) node {\scalebox{2}{$w$}};

\filldraw[pattern color=Orange, opacity=0.5, draw=none, pattern=north west lines] (0,-.3) circle (.5);
\draw [->, line width=1.5pt, Orange] (.4,0)--(1.4,.5);

\filldraw[pattern color=brown, opacity=0.5, draw=none, pattern=north east lines] (2.3,11) circle (.5);
\draw [->, line width=1.5pt, brown] (2.7,10.7)--(3.2,10);

\draw [Green] (7.4,11) node {\scalebox{2}{$a$}};
\draw [red] (11.6,9) node {\scalebox{2}{$m$}};
\draw [Yellow!50!black] (11,3.5) node {\scalebox{2}{$b$}}; 
\draw [blue] (8.5,3.4) node {\scalebox{2}{$n$}};
\draw [brown] (2.3,11) node {\scalebox{2}{$p$}}; 
\draw [Orange] (0,-.3) node {\scalebox{2}{$q$}};

\begin{knot}[clip width=8pt,clip radius=0pt, 
]
\strand [line width=2.5pt,blue]
(-.5,3) to [out=up, in=left] (1,6) to [out=right, in=left] (3,3) to [out=right, in=left] (5,6) 
 to [out=right, in=left] (7,3)  to [out=right, in=left] (9,6)  to [out=right, in=left] (11,3)
 to [out=right, in=left] (13,6) to [out=right, in=up] (14.5,3) to [out=down, in=right] (7.5,-.8)
 to [out=left, in=down] (-.5,3);
 \strand [line width=2.5pt,red]
(-.5,7.5) to [out=down, in=left] (1,4.5) to [out=right, in=left] (3,7.5) to [out=right, in=left] (5,4.5) 
 to [out=right, in=left] (7,7.5)  to [out=right, in=left] (9,4.5)  to [out=right, in=left] (11,7.5)
 to [out=right, in=left] (13,4.5) to [out=right, in=down] (14.5,7.5) to [out=up, in=right] (7.5,10.5)
  to [out=left, in=up] (-.5,7.5);
 \strand [line width=2.5pt,Green]
(7,6.5) to [out=right, in=right] (7,11.5) to [out=left, in=left] (7,6.5) ;
 \strand [line width=2.5pt,Yellow!70!black]
(3.2,3.8) to [out=right, in=left] (7,2) to [out=right, in=left] (10.8,4) to [out=right, in=up] (11.6,2.8)
to [out=down, in=right] (7,.4) to [out=left, in=down] (2.4,2.8) to [out=up, in=left] (3.2,3.8);
\flipcrossings{2,4, 6, 8, 9, 12, 13, 16}
\end{knot}
\end{tikzpicture}}
\]
Then, we have the following transversal intersection products $\pitchfork$ in $C_\bu(S^3,L)$:
\[
a\pitchfork m=r= \pm\del p,\quad b\pitchfork n=s=\pm\del q,\quad\text{ and }\quad m\pitchfork b=n\pitchfork a =p\pitchfork b=q\pitchfork a=0,
\]
which gives the cyclic defining system, and using $p\pitchfork q=\pm t$, the cyclic staircase product $\mc C$:
\[
\begin{matrix}
&  & \ceco 0 & \ceco 0 &  a\\
&  & \ceco \pm q &\ceco n \\
 \ceco 0 & \ceco 0 &   b\\
 \ceco \pm p  & \ceco m \\
 a  
\end{matrix}
\quad \implies\quad  \mc C=\pm 0\pitchfork n\pm m\pitchfork 0\pm p\pitchfork q\pm 0\pitchfork 0=\pm t
\]

To get the general form for a cyclic defining system we add linear combinations of cycles $a,m,b,n$ which are representatives of a basis of $H_2(S^3,L)$, and keep solving the conditions of the defining system. Since the only non-trivial intersection products on homology come from $[p\pitchfork q]=\pm[t]$, $[p\pitchfork n]=\pm[t+u]=\pm2[t]$, $[m\pitchfork q]=\pm[t+w]=\pm 2[t]$, and $[m\pitchfork n]=\pm [t+u+v+w]=\pm 4[t]$, this gives the following non-trivial cyclic Massey product:
\[
\la\la a,m,b,n \ra\ra=\{\quad (2c+1)\cdot [t]\quad |\quad c\in \Z\quad \}
\]

Note that the quadruple Massey product $\la a,m,b,n\ra$ for this link is the empty set: there is no defining system because there is no trivialization of the triple Massey product $\la m, b, n\ra=\{(4c+2)\cdot [t] \,\, |\,\, c\in \Z \}$).

\end{example}

The intersection of chains in a manifold has been formalized as a \emph{partial algebra}, and shown to be functorially equivalent to a (fully defined) $A_\infty$ algebra (over $\R$) and to an $E_\infty$ algebra (over $\Z$), \cite{SW1}.
In short, intersection yields a structure defined on the subcomplex of chains in general position, and one may replace this by a fully defined structure on an equivalent complex, using the bar construction. 

It is reasonable to expect that the entire setup of defining systems and Massey-type products can be extended to the category of partial algebras, though we do not pursue it here. On the other hand, the intersection product on homology is identified with the cup product in cohomology via the Lefschetz duality map described above. So, one expects an equivalence between the infinity structures on chains  with that of cochains, using local homotopy arguments as in \cite{TZS07}, both for the product and inner product structures. Completing this sketch, together with naturality, would establish that the previous computation with chains yields the Massey inner product sets cohomologically as well.

In our second approach, we compute using the transferred $A_\infty$ structure over $\Z_2$ (see Theorem \ref{THM:transfer-theorem}, \ref{thm;pullbacknat}) and the induced Massey inner product, again using representative chains that are in general position. It turns out that the higher $(1,1)$-part of the inner product contains the non-trivial Massey inner product information.

\begin{example}[4-component link; approach \#2]
Let $R=\Z_2$, and use the notation from the previous approach to this example. In the current approach, we will restrict our computations  to the chains and classes relevant for determining the specific Massey inner product set.

Let $B=C_\bu(S^3,L)$ with the intersection product and let $A=H_\bu(S^3,L)$ with generators $[a], [m], [b], [n]$ in degree $2$ and $[t]$ in degree $1$. We use the formulas in Setup \ref{SETUP:transfer-theorem} to get an induced $A_\infty$ algebra structure on $A$, where we take 
$f: B\to A$ to be a projection onto homology with $f(a)=[a]$, $f(r)=0$, and similarly for the other generators from the previous approach,
 we take $g:A\to B$ to be a map that chooses the given representative from the previous approach, i.e., $g([a])=a$, and similarly for the other generators, and we take
$h: B \to B$ to be a chain homotopy with $h(r)=p$, $h(s)=q$, and $h(t)=0$, which is consistent with $\del p=r$ and  $\del q=s$.

We want to find a cyclic defining system of $[a],[m],[b],[n]$. Note, that all relevant $\mu_2$ and $\mu_3$ products vanish; for example, $\mu_2([a],[m])=f\circ \nu_2(g[a],g[m])=f(a\pitchfork m)=f(r)=0$ (since $r$ is exact, and so it vanishes in homology), or 
\begin{multline*}
\mu_3([a],[m],[b])=f\circ (\nu_2(h\circ \nu_2\otimes id)+\nu_2(id\otimes h\circ \nu_2))(g[a],g[m],g[b]) \\
= f(h(a\pitchfork m)\pitchfork b+a\pitchfork h(m\pitchfork b)) = f(h(r)\pitchfork b+a\pitchfork h(0))=f(p\pitchfork b)=0.
\end{multline*}
Thus, we have a cyclic defining system for $[a],[m],[b],[n]$ given by all zeros off the diagonal.

We use the inner product $I_{t^*}:B\to B^*$ on $B$ (see Example \ref{EXA:cdga=>IP}\eqref{ITM:cdga-example}) given by the intersection product composed with evaluating by $t^*$. The cyclic staircase product $\mc I$ in $A$ is given by \eqref{EQU:staircase-product-for-11} where the $A_\infty$ inner product is the pull back inner product $K^\sharp(I_{t^*})$ for the module morphism $K$ coming from the transfer theorem \ref{THM:transfer-theorem}:
\begin{multline*}
\mc I=K^\sharp(I_{t^*})_{1,1}([a],[m],[b])([n])=I_{t^*}(k_2([a],[m]))(k_2([b],[n]))+I_{t^*}(k_2([m],[b])(k_2([n], [a]))
\\ +I_{t^*}(k_3([a],[m],[b]))(k_1([n]))+I_{t^*}(k_1([m]))(k_3([b],[n],[a]))=t^*(p \pitchfork q)+t^*(0\pitchfork 0)=t^*(t)=1,
\end{multline*}
where we used that all off diagonal terms in the cyclic defining system are zero in the first equality, and $k_3$ vanishes on the terms in the third equality. Any other cyclic staircase produces the same result over $\Z_2$, by the same reasoning as in the previous approach of this example, so the $(1,1)$ Massey inner product set is simply 
\[
\la[a],[m],[b],[n]\ra_{K^\sharp(I_{t^*})}=\{1\}.
\]

To sum up, the intersection product and inner product $I_{t^*}$ were transferred to give vanishing multiplications on the relevant inputs, but the non-zero $(1,1)$-part of the pull back inner product $K^\sharp(I_{t^*})$ gives a non-trivial Massey inner product set.
\end{example}

\appendix 

\section{A minimal $A_\infty$ model for the lens space $L(p,q)$ over $\Z_p$}\label{APP:L(p,q)-mimimal}

In order to perform computations, we express the $A_\infty$ conditions in the dual language of $A_\infty$ coalgebras.

\subsection{$A_\infty$ coalgebras}\label{SEC:A-infty-coalg}

Assume that $R$ is a (possibly finite) field. We want $A=\bigoplus_{k\geq 0}A_k$ to come from a dual $A_\infty$ coalgebra, i.e., we assume that each $A_k=C_k^*=Hom(C_k,R)$, and that $C=\bigoplus_{k\geq 0} C_k$ is an $A_\infty$ coalgebra.

This means that there are maps $\Delta_k:C\to C^{\otimes k}$ for $k=1, 2, 3, \dots$ of degree $|\Delta_k|=k-2$, which after shifting down by one become maps $\ul{\Delta_k}:\ul{C}\to \ul{C}^{\otimes k}$ of degree $|\ul{\Delta_k}|=-1$, 
satisfying
\begin{equation}\label{EQU:ULDelta_ULDelta}
 \sum_{k+\ell-1=n}\,\, \sum_{i+j+1=k} (\underbrace{id\otimes \dots \otimes id}_{i\text{ many}}  \otimes\ul{\Delta_\ell}\otimes \underbrace{id\otimes \dots \otimes id}_{j\text{ many}})\circ \ul{\Delta_k}=0, \quad \forall n\geq 1
\end{equation}
Note, that the sign of each term is always ``$+$'', indicating that in the shifted setting, all signs only come from the Koszul rule. We will sometimes simply refer to this equation as $\Delta^2=0$ (when lifting $\Delta$ as a derivation of the tensor algebra).

An $A_\infty$ coalgebra morphism $r:G\to C$ between $A_\infty$ coalgebras $(G,\dDel)$ and $(C,\Delta)$ consists of maps $r_k:G\to C^{\otimes k}$ for $k=1,2,3, \dots$ of degree $|r_k|=k-1$, such that the shifted maps $\ul{r_k}:\ul{G}\to \ul{C}^{\otimes k}$  (of degree $|\ul{r_k}|=0$) satisfy
\begin{multline}\label{EQU:r_r}
\sum_{k+\ell-1=n}\,\, \sum_{i+j+1=k}  (\underbrace{id\otimes \dots \otimes id}_{i\text{ many}}  \otimes\ul{\Delta_\ell}\otimes \underbrace{id\otimes \dots \otimes id}_{j\text{ many}})\circ \ul{r_k}
\\
=\sum_{\ell_1+\dots+\ell_k=n}   (\ul{r_{\ell_1}},\dots,\ul{r_{\ell_k}})\circ \ul{\dDel_k}=0, \quad \forall n\geq 1
\end{multline}
In short, we may state this equation as $\Delta\circ r=r\circ \dDel$ (where $r$ is lifted to the tensor algebra as an algebra map). We note, that we may determine the signs in \eqref{EQU:ULDelta_ULDelta} and \eqref{EQU:r_r} for the unshifted maps via the shift operators just as we did in Section \ref{SEC:background-A-infty}.

\begin{lemma}
Let $(C=\bigoplus_{k\geq 0} C_k,\Delta_1,\Delta_2, \dots)$ be an $A_\infty$ coalgebra such that each $C_k$ is finite dimensional. Then $(A,\mu_1,\mu_2,\dots)$ defined by setting $A_k:=C_k^*$ and $\mu_k:=\Delta_k^*$ for each $k$, is an $A_\infty$ algebra. Moreover, dualizing an $A_\infty$ coalgebra map $r:G\to C$ by setting $f_k:=r_k^*$ for each $k$ gives an $A_\infty$ algebra map $f=r^*:C^*\to G^*$.
\end{lemma}
Indeed, dualizing Equation \eqref{EQU:ULDelta_ULDelta} gives Equation  \eqref{EQU:ULmu_ULmu}, while dualizing Equation \eqref{EQU:r_r} gives \eqref{EQU:f_f}.

\subsection{Setup and main statements}
We fix integers $p>2$ and $1<q<p$. We now review a well-known presentation of the lens space $L(p,q)$ for these $p$ and $q$; see for example \cite[p.133, Sec. 2.1 Ex. 8]{H}.

\begin{setup}
We use the space $C=C_\bu(L(p,q),\Z_p)$ with the following generators:
\begin{equation*}
C= \Z_p<
\unbr{\cA,\cB}_{\text{degree }0}, 
\unbr{\ca, \cb, \cc_0, \dots, \cc_{p-1}}_{\text{degree }1} ,
\unbr{\cal_0, \dots, \cal_{p-1}, \cbe_0, \dots, \cbe_{p-1}}_{\text{degree }2},
\unbr{\cs_0, \dots, \cs_{p-1}}_{\text{degree }3} >
\end{equation*}

In the following we will write indices that always need to be interpreted mod $p$. For example, if $p=5, q=3, j=2$, then $\cc_{j-q}=\cc_{-1}=\cc_4$, etc. We picture these generators as follows:
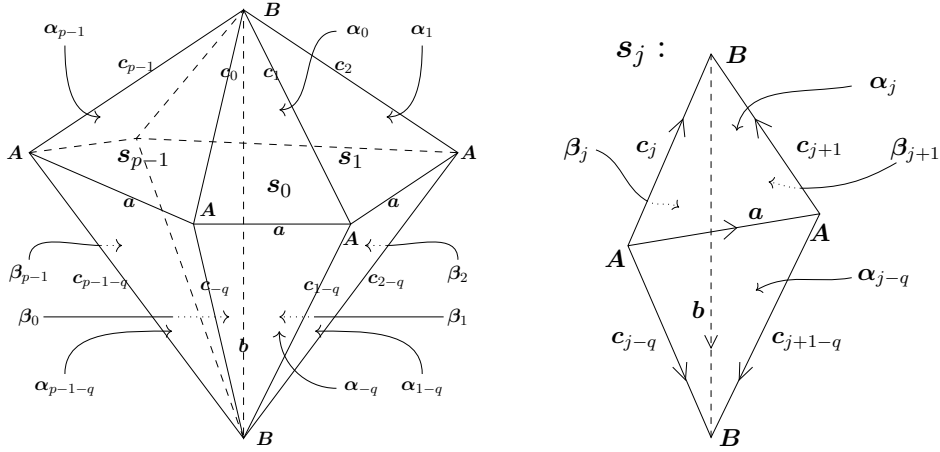
\[
\scalebox{0.95}{\begin{tikzpicture}
\draw (4,0)--(7,4)--(4,6)--(1,4)--(4,0); 
\draw (7,4)--(5.5,3)--(3.3,3)--(1,4);
\draw (5.5,3)--(4,0)--(3.3,3)--(4,6)--(5.5,3);
\draw[dashed] (1,4)--(2.5,4.2)--(7,4);
\draw[dashed] (4,0)--(2.5,4.2)--(4,6)--(4,0);

 \node at (0.8,4) {\scb{$\cA$}};  \node at (7.15,4) {\scb{$\cA$}}; 
 \node at (3.5,3.2) {\scb{$\cA$}};  \node at (5.5,2.8) {\scb{$\cA$}}; 
 \node at (4.3,0) {\scb{$\cB$}};  \node at (4.4,6) {\scb{$\cB$}}; 
 \node at (1.5,5.7) {\scb{$\cal_{p-1}$}};  \draw[->] (1.5,5.5) to[out=270,in=180] (2,4.5);
 \node at (5.6,5.7) {\scb{$\cal_0$}};       \draw[->] (5.3,5.7) to[out=180,in=0] (4.5,4.6);
 \node at (6.5,5.7) {\scb{$\cal_1$}};       \draw[->] (6.5,5.5) to[out=270,in=0] (6,4.5);
 \node at (1.5,.7) {\scb{$\cal_{p-1-q}$}};  \draw[->] (1.5,.9) to[out=90,in=180] (3,1.5);
 \node at (5.65,.7) {\scb{$\cal_{-q}$}};       \draw[->] (5.3,.7) to[out=180,in=270] (4.5,1.5);
 \node at (6.5,.7) {\scb{$\cal_{1-q}$}};       \draw[->] (6.5,.9) to[out=90,in=0] (5,1.5);
 \node at (4.5,3.5) {{$\cs_0$}};   \node at (5.5,3.9) {{$\cs_1$}};  
 \node at (2.6,3.9) {{$\cs_{p-1}$}}; 
  \node at (2.4,3.3) {\scb{$\ca$}};    \node at (4.5,2.9) {\scb{$\ca$}};
 \node at (6.1,3.3) {\scb{$\ca$}};   \node at (4,1.3) {\scb{$\cb$}}; 
  \node at (2.5,5.2) {\scb{$\cc_{p-1}$}};   \node at (3.8,5.1) {\scb{$\cc_{0}$}}; 
  \node at (4.4,5.1) {\scb{$\cc_{1}$}};   \node at (5.4,5.2) {\scb{$\cc_{2}$}}; 
  \node at (2,2.2) {\scb{$\cc_{p-1-q}$}};   \node at (3.6,2.1) {\scb{$\cc_{-q}$}}; 
  \node at (5.1,2.1) {\scb{$\cc_{1-q}$}};   \node at (6,2.2) {\scb{$\cc_{2-q}$}}; 

 \node at (1,2.3) {\scb{$\cbe_{p-1}$}};  
 \draw (1,2.5) to[out=90,in=180] (2,2.7);  \draw[dotted, ->] (2,2.7) to[out=0,in=180] (2.3,2.7);
 \node at (1,1.7) {\scb{$\cbe_{0}$}};  
 \draw (1.2,1.7) to[out=0,in=180] (3,1.7);  \draw[dotted, ->] (3,1.7) to[out=0,in=180] (3.8,1.7);
 \node at (7,1.7) {\scb{$\cbe_{1}$}};      
 \draw (6.8,1.7) to[out=180,in=0] (5,1.7);  \draw[dotted, ->] (5,1.7) to[out=180,in=0] (4.5,1.7);
 \node at (7,2.3) {\scb{$\cbe_{2}$}};       
 \draw (7,2.5) to[out=90,in=0] (6,2.7);  \draw[dotted, ->] (6,2.7) to[out=180,in=0] (5.7,2.7);
  \end{tikzpicture} }
 \hspace{.8cm}
 \scalebox{0.85}{\begin{tikzpicture}
\draw (11,3.5)--(9.3,0)--(8,3)--(9.3,6)--(11,3.5)--(8,3);
\draw[dashed] (9.3,0)--(9.3,6);
 \node at (7.8,2.8) {$\cA$};  \node at (11,3.2) {$\cA$}; 
  \node at (9.6,0) {$\cB$};  \node at (9.7,6) {$\cB$}; 
 \node at (10,3.5) {$\ca$};  \node at (9.1,2) {$\cb$}; 
 \node at (8.3,4.5) {$\cc_j$};  \node at (11,4.5) {$\cc_{j+1}$}; 
 \node at (8.1,1.5) {$\cc_{j-q}$};  \node at (10.8,1.5) {$\cc_{j+1-q}$}; 
 \node at (12,5.5) {$\cal_j$};  \node at (12,2.5) {$\cal_{j-q}$}; 
 \draw[->] (11.5,5.5) to[out=190,in=80] (9.7,4.8);
 \draw[->] (11.5,2.5) to[out=190,in=30] (10,2.3);
 \node at (7.2,4.5) {$\cbe_j$};  \node at (12.5,4.5) {$\cbe_{j+1}$}; 
 \draw[] (7.5,4.5) to[out=-20,in=150] (8.3,3.7);
  \draw[dotted, ->] (8.3,3.7) to[out=-30,in=150] (8.8,3.6);
 \draw[] (12.2,4.3) to[out=210,in=0] (10.7,3.9);
 \draw[dotted, ->] (10.7,3.9) to[out=180,in=-10] (10.2,4);
 \node at (8.2,6) {\scalebox{1.4}{$\cs_j:$}};  
 \draw (9.5,3.4)--(9.7,3.28)--(9.5,3.1);   \draw (9.2,1.6)--(9.3,1.4)--(9.4,1.6);
 \draw (8.92,1.18)--(8.9,.9)--(8.7,1.1);   \draw (9.75,1.18)--(9.74,.9)--(9.95,1.1);  
 \draw (8.89,4.7)--(8.86,4.98)--(8.64,4.78);   \draw (10.04,4.68)--(9.99,4.98)--(10.25,4.86);  
 \end{tikzpicture} }
\]
Here, the $\cal_j$ are the $2$-simplicies that are on the boundary of the $3$-ball, which identify the top with the bottom of its boundary $2$-sphere, whereas the $\cbe_j$ are the $2$-simplicies in the interior of the $3$-ball, whose interior face $\cb$ connects the north- and south-poles $\cB$.

The differential $\del=\Delta_1:C\to C$ is given by:
\begin{align*}
& \del(\cA)=0, &&\del(\cB)=0, \\
& \del(\ca)=0,&& \del(\cb)=0, \\
&&& \del(\cc_j)=\cB-\cA, \\
& \del(\cal_j)=\cc_{j+1}-\cc_j+\ca, \hspace{-1cm}&&\del(\cbe_j)= \cb-\cc_{j-q}+\cc_{j},\\
& \del(\cs_j)=\cbe_{j+1}-\cbe_j+\cal_{j-q}-\cal_j\hspace{-2cm}
\end{align*}
The usual Alexander-Whitney product gives the following coproducts $\Del_2:C\to C^{\ot 2}$:
\begin{align*}
& \Del_2(\cA)=\cA\ot\cA, &&\Del_2(\cB)=\cB\ot\cB, \\
& \Del_2(\ca)=\cA\ot\ca+\ca\ot \cA, &&\Del_2(\cb)=\cB\ot\cb+\cB\ot\cb, \\
&&& \Del_2(\cc_j)=\cA\ot\cc_j+\cc_j\ot \cB, \\
& \Del_2(\cal_j)=\cA\ot\cal_j+\ca\ot \cc_{j+1}+\cal_j\ot\cB,\hspace{-3cm} \\
& \Del_2(\cbe_j)= \cA\ot\cbe_j+\cc_j\ot \cb+\cbe_j\ot\cB,\hspace{-3cm}\\
& \Del_2(\cs_j)=\cA\ot\cs_j+\ca\ot \cbe_{j+1}+\cal_j\ot \cb+\cs_j\ot\cB \hspace{-4cm}
\end{align*}

We next combine $\del$ and $\Del_2$ and interpret them as an $A_\infty$ coalgebra structure. In order to get a more convenient sign, we first apply the co-dga isomorphism $(C,\del,\Del_2)\to (C,-\del,\Del_2)$, $c\mapsto (-1)^{|c|}\cdot c$ for all $c\in C$. Next, we shift the degrees of $-\del$ and $\Del_2$ down by one, so that we obtain the following total $A_\infty$ coalgebra structure $\unDel:\ul C\to \ul C\oplus \ul C^{\ot 2}$, where we suppress the ``$\otimes$''-symbols:\begin{align}
\label{EQU:Del-A-B}
& \unDel(\uA)=-\uA\uA, 
&&\unDel(\uB)=-\uB\uB, \\
& \unDel(\ua)=-\uA\ua+\ua \uA, 
&&\unDel(\ub)=-\uB\ub+\ub\uB, \\
&&& \unDel(\uc_j)=\uA-\uB-\uA\uc_j+\uc_j \uB, \\
& \unDel(\ual_j)=-\uc_{j+1}+\uc_j-\ua-\uA\ual_j+\ua \uc_{j+1}-\ual_j\uB,\hspace{-3cm} \\
& \unDel(\ube_j)=  -\ub+\uc_{j-q}-\uc_{j}-\uA\ube_j+\uc_j \ub-\ube_j\uB,\hspace{-3cm}\\
\label{EQU:Del-sj}
& \unDel(\us_j)=-\ube_{j+1}+\ube_j-\ual_{j-q}+\ual_j-\uA\us_j+\ua \ube_{j+1}-\ual_j \ub+\us_j\uB \hspace{-8cm}
\end{align}
\end{setup}

\[
\textit{From here on we assume that $p$ is a prime $\geq 3$.}
\]
 It is well-known that the homology of $L(p,q)$ with coefficients in $\Z_p$ is $H_\bu(L(p,q),\Z_p)=\Z_p\oplus \Z_p\oplus \Z_p\oplus \Z_p$ with one generator in degrees $0, 1, 2$, and $3$, respectively; see \cite{H}. We give this the following $A_\infty$ coalgebra structure, which we claim describes a minimal $A_\infty$ coalgebra structure of $C$.
\begin{definition}\label{DEF:L(p,q)-minimal-model}
Let 
\[
H=H_\bu(L(p,q),\Z_p)\cong \Z_p<\cE,\cX,\cY,\cZ>
\]
with degrees given by $|\cE|=0, |\cX|=1, |\cY|=2, |\cZ|=3$. Denote by $\Delta$ the $A_\infty$ coalgebra structure given by
\begin{align*}
&\unDel(\uE)=-\uE\uE,\\
&\unDel(\uX)=\uX\uE-\uE\uX, \\
&\unDel(\uY)=-\uY\uE-\uE\uY-\uX^{\otimes p}\\
&\unDel(\uZ)=\uZ\uE-\uE\uZ+\uX\uY-\uY\uX
\end{align*}
Note, that the shifted degrees are $|\uE|=-1, |\uX|=0, |\uY|=1, |\uZ|=2$. Moreover, we note that $\Delta^2=0$ in the sense of Equation \eqref{EQU:ULDelta_ULDelta}.
\end{definition}

We will define an $A_\infty$ coalgebra quasi-isomorphism $C\to H$  which is given as a composition of two $A_\infty$ coalgebra quasi-isomorphisms 
\begin{equation}\label{EQU:C-H-H}
C\stackrel r\to \dH\stackrel {s}\to H
\end{equation} in the following two Sections \ref{SEC:dH->H} and \ref{SEC:C->dH}. Here, $\dH$ is a variation of $H$ with a slightly more general version of the $A_\infty$ coalgebra structure of $H$ (see Proposition \ref{PROP:H'-to-H}). The map $s$ is defined in Section \ref{SEC:dH->H}, while $r$ is defined in Section \ref{SEC:C->dH}. We then obtain the following.
\begin{corollary}
$C$ is quasi-isomorphic as an $A_\infty$ coalgebra to $H_\bu(L(p,q),\Z_p)$ from Definition \ref{DEF:L(p,q)-minimal-model}. Dually, this means that $C^*$ is quasi-isomorphic as an $A_\infty$ algebra to $H^\bu(L(p,q),\Z_p)$ with the $A_\infty$ structure described in Proposition \ref{PROP:HL(p,q)-minimal-model}.
\end{corollary}
We remark, that dualizing the $A_\infty$ coalgebra over $\Z_p$ as we did in the above corollary (as described in Section \ref{SEC:A-infty-coalg}) gives the usual cochain model over $\Z_p$, since, for any $\Z$ module $C$, there is an isomorphism of $\Z_p$ modules $K:Hom_\Z(C,\Z_p)\to Hom_{\Z_p}(C\ot_\Z \Z_p,\Z_p)$, given by $(K(u))(c\ot z)=u(c)\cdot z$ for $c\in C, z\in \Z_p$, and whose inverse is given by $(K^{-1}(v))(c)=v(c\ot 1)$. The map $K$ is natural with respect to $\Z$ module maps $g:C\to C'$ in the sense that there is a commutative diagram
\[\begin{tikzcd}
  Hom_\Z(C,\Z_p) \arrow[r, "K"]  
& Hom_{\Z_p}(C\ot_\Z \Z_p,\Z_p)
\\
Hom_\Z(C',\Z_p) \arrow[u, "g_\Z"] \arrow[r, "K'"]
& Hom_{\Z_p}(C'\ot_\Z \Z_p,\Z_p) \arrow[u, "g_{\Z_p}", swap] 
\end{tikzcd}\]
where $(g_\Z(u'))(c)=u'(g(c))$ and $(g_{\Z_p}(v'))(c\ot z)=v'(g(c)\ot z)$.

\subsection{An $A_\infty$ coalgebra quasi-isomorphism $\dH\stackrel {s}\to H$}\label{SEC:dH->H}

\begin{proposition}\label{PROP:H'-to-H}
Let  $\dH=\Z_p<\dE,\dX,\dY,\dZ>$ with degrees $|\dE|=0, |\dX|=1, |\dY|=2, |\dZ|=3$ and $\dH$ be an $A_\infty$ coalgebra whose $A_\infty$ coalgebra structure $\dDel$ is of the form
\begin{align*}
&\undDel(\vE)=-\vE\vE,\\
&\undDel(\vX)=\vX\vE-\vE\vX, \\
&\undDel(\vY)=-\vY\vE-\vE\vY-\vX^{\otimes p}+\sum_{j\geq p+1}y_j\cdot \vX^{\ot j}\\
&\undDel(\vZ)=\vZ\vE-\vE\vZ+c(\vX\vY-\vY\vX)+\sum_{i+j\geq 2} z_{i,j}\cdot \vX^{\ot i}\vY \vX^{\ot j}
\end{align*}
Here, $c$ is some non-zero constant and $\sum_{i+j=n}z_{i,j}=0$ for each $n\geq 2$. 

Then, there is an $A_\infty$ coalgebra quasi-isomorphism $s:\dH\to H$, where $H$ is the $A_\infty$ coalgebra from Definition \ref{DEF:L(p,q)-minimal-model}
\end{proposition}
\begin{proof}
We make the Ansatz $\ul{s}(\vE)=\uE$, $\ul{s}(\vX)=\uX$, and
\begin{align*}
\ul{s}(\vY)=& \uY+\sum_{i+j\geq 1}v_{i,j} \cdot \uX^{\ot i}\uY \uX^{\ot j}
\\
\ul{s}(\vZ)=&c\cdot \uZ+\sum_{i+j\geq 1}w_{i,j} \cdot \uX^{\ot i}\uZ \uX^{\ot j}
\end{align*}
Note that such an $s$ is certainly a quasi-isomorphism since $s_1:\dH\to H$ is an isomorphism. We need to check that there are numbers $v_{i,j}$ and $w_{i,j}$ such that $\Del\circ s=s\circ \undDel$ in the sense of  Equation \eqref{EQU:r_r}. First, we check what conditions are necessary for $\unDel\circ \un{s}(\vY)=\un{s}\circ \undDel(\vY)$ to hold. We have:
\begin{multline*}
\un{s}\circ \undDel(\vY)
= \un{s}\Big(-\vY\vE-\vE\vY-\vX^{\otimes p}+\sum_{j\geq p+1}y_j\cdot \vX^{\ot j}\Big)\\
=-\Big(\uY+\sum_{i+j\geq 1}v_{i,j} \cdot \uX^{\ot i}\uY \uX^{\ot j}\Big)\uE-\uE\Big(\uY+\sum_{i+j\geq 1}v_{i,j} \cdot \uX^{\ot i}\uY \uX^{\ot j}\Big)
 -\uX^{\otimes p}+\sum_{j\geq p+1}y_j\cdot \uX^{\ot j}
\end{multline*}
whereas (using a telescopic sum argument), we get:
\begin{multline*}
\unDel\circ \un{s}(\vY)
= \unDel\Big( \uY+\sum_{i+j\geq 1}v_{i,j} \cdot \uX^{\ot i}\uY \uX^{\ot j}\Big)\\
=\Big(-\uY\uE-\uE\uY-\uX^{\otimes p}\Big)+\sum_{i+j\geq 1}v_{i,j} \cdot \Big(-\uE \uX^{\ot i}\uY \uX^{\ot j}-\uX^{\ot i}\uY \uX^{\ot j}\uE\Big)
-\sum_{i+j\geq 1} v_{i,j}\cdot \uX^{i+p+j}
\end{multline*}
Comparing the two expressions, we see that the terms are equal as long as $\sum_{j\geq p+1}y_j\cdot \uX^{\ot j}=-\sum_{i+j\geq 1} v_{i,j}\cdot \uX^{i+p+j}$, or $y_k=-\sum_{i+j=k}v_{i,j}$ for all $k>p$. Thus, in order for  $\unDel\circ \un{s}(\vY)=\un{s}\circ \undDel(\vY)$ to hold, we may choose any $v_{i,j}$ such that $y_k=-\sum_{i+j=k}v_{i,j}$ (which is clearly always possible).

Next we check the conditions for $\unDel\circ \un{s}(\vZ)=\un{s}\circ \undDel(\vZ)$ to hold. We compute:
\begin{multline*}
\unDel\circ \un{s}(\vZ)
= \unDel\Big( c\cdot \uZ+\sum_{i+j\geq 1}w_{i,j} \cdot \uX^{\ot i}\uZ \uX^{\ot j}\Big)\\
=  c\cdot \Big(\uZ\uE-\uE\uZ+\uX\uY-\uY\uX\Big)+\sum_{i+j\geq 1}w_{i,j} \cdot \Big(-\uE\uX^{\ot i}\uZ \uX^{\ot j}+\uX^{\ot i}\uZ \uX^{\ot j}\uE\Big)\\
+\sum_{i+j\geq 1}w_{i,j} \cdot \uX^{\ot i}\Big(\uX\uY-\uY\uX\Big) \uX^{\ot j}
\end{multline*}
and
\begin{multline*}
\un{s}\circ \undDel(\vZ)
= \un{s}\Big(\vZ\vE-\vE\vZ+c(\vX\vY-\vY\vX)+\sum_{i+j\geq 2} z_{i,j}\cdot \vX^{\ot i}\vY \vX^{\ot j}\Big)\\
= \Big(c\cdot \uZ+\sum_{i+j\geq 1}w_{i,j} \cdot \uX^{\ot i}\uZ \uX^{\ot j}\Big)\uE-\uE\Big(c\cdot \uZ+\sum_{i+j\geq 1}w_{i,j} \cdot \uX^{\ot i}\uZ \uX^{\ot j}\Big)\\
+c\cdot \Big[\uX\Big( \uY+\sum_{i+j\geq 1}v_{i,j} \cdot \uX^{\ot i}\uY \uX^{\ot j}\Big)-\Big( \uY+\sum_{i+j\geq 1}v_{i,j} \cdot \uX^{\ot i}\uY \uX^{\ot j}\Big)\uX\Big]\\
+\sum_{i+j\geq 2} z_{i,j}\cdot \uX^{\ot i} \Big(\uY+\sum_{k+\ell\geq 1}v_{k,\ell} \cdot \uX^{\ot k}\uY \uX^{\ot \ell}\Big)\uX^{\ot j}
\end{multline*}
Thus for the two equations to be equal, we need that
\begin{multline*}
\sum_{i+j\geq 1}w_{i,j} \cdot \uX^{\ot i}\Big(\uX\uY-\uY\uX\Big) \uX^{\ot j}
=c\cdot\sum_{i+j\geq 1} v_{i,j} \cdot \uX^{\ot i}\Big(\uX\uY-\uY\uX\Big) \uX^{\ot j}
\\+\sum_{i+j\geq 2} z_{i,j}\cdot \uX^{\ot i} \uY\uX^{\ot j}
+\sum_{i+j\geq 2} \sum_{k+\ell\geq 1} z_{i,j}\cdot v_{k,\ell}\cdot \uX^{\ot i+k}\uY\uX^{\ot j+\ell}
\end{multline*}
Notice that the terms on the right are such that for a fixed number of total tensors, the coefficients add to $0$, i.e., when looking at exactly $n$ $\uX$'s, we get a sum of coefficients $c\cdot \sum_{i+1+j=n}v_{i,j}-c\cdot \sum_{i+j+1=n}v_{i,j}=0$, and by assumption $\sum_{i+j=n}z_{i,j}=0$, and also $\sum_{i+k+j+\ell=n}z_{i,j}\cdot v_{k,\ell}=\sum_{m+m'=n} \unbr{\sum_{i+j=m} z_{i,j}}_{=0}\cdot \sum_{k+\ell =m'} v_{k,\ell}=0$. Similarly, the left hand side has coefficients that sum to $\sum_{i+1+j=n}w_{i,j}- \sum_{i+j+1=n}w_{i,j}=0$. Therefore, we can pick $w_{0,n-1}$ to match the coefficient of $\uY\uX^{\ot n}$ on the right hand side, then similarly for $w_{1,n-2}, w_{2,n-3},\dots$ all the way up to $w_{n-1,0}$ for the coefficient of $\uX^{\ot n-1}\uY\uX$. The coefficient of $\uX^{\ot n}\uY$ then matches both sides, since the coefficients add to $0$ on both sides (i.e., we have $n$ variables $w_{0,n-1},\dots, w_{n-1,0}$ with which we can solve $n$ linear equations).
\end{proof}

\subsection{An $A_\infty$ coalgebra quasi-isomorphism $C\stackrel {r}\to \dH$}\label{SEC:C->dH}

\begin{proposition}\label{PROP:C-to-H'}
There is an $A_\infty$ coalgebra structure on $\dH$ which is of the form of Proposition \ref{PROP:H'-to-H}, and an $A_\infty$ coalgebra quasi-isomorphism $r:C\to \dH$. The constant $c$ in the definition of $\dH$ stated in Proposition \ref{PROP:H'-to-H} is $c=q$.
\end{proposition}

Recall that the $A_\infty$ coalgebra $\Del$ on $C$ is described by Equations \eqref{EQU:Del-A-B}--\eqref{EQU:Del-sj}. We need to construct an $A_\infty$ coalgebra map $r$ from $C$ to $\dH$ together with the $A_\infty$ coalgebra $\dDel$ on $\dH$ which is of the following form $\dDel:\dH\to \prod_{k\geq 1}\dH^{\ot k}$ (for some to-be-specified numbers $y_k$, $c$, and $z_{k,\ell}\in \Z_p$):
\begin{align*}
&\undDel(\vE)=-\vE\vE,\\
&\undDel(\vX)=\vX\vE-\vE\vX, \\
&\undDel(\vY)=-\vY\vE-\vE\vY-\vX^{\otimes p}+\sum_{k\geq p+1}y_k\cdot \vX^{\ot k}\\
&\undDel(\vZ)=\vZ\vE-\vE\vZ+c(\vX\vY-\vY\vX)+\sum_{k+\ell\geq 2} z_{k,\ell}\cdot \vX^{\ot k}\vY \vX^{\ot \ell}
\end{align*}
The construction of these number will be given in the next four subsections.
\subsubsection{An Ansatz for $r$}\quad

$\bu$
To construct the $A_\infty$ coalgebra map $r$ from $C$ to $\dH$, we make the following Ansatz for the map $r:C\to \prod_{k\geq 1}\dH^{\ot k}$, where the numbers $a_k, b_k, c_{j,k}, \al_{j,(k,\ell)}\in \Z_p$ below will still have to be specified.
\begin{align*}
 \ur(\uA)&=\vE, \\
\ur(\uB)&=\vE, \\ \cline{1-3}
 \ur(\ua)&=\vX+\sum_{k\geq 2}a_k\cdot \vX^{\otimes k}, \\  
\ur(\ub)&=q\cdot \vX+\sum_{k\geq 2}b_k\cdot \vX^{\otimes k}, \\
 \ur(\uc_0)&=0, \\
 \ur(\uc_j)&=(p-j)\cdot \vX+\sum_{k\geq 2}c_{j,k} \cdot\vX^{\otimes k} & \text{for }j=1,\dots, p-2\\
 \ur(\uc_{p-1})&=\vX, \\ \cline{1-3}
 \ur(\ual_0)&=\vY, \\
 \ur(\ual_j)&=\al_{j,(1,0)}\cdot \vX\vY+\sum_{k+\ell\geq 2} \al_{j,(k,\ell)}\cdot \vX^{\ot k}\vY\vX^{\ot \ell}\hspace{-8mm}& \text{for } j=1,\dots, p-1\\
 \ur(\ube_0)&=0,\\
 \ur(\ube_j)&=\vY& \text{for } j=1,\dots, q\\
 \ur(\ube_j)&=0& \text{for } j=q+1,\dots, p-1\\ \cline{1-3}
 \ur(\us_0)&=\vZ \\
 \ur(\us_j)&=0 & \text{for } j=1,\dots, p-1 
\end{align*}
We will sometimes simplify notation by writing expressions such as $\ur(\ub)=\sum_{k\geq 1} b_k \vX^{\ot k}$ where the index set includes $k=1$ by implicitly assuming that $b_1=q$ (compare with the above); similarly we take $\al_{j,(0,1)}=0$, etc.

$\bu$
We need to show that we can find constants $y_k,c, z_{k,\ell},a_k, b_k,c_{j,k}, \al_{j,(k,\ell)}\in \Z_p$ such that $r\circ \Del=\dDel\circ r$ in the sense of Equation \eqref{EQU:r_r}. Clearly, we have $\ur(\unDel(\uA))=-\vE\vE=\undDel(\ur(\uA))$ and $\ur(\unDel(\uB))=-\vE\vE=\undDel(\ur(\uB))$. Moreover, if $\ue$ is any of the generators $\ua,\ub,\uc_0,\dots, \uc_{p-1}$, say we have $\ur(\ue)=\sum_{k\geq 1}e_k\vX^{\ot k}$, then, using a telescopic sum argument, we get:
\begin{align*}
\undDel(\ur(\ue))=&\undDel\big(\sum_{k\geq 1}e_k\vX^{\ot k}\big)=\sum_{k\geq 1}e_k\cdot \big(-\vE\vX^{\ot k}+\vX^{\ot k}\vE\big)
\\
=&-\vE\ur(\ue)+\ur(\ue)\vE=\ur(\unDel(\ue))
\end{align*}

\subsubsection{Solving for $\al_{j,(k,\ell)}$, $c$, and $z_{k,\ell}$}\quad

$\bu$
For $j=1,\dots, p-1$, we apply $r\circ \Del=\dDel\circ r$ to $\us_j$ which gives
\begin{align}\label{EQU:Th-r-sj=r-De-sj}
0=&\undDel(\ur(\us_j))\stackrel !=\ur(\unDel(\us_j))\\ \nonumber
=&\ur\big(-\ube_{j+1}+\ube_j-\ual_{j-q}+\ual_j-\uA\us_j+\ua \ube_{j+1}-\ual_j \ub+\us_j\uB\big)
\end{align}
Thus, we have to fulfill the following conditions for the cases $j=1,\dots, p-1$ (the case $j=0$ will be considered separately below):
\begin{align*}
\scalebox{0.7}{$j=1:$}&&0&=-\vY+\vY&&-\ur(\ual_{1-q})&&+\ur(\ual_1)&&+\ur(\ua)\vY&&-\ur(\ual_1)\ur(\ub) \\
&&&&&\vdots \\
\scalebox{0.7}{$j=q-1:$}&&0&=-\vY+\vY&&-\ur(\ual_{q-1-q})&&+\ur(\ual_{q-1})&&+\ur(\ua)\vY&&-\ur(\ual_{q-1})\ur(\ub) \\
\scalebox{0.7}{$j=q:$}&&0&= \quad\quad+\vY &&-\vY&&+\ur(\ual_q)&&&&-\ur(\ual_q)\ur(\ub) \\
\scalebox{0.7}{$j=q+1:$}&&0&=&&-\ur(\ual_{q+1-q})&&+\ur(\ual_{q+1})&&&&-\ur(\ual_{q+1})\ur(\ub) \\
&&&&&\vdots \\
\scalebox{0.7}{$j=p-1:$}&&0&=&&-\ur(\ual_{p-1-q})&&+\ur(\ual_{p-1})&&&&-\ur(\ual_{p-1})\ur(\ub) 
\end{align*}
Clearly, all of the individual terms of $\pm\vY$ cancel.

$\bu$
If we keep the variables $a_k$ in $\ur(\ua)$ and $b_k$ in $\ur(\ub)$ as unknowns, then we can solve the above equations for the $\ur(\ual_j)$ depending on the $a_k$ and $b_k$. More precisely, starting from $j=q$, we can solve $0=\ur(\ual_q)-\ur(\ual_q)\cdot \ur(\ub)$ by setting $\ur(\ual_q)=0$.

$\bu$
Next, we use $\ur(\ual_q)=0$ to solve for $\ur(\ual_{j})$ with $j\equiv 2q($mod $p)$. It is worth going through this next step in detail as all of the following arguments are of similar nature. There are two possibilities: First, if $j>q$, then we solve again $0=-\underbrace{\ur(\ual_{j-q})}_{=0}+\ur(\ual_{j})-\ur(\ual_{j})\ur(\ub)$ with $\ur(\ual_{j})=0$. Second, if $j<q$, we cannot set $\ur(\ual_{j})$ to be $0$ since we need to solve an equation with a non-zero term $\ur(\ua)\vY$:
\begin{align*}
0=&-\underbrace{\ur(\ual_{j-q})}_{=0}+\ur(\ual_{j})+\ur(\ua)\vY-\ur(\ual_{j})\ur(\ub)\\
=&\sum_{k+\ell\geq 1} \al_{j,(k,\ell)}\cdot \vX^{\ot k}\vY\vX^{\ot \ell}+\Big(\vX+a_2 \vX\vX+a_3 \vX\vX\vX+\dots\Big)\vY\\
&-\sum_{k+\ell\geq 1} \al_{j,(k,\ell)}\cdot \vX^{\ot k}\vY\vX^{\ot \ell}\cdot \sum_{k\geq 1} b_k \vX^{\ot k}
\end{align*}
In the lowest tensor powers, where we have a total of $k+\ell=1$ tensors of $\vX$, we obtain $\al_{j,(1,0)}=-1$ and $\al_{j,(0,1)}=0$, and then we may use the equation to solve for higher $\al_{j,(k,\ell)}$, which are dependent on the yet unknown numbers $a_k$ and $b_k$. Note that solving for $\al_{j,(k,\ell)}$ only depends on $a_i$ and $b_i$ for which $i\leq k+\ell$, since there are $k+\ell$ tensors of $\vX$ in the equation that determines $\al_{j,(k,\ell)}$. We write this in short as $\al_j=\al_j[a,b]$, i.e., we implicitly mean that we get a (polynomial) formula for $\al_{j,(k,\ell)}$ depending only on the numbers $a_i$ and $b_i$ with $i\leq k+\ell$.

$\bu$
Knowing $\ur(\ual_{2q})$, we proceed to solve for $\ur(\ual_{j})$ with $j\equiv 3q($mod $p)$:
\[
0=-\ur(\ual_{2q})+\ur(\ual_{3q})+\varepsilon \cdot \ur(\ua)\vY-\ur(\ual_{3q})\ur(\ub),
\]
where $\varepsilon \in \{0,1\}$ depends on $j<q$ or $j>q$. Knowing $\ur(\ual_{2q})$, we can solve for $\ur(\ual_{3q})$, where we get that $\al_{j,(1,0)}$ is equal to $\al_{2q,(1,0)}$ if $\varepsilon=0$ (no extra $\ur(\ua)\vY$ term) or is equal to $\al_{2q,(1,0)}-1$ if $\varepsilon=1$ (there is an extra $\ur(\ua)\vY$ term). The other terms solve as $\al_{j,(0,1)}=0$, and again $\al_{j}=\al_j[a,b]$ is a polynomial in the numbers $a_i$ and $b_i$ of lower or equal tensor order.

$\bu$
Continuing this way we keep solving for  $\ur(\ual_{2q}), \ur(\ual_{3q}), \dots, \ur(\ual_{(p-1)q})$, each time getting $\al_{n\cdot q}=\al_{n\cdot q}[a,b]$, and $\al_{n\cdot q,(0,1)}=0$ while $\al_{n\cdot q,(1,0)}$ subtracts one each time $j\equiv n\cdot q($mod $p)$ is less than $q$. In particular, for $j\equiv (p-1)q($mod $p)$, we have picked up all $q-1$ terms less than $q$, so that $\al_{(p-1)q,(1,0)}=-(q-1)$. We now use this in the missing condition $\ur\circ \unDel(\us_0)=\undDel\circ \ur(\us_0)$ that we have not yet used above:
\begin{align*}
\undDel\circ \ur(\us_0)=&\undDel(\vZ)=\vZ\vE-\vE\vZ+c(\vX\vY-\vY\vX)+\sum_{k+\ell\geq 2} z_{k,\ell}\cdot \vX^{\ot k}\vY \vX^{\ot \ell}\\
\ur\circ \unDel(\us_0)=&\ur\Big(-\ube_{1}+\ube_0-\ual_{-q}+\ual_0-\uA\us_0+\ua \ube_{1}-\ual_0 \ub+\us_0\uB \Big)\\
=&-\vY+0-\ur(\ual_{(p-1)q})+\vY-\vE\vZ+\ur(\ua)\vY-\vY\ur(\ub)+\vZ\vE\\
=&\vZ\vE-\vE\vZ-\Big(-(q-1)\vX\vY+\sum_{k+\ell\geq 2}\al_{(p-1)q,(k,\ell)}\vX^{\ot k}\vY\vX^{\ot \ell}\Big)\\
&+\Big(\vX+\sum_{k\geq 2}a_k\cdot \vX^{\otimes k}\Big)\vY-\vY\Big(q\cdot \vX+\sum_{k\geq 2}b_k\cdot \vX^{\otimes k}\Big)\\
=& \vZ\vE-\vE\vZ+q(\vX\vY-\vY\vX)+\sum_{k+\ell\geq 2}\Big(
\begin{matrix}
\text{\small some function of}\\
\text{\small $a_i$, $b_i$, $\al_{(p-1)\cdot q,(i,j)}$}
\end{matrix}
\Big)\cdot \vX^{\otimes k}\vY \vX^{\otimes \ell}
\end{align*}
Thus, we see that we can set $c=q$, and that we can solve for $z_{k,\ell}$. We will check later that the sum $\sum_{k+\ell=n}z_{k,\ell}$ vanishes, as required, which is a consequence of $\undDel^2(\vZ)=0$.

$\bu$
Thus, we have solved for $\al_j=\al_j[a,b]$ and $z=z[a,b]$ (depending on the coefficients $a_k$ and $b_k$), as well as $c=q$, which implied that we get $\undDel(\ur(\us_j))= \ur(\unDel(\us_j))$ for all $j=0,\dots, p-1$.

\subsubsection{Solving for $c_{k,\ell}$, $y_k$, and $a_k$}\quad

$\bu$
We now look at the conditions $\undDel(\ur(\ual_j))= \ur(\unDel(\ual_j))$. Using the unifying notation for the indicies of $\ur(\ual_j)$, written as $\ur(\ual_j)=\sum_{k+\ell\geq 0} \al_{j,(k,\ell)}\cdot \vX^{\ot k}\vY\vX^{\ot \ell}$,  as well as $y_p=-1$, we get:
\begin{align}\label{EQU:The-r-alj}
\undDel(\ur(\ual_j))=&\undDel\Big(\sum_{k+\ell\geq 0} \al_{j,(k,\ell)}\cdot \vX^{\ot k}\vY\vX^{\ot \ell}\Big)
\\ \nonumber
=& \sum_{k+\ell\geq 0} \al_{j,(k,\ell)}\cdot\Big(-\vE \vX^{\ot k}\vY\vX^{\ot \ell}- \vX^{\ot k}\vY\vX^{\ot \ell}\vE\Big)\\ \nonumber
&+\sum_{k+\ell\geq 0}\sum_{j\geq p} \al_{j,(k,\ell)}\cdot y_i\cdot \vX^{\ot k}\vX^{\ot i}\vX^{\ell}
\\ \label{EQU:r-Del-alj}
 \ur(\unDel(\ual_j))=&\ur\big(-\uc_{j+1}+\uc_j-\ua-\uA\ual_j+\ua \uc_{j+1}-\ual_j\uB\big)
 \\ \nonumber
 =& -\ur(\uc_{j+1})+\ur(\uc_j)-\ur(\ua)
 -\vE\sum_{k+\ell\geq 0} \al_{j,(k,\ell)}\cdot \vX^{\ot k}\vY\vX^{\ot \ell}
 \\ \nonumber
 & +\ur(\ua)\ur(\uc_{j+1}) -\sum_{k+\ell\geq 0} \al_{j,(k,\ell)}\cdot \vX^{\ot k}\vY\vX^{\ot \ell}\vE
 \end{align}
Thus, Equations \eqref{EQU:The-r-alj} and \eqref{EQU:r-Del-alj} agree as long as
\begin{equation}\label{EQU:Theta-R-al=r-Del-al}
\sum_{k+\ell\geq 0}\sum_{i\geq p} \al_{j,(k,\ell)}\cdot y_i\cdot \vX^{\ot k}\vX^{\ot i}\vX^{\ell} 
= -\ur(\uc_{j+1})+\ur(\uc_j)-\ur(\ua)+\ur(\ua)\ur(\uc_{j+1})
\end{equation}

$\bu$
Now, for $j\neq 0$, we have $\al_{j,(0,0)}=0$, and we use the abbreviation 
\begin{equation}\label{EQU:al-hat-y-hat}
\wh{\ual_j}:=\sum\limits_{k+\ell\geq 1} \al_{j,(k,\ell)} \vX^{\ot k+\ell} \quad\text{  and }\quad  \wuy:=\sum\limits_{i\geq p}  y_i\vX^{\ot i}
\end{equation}
we can write the condition $\undDel(\ur(\ual_j))= \ur(\unDel(\ual_j))$ as
\begin{equation}\label{EQU:c_j-induct}
\ur(\uc_j)=\wh{\ual_j}\wuy+\ur(\ua)+\ur(\uc_{j+1})-\ur(\ua)\ur(\uc_{j+1})
\end{equation}

$\bu$
For $j=p-1$, we have $\ur(\uc_{j})=\ur(\uc_{p-1})=\vX$ and $\ur(\uc_{j+1})=\ur(\uc_{p})=0$, this means that
$\vX=\wh{\ual_{p-1}}\wuy+\ur(\ua)$, thus we have $\ur(\ua)=\vX-\wh{\ual_{p-1}}\wuy$.
Plugging this back into \eqref{EQU:c_j-induct} gives
\begin{equation}\label{EQU:c_j-up-to-y}
\ur(\uc_j)=\vX +\ur(\uc_{j+1})-\vX\ur(\uc_{j+1})+\wuy\cdot \big(\wh{\ual_j}-\wh{\ual_{p-1}}+\ur(\uc_{j+1})\wh{\ual_{p-1}}\big)
\end{equation}
Proceeding now by applying this for $j=p-2$, then $j=p-3, p-4, \dots$, we can solve \eqref{EQU:c_j-up-to-y} for $\ur(\uc_{j})$. For $j=p-2$, use that $\ur(\uc_{j+1})=\ur(\uc_{p-1})=\vX$:
\begin{align*}
\scalebox{0.7}{$j=p-2:$}&&\ur(\uc_{p-2})&=2\vX-\vX\vX&&+\wuy\cdot (\wh{\ual_{p-2}}-\wh{\ual_{p-1}}+\vX\wh{\ual_{p-1}}) \\
\scalebox{0.7}{$j=p-3:$}&&\ur(\uc_{p-3})&=3\vX-3\vX\vX+\vX\vX\vX &&+\wuy\cdot\Big(\begin{matrix}\text{some polynomial in}\\ \vX,\wuy, \wh{\ual_{p-3}},\wh{\ual_{p-2}}, \wh{\ual_{p-1}}\end{matrix}\Big)\\
&& \vdots\\
\scalebox{0.7}{$j=k:$}&&\ur(\uc_{k})&=g_k(\vX) &&+\wuy\cdot h_k\big( \vX,\wuy, \wh{\ual_{k}}, \dots, \wh{\ual_{p-1}}\big)\\
&& \vdots\\
\scalebox{0.7}{$j=1:$}&&\ur(\uc_{1})&=g_1(\vX) &&+\wuy\cdot h_1\big( \vX,\wuy, \wh{\ual_{1}}, \dots, \wh{\ual_{p-1}}\big)
\end{align*}
Here, $h_k$ is some polynomial in the indicated variables, whose constant term is zero; thus $h_k\big( \vX,\wuy, \wh{\ual_{k}}, \dots, \wh{\ual_{p-1}}\big)$ starts with at least one $\vX$. Moreover $g_k$ is a polynomial in $\vX$, and an induction shows that this polynomial is given by $g_k(x)=1-(1-x)^{p-k}$.

$\bu$
Now, for $j=0$, our Ansatz $\ur(\ual_0)=\vY$ shows that $\al_{0,(0,0)}=1$, but $\al_{0,(k,\ell)}=0$ for $k+\ell\geq 1$. Moreover, we have that $\ur(\uc_0)=0$, so that Equation \eqref{EQU:Theta-R-al=r-Del-al} for $j=0$ becomes
\begin{align*}
\wuy=&-\ur(\uc_{1})-\ur(\ua)+\ur(\ua)\ur(\uc_{1})
\\
=&-g_1(\vX)-\vX+\vX g_1(\vX)\\
&+\wuy\cdot (-h_1+\wh{\ual_{p-1}}+\vX\cdot h_1- \wh{\ual_{p-1}}\cdot g_1(\vX)-\wh{\ual_{p-1}}\cdot \wuy\cdot h_1)
\end{align*}
where on the right side we simply wrote $h_1$ for $h_1( \vX,\wuy, \wh{\ual_{1}}, \dots, \wh{\ual_{p-1}})$. Now, using $g_1(x)=1-(1-x)^{p-1}$, we compute $-g_1(x)-x+xg_1(x)=-1+(1-x)^{p}\equiv -1+(1-x^p)=-x^p($mod $p)$.

\noindent
{\bf Note:} Here, we used the assumption that we are working over $\Z_p$ with $p$ prime, i.e., we used that $(1-x)^{p}\equiv 1-x^p($mod $p)$.

\noindent
Thus, the above becomes
\[
\wuy\cdot (1+\xi_1 \vX+\xi_2 \vX\vX+\dots) = -\vX^{\ot p}
\]
This can now inductively be solved for $\wuy:=\sum\limits_{i\geq p}  y_i\vX^{\ot i}$ (even though the $\xi_j$ depend on $y_i$ but only of lower indices). Note, in particular, that $y_{p}=-1$, and the higher $y_j$ depend on the number of $a_i$ and $b_i$ but only of lower indices, i.e., in our notation, $y=y[a,b]$.

$\bu$
With this, we can also solve for the coefficients in $\ur(\ua)$ using $\ur(\ua)=\vX-\wh{\ual_{p-1}}\wuy=\vX-\sum\limits_{k+\ell\geq 1}\sum\limits_{i\geq p} \al_{p-1,(k,\ell)}\cdot y_i\cdot \vX^{\ot k+i+\ell}$. Indeed, since $\al_{p-1}=\al_{p-1}[a,b]$ and $y=y[a,b]$ depends on the $a_i$ and $b_i$ of lower indices, we can solve this for $a_j$ depending again on $b_i$ of equal or lower indices.

$\bu$
Thus, we have solved $a=a[b]$, $y=y[a[b],b]=y[b]$, $c_j=c_j[a[b],b]=c_j[b]$, and from the previous subsection, we also get $\al_j=\al_j[a[b],b]=\al_j[b]$ and $z=z[a[b],b]=z[b]$ depend only on the $b_i$. With the exception of the $b_k$, we now have solved for all other variables in terms of these: $y=y[b],c=q, z=z[b],a=a[b],c_j=c_j[b], \al_j=\al_j[b]$.

\subsubsection{Solving for $b_k$, and checking the remaining conditions}\quad

$\bu$
We next look at the condition $\undDel(\ur(\ube_j))= \ur(\unDel(\ube_j))$ for $j=0$. We get $0=\undDel(\ur(\ube_0))=\ur(\unDel(\ube_0))=\ur( -\ub+\uc_{-q}-\uc_{0}-\uA\ube_0+\uc_0 \ub-\ube_0\uB)=-\ur( \ub)+\ur( \uc_{-q})$, so in order to satisfy $\undDel(\ur(\ube_0))= \ur(\unDel(\ube_0))$ we need to have that $\ur( \ub)=\ur( \uc_{-q})=\ur( \uc_{p-q})$. If $q=1$, this is obviously solved due to our Ansatz of $\ur( \uc_{p-1})=\vX$. If $q>1$, we recall how $\ur( \uc_{p-q})$ was defined above (which was dependent on $\ur(\ub)$):
\[
\ur(\uc_{p-q})=g_{p-q}(\vX) +\wuy\cdot h_{p-q}\big( \vX,\wuy, \wh{\ual_{p-q}}, \dots, \wh{\ual_{p-1}}\big)
\]
Here, $g_{p-q}$ is the polynomial $g_{p-q}(x)=1-(1-x)^{p-q}$, and $\wuy=-\vX^{\ot p}+\sum_{k\geq p+1} y_k \vX^{\ot k}$, and $\wh{\ual_j}=\sum_{k+\ell\geq 1} \al_{j,(k,\ell)} \vX^{\ot k+\ell}$ where the coefficients $y=y[b]$ and $\al_j=\al_j[b]$ depend on the $b_i$ (of lower indices). Setting the left hand-side equal to $\ur(\ub)$ immediately gives the first $p$ coefficients $b_1,\dots, b_p$ from $\ur(\ub)=g_{p-q}(\vX)+\sum_{k\geq p+1}b_k \vX^{\ot k}$ and we can keep solving for $b_{p+1}, b_{p+2}, \dots$ since the right hand side only depends on $b_i$ of lower indices. We thus can determine $\ur(\ub)$ by solving $\undDel(\ur(\ube_0))= \ur(\unDel(\ube_0))$.

In particular, we note that the coefficients $b_k=c_{p-q,k}$ for all $k$, and thus also $b_1=c_{p-q,1}=p-(p-q)= q$ as we had assumed in our Ansatz. Using the numbers $b_k$ for this solution of $\ur( \ub)$ we now have completely determined all of the variables  $y_k,c, z_{k,\ell},a_k, b_k,c_{j,k}, \al_{j,(k,\ell)}$.

$\bu$
To fully check that $\undDel\circ\ur= \ur\circ\unDel$, we still need to check the conditions $\undDel(\ur(\ube_j))= \ur(\unDel(\ube_j))$ for $j=1,\dots, p-1$, which follows from the above. We first recall the condition for $\undDel(\ur(\ual_j))= \ur(\unDel(\ual_j))$ for any $j=0,\dots, p-1$ from from \eqref{EQU:Theta-R-al=r-Del-al}. Using the slightly more general notation than \eqref{EQU:al-hat-y-hat}
\begin{equation}\label{EQU:Th-r-al=r-De-al}
\wh{\ual_j}:=\sum\limits_{k+\ell\geq 0} \al_{j,(k,\ell)} \vX^{\ot k+\ell} \quad\text{  and }\quad  \wuy:=\sum\limits_{i\geq p}  y_i\vX^{\ot i}
\end{equation}
(which includes the case $j=0$, i.e., from our Ansatz $\ur(\ual_0)=\vY$, we use $\al_{0,(0,0)}=1$ and $\al_{0,(k,\ell)}=0$ for $k+\ell>0$), the condition $\undDel(\ur(\ual_j))= \ur(\unDel(\ual_j))$ from \eqref{EQU:Theta-R-al=r-Del-al} (which we have already estabished) becomes:
\begin{equation}\label{EQU:condition-for-al_j}
\wh{\ual_j}\wuy+\ur(\ua)-\ur(\uc_j)+\big(1-\ur(\ua)\big)\cdot \ur(\uc_{j+1})=0, \quad \forall j=0,\dots, p-1
\end{equation}

We repeat the calculation from \eqref{EQU:The-r-alj} and \eqref{EQU:r-Del-alj} but for $\ube_j$ instead of $\ual_j$, where we now use $\unDel(\ube_j)=  -\ub+\uc_{j-q}-\uc_{j}-\uA\ube_j+\uc_j \ub-\ube_j\uB$ and the abbreviation
\begin{equation*}
\wh{\ube_j}:=\sum\limits_{k+\ell\geq 0} \be_{j,(k,\ell)} \vX^{\ot k+\ell} 
\end{equation*}
(which, by our Ansatz, is just $\be_{j,(0,0)}=1$ for $1\leq j\leq q$, and $\be_{j,(k,\ell)}=0$ in all other cases). Then, a similar calculation to \eqref{EQU:The-r-alj} and \eqref{EQU:r-Del-alj} shows that $\undDel(\ur(\ube_j))=\ur(\unDel(\ube_j))$ is equivalent to 
\begin{equation}\label{EQU:Th-r-be=r-De-be}
\wh{\ube_j}\wuy+\ur(\ub)-\ur(\uc_{j-q})+ \ur(\uc_{j})\cdot \big(1-\ur(\ub)\big)=0, \quad \forall j=0,\dots, p-1
\end{equation}

We will prove \eqref{EQU:Th-r-be=r-De-be} by using $\undDel(\ur(\us_j))=\ur(\unDel(\us_j))$ for $j=1,\dots, p-1$ and applying \eqref{EQU:Th-r-al=r-De-al}. Applying $\dDel$ to \eqref{EQU:Th-r-sj=r-De-sj}, we get
\begin{align*}
0=& \dDel\Big(\ur\big(-\ube_{j+1}+\ube_j-\ual_{j-q}+\ual_j-\uA\us_j+\ua \ube_{j+1}-\ual_j \ub+\us_j\uB\big)\Big)
\\
=&-\vE\ur\big(-\ube_{j+1}+\ube_j-\ual_{j-q}+\ual_j+\ua \ube_{j+1}-\ual_j \ub\big)
\\
&-\ur\big(-\ube_{j+1}+\ube_j-\ual_{j-q}+\ual_j+\ua \ube_{j+1}-\ual_j \ub\big)\vE
\\
&-\wh{\ube_{j+1}}\wuy+\wh{\ube_j}\wuy-\wh{\ual_{j-q}}\wuy+\wh{\ual_j}\wuy+\ur(\ua) \wh{\ube_{j+1}}\wuy-\wh{\ual_j}\wuy \ur(\ub)
\\
=&-\wh{\ube_{j+1}}\wuy+\wh{\ube_j}\wuy-\wh{\ual_{j-q}}\wuy+\wh{\ual_j}\wuy+\ur(\ua) \wh{\ube_{j+1}}\wuy-\wh{\ual_j}\wuy \ur(\ub)
\end{align*}
Here, we used that $\ur(\us_j)=0$, and $\dDel(\ur(\ual_j))=-\vE\ur(\ual_j)-\ur(\ual_j)\vE+\wh{\ual_j}\wuy$, and similarly $\dDel(\ur(\ube_j))=-\vE\ur(\ube_j)-\ur(\ube_j)\vE+\wh{\ube_j}\wuy$, while $\dDel(\ur(\ua))=-\vE\ur(\ua)+\ur(\ua)\vE$ and $\dDel(\ur(\ub))=-\vE\ur(\ub)+\ur(\ub)\vE$, and the terms involving $\vE$ vanish according to \eqref{EQU:Th-r-sj=r-De-sj}. We rewrite the above equation as
\begin{align*}
&\wh{\ube_{j}}\wuy-\wh{\ube_{j+1}}\wuy\cdot \big(1-\ur(\ua)\big)
=
\wh{\ual_{j-q}}\wuy-\wh{\ual_j}\wuy \cdot \big(1-\ur(\ub)\big)
\\
&\stackrel{\eqref{EQU:condition-for-al_j}}=
-\Big(\ur(\ua)-\ur(\uc_{j-q})+\big(1-\ur(\ua)\big)\cdot \ur(\uc_{j-q+1})\Big)\\
&\hspace{1cm} +\Big(\ur(\ua)-\ur(\uc_j)+\big(1-\ur(\ua)\big)\cdot \ur(\uc_{j+1})\Big)\cdot \big(1-\ur(\ub)\big)\\
&\hspace{3mm}=\ur(\uc_{j-q})-\big(1-\ur(\ua)\big)\cdot \ur(\uc_{j-q+1})-\ur(\uc_j)\cdot \big(1-\ur(\ub)\big)\\
&\hspace{1cm}+\ur(\uc_{j+1})\cdot \big(1-\ur(\ua)\big)\cdot \big(1-\ur(\ub)\big)-\ur(\ua)\cdot \ur(\ub)
\end{align*}
which is equivalent to
\begin{multline}\label{EQU:bej-vs-be(j+1)}
\wh{\ube_j}\wuy+\ur(\ub)-\ur(\uc_{j-q})+ \ur(\uc_{j})\cdot \big(1-\ur(\ub)\big)
\\=
\Big(\wh{\ube_{j+1}}\wuy+\ur(\ub)-\ur(\uc_{j+1-q})+ \ur(\uc_{j+1})\cdot \big(1-\ur(\ub)\big)\Big)\cdot \big(1-\ur(\ua)\big)
\end{multline}
Now, according to \eqref{EQU:Th-r-be=r-De-be}, the left hand side of \eqref{EQU:bej-vs-be(j+1)} vanishes iff $\undDel(\ur(\ube_j))=\ur(\unDel(\ube_j))$, while the right hand side of \eqref{EQU:bej-vs-be(j+1)} vanishes when $\undDel(\ur(\ube_{j+1}))=\ur(\unDel(\ube_{j+1}))$ (in fact it vanishes iff this is true, since $(1-\ur(\ua))$ is invertible).

Now, when we determined $\ur(\ub)$ above, we already established $\undDel(\ur(\ube_0))= \ur(\unDel(\ube_0))$. Thus, \eqref{EQU:bej-vs-be(j+1)} for $j=p-1$ shows that we also have $\undDel(\ur(\ube_{p-1}))= \ur(\unDel(\ube_{p-1}))$. Then, using \eqref{EQU:bej-vs-be(j+1)} for $j=p-2,p-3,\dots,1$ shows that we also have all of the remaining conditions $\undDel(\ur(\ube_j))=\ur(\unDel(\ube_j))$ for $j=p-2,p-3, \dots,1$.

Thus we have showed that $\dDel\circ r=r\circ \Del$ by checking it on all generators.

$\bu$ Since the $r_1$ component of $r$ is surjective, the fact that $\dDel\circ r=r\circ \Del$ then implies that $\dDel$ is indeed an $A_\infty$ coalgebra structure: $\undDel^2(\vE)=\undDel^2(\ur(\uA))=\ur(\unDel^2(\uA))=0$, and $\undDel^2(\vX)=\ur(\unDel^2(\uc_{p-1}))=0$, $\undDel^2(\vY)=\ur(\unDel^2(\ual_{0}))=0$, and $\undDel^2(\vZ)=\ur(\unDel^2(\us_{0}))=0$.

$\bu$
The last missing assertion that we need to check is that for the $z_{k,\ell}$ appearing in $\undDel(\vZ)$, we have the property that $\sum_{k+\ell=n}z_{k+\ell}=0$ for $n\geq 2$, which follows from $\undDel^2(\vZ)=0$, as we show now. Using the notation $y_p=-1$ and $z_{1,0}=-z_{0,1}=c$, we compute that
\begin{multline*}
0=\undDel^2(\vZ)=\undDel\Big(\vZ\vE-\vE\vZ+\sum_{k+\ell\geq 1} z_{k,\ell}\cdot \vX^{\ot k}\vY \vX^{\ot \ell}\Big)\\
=\sum_{i+j\geq 1}z_{i,j} \vX^{\ot i}\Big(\sum_{k\geq p}y_k \vX^{\ot k}\Big)\vX^{\ot j}
=\sum_{k\geq p} \,\,\sum_{i+j\geq 1}z_{i,j}\cdot y_k\cdot \vX^{\ot i+j+k}
\end{multline*}
For a given fixed $n\geq p+1$, this means that the coefficients $\sum_{i+j+k=n} z_{i,j}\cdot y_k$ must be zero, which we can expand as
\begin{equation}\label{EQU:yp-zij}
0=\Big(y_p\cdot \sum_{i+j=n-p} z_{i,j}\Big)+\Big(y_{p+1}\cdot \sum_{i+j=n-p-1} z_{i,j}\Big)+\dots +\Big(y_{n-1}\cdot \sum_{i+j=1} z_{i,j}\Big)
\end{equation}
We already know that $\sum_{i+j=1} z_{i,j}=c-c=0$, and so for $n=p+2$, we get
\[
0=\Big(y_p\cdot \sum_{i+j=2} z_{i,j}\Big)+\Big(y_{p+1}\cdot \sum_{i+j=1} z_{i,j}\Big)=-1\cdot  \sum_{i+j=2} z_{i,j}
\]
Thus, $\sum_{i+j=2} z_{i,j}=0$. Assuming by induction that $\sum_{i+j=m} z_{i,j}=0$ for $m=1,\dots,n-p-1$, Equation \eqref{EQU:yp-zij} and $y_p=-1$ implies that $\sum_{i+j=n-p} z_{i,j}=0$ as well, which completes the inductive step that the sum of the $z_{i,j}$ is zero.

\bibliographystyle{alpha}
\bibliography{biblio}

\end{document}